\documentclass[12pt]{amsart}       
\usepackage{txfonts}
\usepackage{amssymb}
\usepackage{eucal}
\usepackage{graphicx}
\usepackage{amsmath}
\usepackage{mathrsfs}
\usepackage{amscd}
\usepackage[all]{xy}           
\usepackage{amsfonts,latexsym}
\usepackage{xspace}
\usepackage{epsfig}
\usepackage{float}
\usepackage{color}
\usepackage{fancybox}
\usepackage{colordvi}
\usepackage{multicol}
\usepackage{colordvi}
\usepackage{tikz}
\usepackage{soul}
\usepackage{wasysym}
\usepackage[active]{srcltx} 
\ifpdf
  \usepackage[colorlinks,final,backref=page,hyperindex]{hyperref}
\else
  \usepackage[colorlinks,final,backref=page,hyperindex,hypertex]{hyperref}
\fi
\usepackage{enumerate}
\usepackage{makecell}

\newcommand{\nc}{\newcommand}
\newcommand{\delete}[1]{}

\nc{\mlabel}[1]{\label{#1}}  
\nc{\mcite}[1]{\cite{#1}}  
\nc{\mref}[1]{\ref{#1}}  
\nc{\meqref}[1]{\eqref{#1}}
\nc{\mbibitem}[1]{\bibitem{#1}} 

\delete{
\nc{\mlabel}[1]{\label{#1}  
{\hfill \hspace{1cm}{\small\tt{{\ }\hfill(#1)}}}}
\nc{\mcite}[1]{\cite{#1}{\small{\tt{{\ }(#1)}}}}  
\nc{\mref}[1]{\ref{#1}{{\tt{{\ }(#1)}}}}  
\nc{\meqref}[1]{\eqref{#1}{{\tt{{\ }(#1)}}}}
\nc{\mbibitem}[1]{\bibitem[\bf #1]{#1}} 
}

\newtheorem{theorem}{Theorem}[section]
\newtheorem{proposition}[theorem]{Proposition}
\newtheorem{lemma}[theorem]{Lemma}
\newtheorem{corollary}[theorem]{Corollary}
\theoremstyle{definition}
\newtheorem{definition}[theorem]{Definition}
\newtheorem{prop-def}{Proposition-Definition}[section]

\newtheorem{remark}[theorem]{Remark}

\newtheorem{tempex}[theorem]{Example}
\newtheorem{tempexs}[theorem]{Examples}
\newtheorem{temprmk}[theorem]{Remark}
\newtheorem{tempexer}{Exercise}[section]
\newenvironment{example}{\begin{tempex}\rm}{\end{tempex}}

\nc{\vsa}{\vspace{-.1cm}} \nc{\vsb}{\vspace{-.2cm}}
\nc{\vsc}{\vspace{-.3cm}} \nc{\vsd}{\vspace{-.4cm}}
\nc{\vse}{\vspace{-.5cm}}

\nc{\tred}[1]{\textcolor{red}{#1}} \nc{\tgreen}[1]{\textcolor{green}{#1}}
\nc{\tblue}[1]{\textcolor{blue}{#1}} \nc{\tpurple}[1]{\textcolor{purple}{#1}}

\nc{\li}[1]{\tred{#1}}
\nc{\lir}[1]{\tred{\underline{Li:} #1}}
\nc{\xing}[1]{\tblue{\underline{Xing:}#1 }}
\nc{\hu}[1]{\tpurple{\underline{Huhu:}#1 }}
\nc{\Hu}[1]{\tpurple{#1 }}

\nc{\mscr}[1]{\mathscr{#1}} \nc{\cal}[1]{\mathcal{#1}} \nc{\bb}[1]{\mathbb{#1}}
\nc{\id}{\rm{id}} \nc{\bfk}{{\bf k}}
\nc{\rba}{\mathscr{RBA}}\nc{\spp}{\mathscr{P}}\nc{\sqq}{\mathscr{Q}} \nc{\stt}{\mscr{T}} \nc{\spb}{\mscr{P}_\circ}  \nc{\sph}{\mscr{P}_\bullet}
\nc{\dera}{\mathscr{D}er\mathscr{A}}
\nc{\add}{\uplus}\nc{\badd}{\biguplus}
\nc{\lp}{\lambda_p}\nc{\la}{\lambda_a}
\nc{\prl}[2]{(x#1y)#2z}\nc{\prr}[2]{x#1(y#2z)}
\nc{\ra}[3]{#1(x)#2#3(y)}\nc{\rb}[3]{#3(#1(x)#2y)+#1(x#2#3(y))}
\nc{\dereq}[6]
{\treey{\cdlr[0.8]{o} \cdl{ol}\cdr{or}#1{o/b}
\node at (0.2,-0.2) {$d$};\cdlr{o}\node  at (0,0) {$#2$};}
-\treey{\cdlr[0.8]{o} \cdl{ol}\cdr{or}
\node at (ol) {$#4$};\node at (-0.2,0.5) {$d$};\cdlr{o}\node  at (0,0) {$#3$};}
-\treey{\cdlr[0.8]{o} \cdl{ol}\cdr{or}
\node at (or) {$#6$};\node at (0.2,0.5) {$d$};\cdlr{o}\node  at (0,0) {$#5$};}}
\nc{\kc}[4]{\beta_{#1,#2,#3}^{b,#4}}\nc{\kb}[3]{\beta_{#1,#2}^{#3,n}}\nc{\ka}[3]{\gamma_{#1,#2}^{#3,n}}\nc{\kaa}[3]{\kappa_{#1,#2}^{#3,m}}
\nc{\otimesh}{\mathop{\otimes}_{\text{H}}}%
\nc{\kong}{\noindent}
\nc{\rblineq}[5]{(
\treey{\cdlr[0.8]{o}\cdl{ol}\cdr{or}
\node at (ol) {$#1$};\node at (or) {$#3$};\node at (-0.2,0.5) {$#2$};\node at (0.2,0.5) {$#4$};\node at (0,0) {$#5$};}
-\treey{\cdlr[0.8]{o} \cdl{ol}\cdr{or}\node at (0,0) {$#5$};\node at (0,-0.2) {$#3$};
\node at (ol) {$#1$};\node at (-0.2,0.5) {${#2}$};\node at (0.2,-0.2) {$#4$};}
-\treey{\cdlr[0.8]{o} \cdl{ol}\cdr{or}\node at (0,0) {$#5$};\node at (0,-0.2) {$#1$};
\node at (or) {$#3$};\node at (0.2,-0.2) {$#2$};\node at (0.2,0.5) {$#4$};})}

\usetikzlibrary{calc}
\newlength\xch\newlength\dbj
\newif\ifqdd
\nc\cddf[3]{%
\coordinate (#2) at ($(#1)+(#3)$);
\draw (#1)--(#2);
\ifqdd\fill (#1) circle (\dbj);\fi}
\nc\cdx[4][1]{\cddf{#2}{#3}{#4:#1*\xch}}
\nc\cdu[2][1]{\cdx[#1]{#2}{#2u}{90}}
\nc\cdl[2][1]{\cdx[#1]{#2}{#2l}{135}}
\nc\cdr[2][1]{\cdx[#1]{#2}{#2r}{45}}
\nc\cdlr[2][1]{%
\foreach \i in {#2} {\cdl[#1]{\i}\cdr[#1]{\i}}}
\nc\cduu[2][1]{%
\foreach \i in {#2} {\cdu[#1]{\i}\cdu[#1]{\i}}}
\nc\cda[2][1]{\cdx[#1]{#2}{#2a}{90}}
\nc\cdb[2][1]{\cdx[#1]{#2}{#2b}{-90}}
\nc\cdbl[2][1]{\cdx[#1]{#2}{#2bl}{-135}}
\nc\cdbr[2][1]{\cdx[#1]{#2}{#2br}{-45}}
\nc\cdblr[2][1]{%
\foreach \i in {#2} {\cdbl[#1]{\i}\cdbr[#1]{\i}}}
\nc\treeo[2][]{\tikz[baseline=-0.58ex,line width=0.06ex,
every node/.style={font=\scriptsize,inner sep=1pt},#1]{%
\coordinate (o) at (0,0);#2}}%
\nc\treeyy[2][scale=0.8]{\treeo[#1]{\cdb o\cda o#2}}%
\nc\treey[2][scale=0.8]{\treeo[#1]{\cdb o\cdlr o#2}}%
\nc\treedy[2][scale=0.8]{\treeo[#1]{\cda o\cdblr o#2}}%
\nc\zhongdian[2]{ \node at($(#1)!0.5!(#2)$) {$\bullet$};}%
\nc\zhd[1]{\foreach \i/\j in {#1} {\zhongdian{\i}{\i\j}}}
\nc\zhongddian[2]{ \node at($(#1)!0.5!(#2)$) {$\circ$};}%
\nc\zhdd[1]{\foreach \i/\j in {#1} {\zhongddian{\i}{\i\j}}}
\nc\zhongsdian[2]{ \node at($(#1)!0.5!(#2)$) [scale=0.5]{$\bullet$};}%
\nc\zhds[1]{\foreach \i/\j in {#1} {\zhongsdian{\i}{\i\j}}}
\nc{\rc}{R_{A,\,}}
\nc{\cubicr}[4]{\kc{k}{l}{i}{1}\treey{\cdlr[0.8]{o}\cdl{ol}\cdr{or}
\node at (ol) {$#1$};\node at (or) {$#2$};\node at (-0.2,0.5) {$P_k$};\node at (0.2,0.5) {$P_l$};\node at (0,0.2) {$i$};\node at (0,0) {$#3$};}
+\kc{k}{l}{i}{2}\treey{\cdlr[0.8]{o} \cdl{ol}\cdr{or}#4{o/b}
\node at (ol) {$#2$};\node at (-0.2,0.5) {$P_l$};\node at (0.2,-0.2) {$P_k$};\node at (0,0.2) {$i$};\node at (0,0) {$#3$};}
+\kc{k}{l}{i}{3}\treey{\cdlr[0.8]{o} \cdl{ol}\cdr{or}#4{o/b}
\node at (or) {$#2$};\node at (0.2,-0.2) {$P_k$};\node at (0.2,0.5) {$P_l$};\node at (0,0.2) {$i$};\node at (0,0) {$#3$};}
+\kc{k}{l}{i}{4}\treey{\cdlr[0.8]{o} \cdl{ol}\cdr{or}\node at (0,-0.35) {$#2$};
\node at (0,-0.15) {$#1$};\node at (0.2,-0.1) {$P_k$};\node at (0.2,-0.4) {$P_l$};\node at (0,0.2) {$i$};\node at (0,0) {$#3$};}}
\nc{\cubicrr}[4]{\kc{k}{l}{i}{1}\treey{\cdlr[0.8]{o}\cdl{ol}\cdr{or}
\node at (ol) {$#1$};\node at (or) {$#2$};\node at (-0.2,0.5) {$P_k$};\node at (0.2,0.5) {$P_l$};\node at (0,0.2) {$i$};\node at (0,0) {$#3$};}
+\kc{k}{l}{i}{2}\treey{\cdlr[0.8]{o} \cdl{ol}\cdr{or}#4{o/b}
\node at (ol) {$#2$};\node at (-0.2,0.5) {$P_l$};\node at (0.2,-0.2) {$P_k$};\node at (0,0.2) {$i$};\node at (0,0) {$#3$};}\\
+&\kc{k}{l}{i}{3}\treey{\cdlr[0.8]{o} \cdl{ol}\cdr{or}#4{o/b}
\node at (or) {$#2$};\node at (0.2,-0.2) {$P_k$};\node at (0.2,0.5) {$P_l$};\node at (0,0.2) {$i$};\node at (0,0) {$#3$};}
+\kc{k}{l}{i}{4}\treey{\cdlr[0.8]{o} \cdl{ol}\cdr{or}\node at (0,-0.35) {$#2$};
\node at (0,-0.15) {$#1$};\node at (0.2,-0.1) {$P_k$};\node at (0.2,-0.4) {$P_l$};\node at (0,0.2) {$i$};\node at (0,0) {$#3$};}}
\nc{\qc}{quadratic/cubic }
\nc{\lin}[1]{{#1}^{\mathrm{LC}}} \nc{\Lin}[1]{\mathrm{Lin}_{#1}}
\nc{\mat}[1]{{#1}^{\mathrm{MT}^{\pmc}}} \nc{\Mat}[1]{\mathrm{Mat}_{#1}}
\nc{\tot}[1]{{#1}^{\mathrm{TC}}} \nc{\Tot}[1]{\mathrm{Tot}_{#1}}
\nc{\lrr}{{R_{{\mathrm{LC}}}}}\nc{\mrr}{R_{{\mathrm{ MT}^{\pmc}}}}
\nc{\lmr}{R_\mathrm{ LMT}}\nc{\trr}{R_{\mathrm{TC}}}\nc{\ttr}{R_{\mathrm{T}}}
\nc{\lrrs}[1]{(#1)_{\mathrm{LC}}}\nc{\mrrs}[1]{(#1)_{\mathrm{MT}}}\nc{\trrs}[1]{(#1)_{\mathrm{TC}}}
\nc{\lr}[1]{#1_{L}} \nc{\mr}[1]{#1_{M}} \nc{\tr}[1]{#1_{T}}
\nc{\tra}[1]{#1_{T,1}} \nc{\trb}[1]{#1_{T,2}} \nc{\trc}[1]{#1_{T,3}}
\nc{\tru}[1]{#1_{T'}}
\nc{\barr}[1]{\bar{#1}}
\nc\dual{\ast}\nc\dkc[4]{\beta'^{b,#4}_{#1,#2,#3}}
\nc{\linop}{\mscr{T}\big(\bigoplus\limits_{\omega\in\Omega} E_\omega\big)/\langle \bigcup\limits_{\omega\in\Omega}R_\omega\cup \lrr\rangle}
\nc\des{{\rm Des_{\leq_{L}}}}\nc\ver{{\rm Vert}}\nc\dess{{\rm Des_{\leq}}}\nc\sopr[1]{{#1'}}
\nc\dec{d}\nc\inp{{\rm in }}
\nc\tpol{treelike monomial\xspace} \nc\tpols{treelike monomials\xspace}
\nc\pordq{\leq_{\rm t}}\nc\pord{>_{\rm t}}\nc\prord{<_{\rm t}}
\nc\bfc{\mathbf{c}} \nc\as{\mscr{A}\hspace{-0.1cm}s} \nc\ass{\mscr{A}\hspace{-0.1cm}ss}
 \nc\dend{\mscr{D}end} \nc\matc{{\rm mt}} \nc\Matc{{\rm Mt}}\nc\sh{{\rm S}}\nc\lMatc{{\rm LMt}}\nc\lsh{{\rm LS}}
\nc\bigdot{\tikz \fill[black] (0.8ex, 0.8ex) circle (0.8ex);}
\nc\bigcir{\tikz \draw (0.8ex, 0.8ex) circle (0.8ex);}
\nc{\lie}{{\mathscr{L}ie}}\nc\lier{{R_{\rm Lie}}}\nc\lieg{{M_{\rm Lie}}}
\nc{\com}{{\mathscr{C}om}}\nc\hap{\underset{\rm H}{\otimes}}\nc\ckd[1]{#1^{\text{!`}}}
\nc\lev{leveled\xspace}\nc\lmat[1]{{#1}^{\mathrm{ LMT}}}\nc{\lmrr}{R_{\mathrm{LMT}}}
\nc\faml{\rm family }\nc\famls{\rm families } \nc\sfaml{\rm subfamily }

\nc\pamn{S(\Omega)_n}\nc\pamnc{C(\Omega)_n}
\nc\bfone{{\bf 1}}
\nc\qpl[2]{{(#1,\dec)}_{#2}}
\nc\fcw{(c_\omega)}  \nc\cw{c_\omega}\nc\fc{c}
\nc\fcwpp[1]{(n_{#1})} \nc\cwpp[1]{n_{#1}}
\nc\fcwp[2]{(#1_{#2})} \nc\cwp[2]{#1_{#2}}
\nc\rdenda[6]{(x\prec_{#1}y)\prec_{#2}z=x\prec_{#3}(y\prec_{#4} z)+x\prec_{#5}(y\succ_{#6} z)}
\nc\rdendb[4]{(x\succ_{#1}y)\prec_{#2}z=x\succ_{#3}(y\prec_{#4} z)}
\nc \rdendc[6]{(x\prec_{#1}y)\prec_{#2}z+(x\succ_{#3}y)\prec_{#4}z=x\succ_{#5}(y\succ_{#6} z)}
\nc\rdiaax[6]{(x\dashv_{#1}y)\dashv_{#2}z=x\dashv_{#3}(y\dashv_{#4} z)}
\nc\rdiaay[6]{(x\dashv_{#1}y)\dashv_{#2}z=x\dashv_{#5}(y\vdash_{#6} z)}
\nc\rdiab[4]{(x\vdash_{#1}y)\dashv_{#2}z=x\vdash_{#3}(y\dashv_{#4} z)}
\nc \rdiacx[6]{(x\dashv_{#1}y)\dashv_{#2}z=x\vdash_{#5}(y\vdash_{#6} z)}
\nc \rdiacy[6]{(x\vdash_{#3}y)\dashv_{#4}z=x\vdash_{#5}(y\vdash_{#6} z)}
\nc\rdendae[6]{(x\prec y)\prec z_{(#1,#2)}=x\prec (y\prec z)_{(#3,#4)}+x\prec(y\succ z)_{(#5,#6)}}
\nc\rdendbe[4]{(x\succ y)\prec z_{(#1,#2)}=x\succ (y\prec  z)_{(#3,#4)}}
\nc \rdendce[6]{(x\prec y)\prec z_{(#1,#2)}+(x\succ y)\prec z_{(#3,#4)}=x\succ(y\succ z)_{(#5,#6)}}
\nc\rdiaaxe[6]{(x\dashv y)\dashv z_{(#1,#2)}=x\dashv (y\dashv  z)_{(#3,#4)}}
\nc\rdiaaye[6]{(x\dashv y)\dashv z_{(#1,#2)}=x\dashv (y\vdash  z)_{(#5,#6)}}
\nc\rdiabe[4]{(x\vdash y)\dashv z_{(#1,#2)}=x\vdash (y\dashv  z)_{(#3,#4)}}
\nc \rdiacxe[6]{(x\dashv y)\dashv z_{(#1,#2)}=x\vdash (y\vdash   z)_{(#5,#6)}}
\nc \rdiacye[6]{(x\vdash y)\dashv z_{(#3,#4)}=x\vdash (y\vdash  z)_{(#5,#6)}}
\nc\polar{quasipolarization\xspace}\nc\polars{quasipolarizations\xspace}
\nc\ZZ{\mathbb{Z}}
\nc\pn{c}\nc\lif[1]{ {\rm Col}(#1)}
\nc\spli{foliation\xspace}\nc\splis{foliations\xspace}
\nc{\lift}{coloring\xspace}\nc{\lifts}{colorings\xspace}
\nc{\type}{type\xspace}\nc{\types}{types\xspace}
\nc\owc{$\Omega$-weak composition\xspace}\nc\owcs{$\Omega$-weak compositions\xspace}
\nc\mwc[1]{{\rm C}(\Omega,#1)}
\nc\tw{twisted by\xspace}
\nc\spl[1]{{
\prc{#1}{\pco{\fc}{#1}}}}
\nc\prc[2]{#1^{#2}}

\nc\sigmab{{\boldsymbol{\sigma}}}
\nc\taub{{\boldsymbol{\tau}}}
\nc\mpr{matching compatibility\xspace}\nc\mprs{matching compatibilities\xspace}
\nc\ufo{symmetrized foliation\xspace}\nc\ufos{symmetrized foliations\xspace}
\nc\deltal{D}
\nc\matpr[2]{#1^{#2}_t}
\nc\ltree[2]{{#1}_{#2}}
\nc\mbin{\tbinom{m}{\fc }}
\nc\cor{\phi}\nc\tcor[2]{#1#2}
\nc\pmc{{{\sigmab}}}
\nc\pmd{{\taub}}
\nc\pco[2]{\pmc^{#1,#2}}
\nc\ui[2]{e^{#1,#2}}
\nc\pcos[2]{\Lambda_{#1,#2}}
\nc\splm[1]{ {{#1}^{\pmc^{#1}}}}
\nc\pr[1]{{\Lambda_{#1}}}
\nc\pre[1]{{\pmc^{#1}}}
\nc\pc[1]{{\Lambda}_{#1}}

\nc\prs[2]{{#1}^{#2}}
\nc{\sbin}{\mathbf{b}}
\nc\ma[1]{_{{\rm MT}{}^{{#1}}}}

\nc{\QQ}{\mathbb{Q}}
\nc{\malpha}{\lambda}

\begin{document}

\title[Compatible structures of operads]{Compatible structures of operads by polarization, their Koszul duality and Manin products}

\author{Xing Gao}
\address{School of Mathematics and Statistics, Lanzhou University,
Lanzhou, 730000, China; Gansu Provincial Research Center for Basic Disciplines of Mathematics and Statistics, Lanzhou, 730000, China;
School of Mathematics and Statistics
Qinghai Nationalities University, Xining, 810007, China}
\email{gaoxing@lzu.edu.cn}

\author{Li Guo}
\address{
Department of Mathematics and Computer Science,
Rutgers University,
Newark, NJ 07102, United States}
\email{liguo@rutgers.edu}

\author{Huhu Zhang}
\address{School of Mathematics and Statistics,
	Lanzhou University, Lanzhou, Gansu 730000, China}
\email{zhanghh20@lzu.edu.cn}

\date{\today}
\begin{abstract}
Algebraic structures with replicate operations interrelated by various compatibility conditions have long been studied in mathematics and mathematical physics. They are broadly referred as linearly compatible, matching, and totally compatible structures. This paper gives a unified approach to these structures in the context of operads. Generalizing polarizations for polynomials in invariant theory to operads leads to linearly compatible operads. Partitioning polarizations into foliations gives matching operads which further yields total compatible operads under an invariance condition.
For unary/binary quadratic operads, linear compatibility and total compatibility are in Koszul dual, and the matching compatibilities are Koszul self-dual among themselves. For binary quadratic operads, these three compatible operads can be achieved by taking Manin products. For some finitely generated binary quadratic operad, Koszulity is preserved under taking the compatibilities.
\end{abstract}

\makeatletter
\@namedef{subjclassname@2020}{\textup{2020} Mathematics Subject Classification}
\makeatother
\subjclass[2020]{
18M70,  
14L24, 
05C05,   
37K10, 
17B37, 
35R60 
}
\keywords{Operad, linearly compatible algebra; polarization; matching compatibility; total compatibility; Koszul duality; Manin product; Koszul operad}

\maketitle

\vspace{-1cm}

\tableofcontents

\setcounter{section}{0}

\allowdisplaybreaks

\section{Introduction}
\mlabel{sec:intro}
This paper extends the classical notion of polarization of polynomials in invariant theory to operads in order to give a uniform formulation of algebraic structures with replicated copies of operations satisfying various compatible conditions among the copies. The relations of compatible structures with Manin products and Koszul duality are established and, as an application, a large class of Koszul operads is obtained.

\vsb

\subsection{Compatibilities of algebraic structures with replicated operations}

Traditionally, a compatible structure is referring to a {\bf linearly compatible} structure, equipping a vector space with two identical copies of operations in a given algebraic structure, so that the sum of the two copies of operations still gives the same algebraic structure.
Together with several other algebraic compatible structures, they have been widely studied in mathematics and mathematical physics.

Linear compatibility is already behind the linear or infinitesimal deformations, when a linear
perturbation $\mu+a\nu$ of an operation $\mu$ is still the same kind of operation.
The deformation of algebraic structures began with the seminal work of Gerstenhaber~\cite{Ge} for associative algebras, and has since been developed for many algebraic structures and for algebraic operads~\cite{LV}.

As an independent structure, linear compatibility first appeared in the pioneering work~\cite{Mag} of Magri on bi-Hamiltonian systems, in which a Poisson algebra has two linearly compatible Poisson (Lie) brackets. Such a structure was abstracted to the notion of a bi-Hamiltonian algebra and was studied in the context of operads and Koszul duality~\cite{BDK,DK}.
Linearly compatible Lie algebras have been studied in connection with integrable systems, the classical Yang-Baxter equation, loop algebras and elliptic theta functions~\mcite{GS1,GS2,GS3,OS}.
The Koszul property of the linearly compatible Lie operad was verified in~\cite{DK2} applying posets of weighted partitions.

In the associative context, the linearly compatible structures were studied in~\cite{OS1} for matrix algebras and especially for linear deformations. Further a linearly compatible algebra gives rise to a hierarchy of integrable systems of ODEs via the Lenard-Magri scheme~\cite{Mag,Sok}.

In~\cite{CGM}, quantum bi-Hamiltonian systems were built on linearly compatible associative algebras.
In~\cite{Dot}, free associative algebras with two linearly compatible multiplications were studied as $S_n$-modules, were constructed in terms of rooted trees and grafting, and were further related to the Hopf algebras of Connes-Kreimer, Grossman-Larson and Loday-Ronco.
Homotopy linearly compatible algebras were introduced and the homotopy transfer theorem was proved in~\cite{Zh}.

Linear compatibility of algebraic structures was generalized to binary quadratic operads in~\mcite{St08} where it was also shown that the linear compatibility for binary quadratic operads is in Koszul duality to another naturally defined type of compatibility conditions called the {\bf total compatibility}.
Totally compatible associative algebras and Lie algebras with two multiplications were further studied in~\cite{ZBG2} in connection with tridendriform algebras and post-Lie algebras.

A third compatible condition, called the {\bf matching compatibility}, has shown its importance recently. It first appeared for the associative operad as $As^{(2)}$~\cite{Zi} and was further studied in~\cite{ZBG}. The significance of such structures was revealed by the appearances of matching pre-Lie algebra (first called multiple pre-Lie algebra) emerged in the remarkable work of Bruned, Hairer and Zambotti~\cite{BHZ,Foi} on algebraic renormalization of regularity structures, and matching Rota-Baxter algebras in the algebraic study of Volterra integral equations~\mcite{GGL}. Algebraic properties of matching Rota-Baxter algebras, dendriform algebras and pre-Lie algebras were studied in~\mcite{GGZy,ZGG20}.

Motivated by the path Hopf algebra described by A.B. Goncharov~\mcite{Go}, the notion of bi-matching dialgebras was introduced in~\mcite{Fl}.
Linear compatibilities of nontraditional algebraic structures have also be studied, such as for the $L_\infty$-algebras and the Nijenhuis structure introduced by Merkulov from differential geometry~\mcite{Da22,Me,St09}.
\vsb

\subsection{Operadic studies of compatible structures}
All these recent developments of algebraic structures with compatible operations call for a uniform approach to the various constructions. After the operadic approach in~\mcite{St08} on the linear compatibility and total compatibility for binary quadratic operads, a systematic study was presented in~\mcite{ZGG23} on all the three types of compatible structures, for nonsymmetric algebraic operads with both unary and binary operations that satisfy cubic as well as quadratic relations. Also the Koszul duality was established between the linear compatibility and total compatibility, and a Koszul self-duality was obtained for the matching compatibility.

The present study generalizes the previous treatments in two directions. First we remove the restrictions on algebraic operads, including the binary quadratic conditions in~\mcite{St08} and the nonsymmetric condition in~\mcite{ZGG23}. Furthermore, a general notion of matching compatibilities is introduced to unify the various known structures of this kind and to be used as the building blocks for other compatibility structures, such as\,\mcite{BruD}. We next elaborate on these two aspects.
\vsc

\subsubsection{Polarization and linear compatibility}

Considering the conceptual and notational complexities already presented in defining linear compatible structures for special operads in~\mcite{St08,ZGG23}, new notions and frameworks are needed to treat the general case.

For this purpose, we extend the notion of polarization for polynomials to operads. Polarization of polynomials is a basic notion in classical invariant theory~\mcite{LMP,Pro,We}. Given a homogeneous polynomial, polarization produces a unique symmetric multilinear form from which the original polynomial can be recovered by evaluating along a certain diagonal call the restitution. In addition to being essential in invariant theory, e.g., Weyl's polarization theorem\,\mcite{DM,We}, the notion has applications in other areas of mathematics, including algebraic geometry, Lie algebras and representation theory. See the Appendix for a brief summary.

As it turns out, linearly compatible structures for general operads are precisely encoded by extending polarization and the more general quasipolarizations for algebras to operads. In this sense, compatible structures could be part of an invariant theory for operads that is to be developed.
In meeting the present goal, quasipolarizations for operads allows a uniform definition of the linear compatibility of an arbitrary locally homogeneous operad. In particular, the approach applies to the compatible Nijenhuis structures~\mcite{St09} and the compatible $L_\infty$-algebras~\mcite{Da22}.

\vsc
\subsubsection{Foliations and matching compatibilities}
In the past, various matching compatibilities have been defined in an ad hoc way depending on the particular algebraic structure. A general approach was taken in~\mcite{ZGG23} for nonsymmetric operads. But it does not cover the matching dendriform algebras which naturally arose from Volterra integral operators~\mcite{ZGG20}. See Remark~\mref{rk:mdend}.

By introducing a partition technique of a quasipolarization, called the foliation which also has an interesting interpretation for polynomial polarization and quadratic algebras (see the Appendix), we give a more general notion of matching compatibility applicable to all locally homogeneous operads. The notion covers the existing instances including the ones in~\mcite{ZBG} and \mcite{ZGG20} from Volterra integral operators mentioned above, as well as the special case of nonsymmetric operads in~\mcite{ZGG23}.

It is worth noting that, by combining matching compatibilities from various foliations as building blocks, even more compatibility structures can be obtained. One of them is the multi-Novikov algebra newly emerged from singular stochastic PDEs\,\mcite{BruD}. See Example~\mref{ex:match}~\mref{it:matchnov}.
\vsc

\subsection{Outline of the paper}
The paper is organized as follows.

In Section~\mref{sec:comp}, we first recall the needed background and notions on linear operads in \S\,\mref{ssec:opd}.
We then introduce in the notion of quasipolarizations of an operad (Definition~\mref{def:lrr}) by extending the classical notion of quasipolarizations of polynomials.
This notion is applied to define the linear compatibility of an arbitrary operad (Definition~\mref{de:comp1}) and show that it satisfies the linear compatibility condition (Theorem~\mref{thm:comp}). We then introduce the notion of a \spli of a quasipolarization (Definition~\mref{defn:partm}) as a partition of the quasipolarization, and apply it to define a matching compatibility of the operad (Definition~\mref{de:mat}). We finally define the total compatibility (Definition~\mref{de:totcomp1}) by adding permutation invariant relations to a matching compatibility.

Focusing on quadratic operads with unary or binary operations in Section~\mref{sec:km}, the Koszul duality between the linear compatibility and the total compatibility is established (Theorem~\mref{thm:dul}). A Koszul self-duality is also obtained for matching compatibilities (Theorem~\mref{thm:mdul}). For a binary quadratic operad, each of the compatibilities can be obtained by taking one of the Manin products (Propositions~\mref{prop:maninbl}  and~\ref{prop:maninbll}).

In Section~\mref{sec:koszul}, we show that the Koszulity of some special binary quadratic operads is preserved by taking the linear, leveled matching and total compatibilities (Theorems~\ref{thm:wak}, \ref{thm:lmtk}). As an application, a class of Koszul operads is obtained (Corollary~\ref{coro:ko}).

The Appendix recalls the notion quasipolarizations for polynomials and then introduces the refined notion of their foliations. It serves as a prototype to help understanding the corresponding notions for operads though it might be interesting on its own right. The reader might benefit from consulting it first though it is not required for the main body of the paper.

\smallskip

{\bf Notations.}
Throughout this paper, we fix a field $\bfk$ of characteristic zero
which will serve as the base field of all vector spaces, algebras,  tensor products, as well as linear maps.
The cardinality of a finite set $S$ is denoted by $|S|$. For a positive integer $n$, denote $[n]:=\{1,2,\ldots,n\}$.
For a function $f(X)$ with variables $x_i\in X$, let $f(X)|_{x_i \mapsto a}$ denote the evaluation of $f$  replacing each variable $x_i$ by $a$. Let $X\sqcup Y$ denote the disjoint union of sets $X$ and $Y$.
\vsc

\section{Compatibilities of locally homogeneous operads}
\mlabel{sec:comp}
\vsb
In this section, we apply the operadic polarization to give the general notions of  the linear compatibility, matching compatibility and total compatibility of a locally homogeneous operad.
\vsb

\subsection{Linear operads}
\mlabel{ssec:opd}
Following~\cite{BreD,LV}, we fix the notations on symmetric operads.

An {\bf $\mathbb{S}$-module} over $\bfk$ is a family
$$M :=\{M(n)\}_{n\geq0}=\{M(0),M(1),\ldots M(n),\ldots \}$$
of right modules $M{(n)}$ over the symmetric group $S_n, n\geq0$.
A {\bf (symmetric) operad } is an $\mathbb{S}$-module $\spp=\{\spp(n)\}_{n\geq0}$ equipped with an element $\id\in\mscr{P}(1)$ and {\bf composition maps}
\begin{align*}
\gamma:=\gamma_{n_1,\ldots,n_k}^k:\mscr{P}(k)\otimes\mscr{P}(n_1)\otimes\cdots\otimes\mscr{P}(n_k)\longrightarrow\mscr{P}({n_1+n_2+\cdots+n_k}),
\end{align*}
satisfying the equivariance with respect to the symmetric groups, associativity and unitarity.
A {\bf morphism of operads} from $\spp$ to $\sqq$ is a family of $S_n$-equivariant maps $f_n: \spp(n)\to \sqq(n)$ that is compatible with the compositions of operads.
An operad $\mathscr{P}$ is said to be {\bf reduced} if $\mathscr{P}(0) = 0$.

{\em In this paper, all operads is assumed to be reduced.}

For a vector space $V$, the endomorphism operad ${\rm End}_V$ is given by
${\rm End}_V(n):={\rm Hom}(V^{\otimes n},V),$
where the action of $S_n$ on ${\rm End}_V(n)$ is given by
$$(f\cdot\sigma)(v_1,\ldots,v_n)=f\big(\sigma\cdot(v_1,\ldots,v_n)\big):=f(v_{\sigma^{-1}(1)},\ldots,v_{\sigma^{-1}(n)}). $$
The composition map $\gamma$ is given by the usual composition of multivariate functions.

Let $\spp$ be an operad. A {\bf $\spp$-algebra} is a vector space $V$ with a morphism $\rho: \spp\to {\rm End}_V$ of operads. We say that the operad $\spp$ {\bf encodes} the $\spp$-algebra.

The free symmetric operad $\mscr{T}(M)$ on an $\mathbb{S}$-module $M$ can be described as follows. For a rooted tree $t$, denote by $V(t)$ the set of its internal vertices (not the leaves and root), by
$L(t)$ the set of its leaves, and by $\inp(v)$ the set of inputs of the vertex $v\in V(t)$. Let $T^m_n$ be the set of planar rooted trees with $m$ internal vertices and $n$ leaves labeled by $[n]$.
\begin{definition}
Let $M$ be an $\mathbb{S}$-module with a basis $X=\{X(1), \ldots, X(n),\ldots\}$.
An {\bf $X$-decorated tree} is a pair $(t,\dec)$ with $t$ in $T^m_n$ and $\dec:V(t)\to X$,
such that $\dec(v)$ is in $X(|\inp(v)|)$. By convention, the empty tree is an $X$-decorated rooted tree.
\end{definition}

By multi-linearity, the notion $(t,\dec)$ is also defined for a map $\dec:V(t)\to M$.

Let $M$ be an $\mathbb{S}$-module  with a basis $X$.
Denote
\begin{eqnarray*}
T^m_n(X)&=&\left\{(t, \dec)\,|\,t\in T^m_n, \dec:V(t)\to X\right\},\\
\stt(X){(1)}&=&\bigsqcup_{m=0}^{\infty} T_1^m(X),\,\stt(X){(n)}=\bigsqcup_{m=1}^{\infty} T_n^m(X)\, \text{ for }\, n\geq2,\\
\stt(M)(n)&=&\bfk \stt(X)(n)\, \text{ for }\, n\geq1.
\end{eqnarray*}

The symmetric group action on $\stt(M)(n)$ is given by
\begin{equation}
\treey{\cdlr[2]{o}\cdu[1.45]{o}\node[scale=0.8] at (-0.3,0.5) {$\cdots$};\node[scale=0.8] at (0.3,0.5) {$\cdots$};
 \node[scale=0.8][right=2] at (o) {$\mu\cdot\sigma$};\node[scale=0.8][above=2] at (ol) {$1$};\node[scale=0.8][above=2] at (ou) {$i$};\node[scale=0.8][above=2] at (or) {$n$};}
 =
\treey{\cdlr[2]{o}\cdu[1.45]{o}\node[scale=0.8] at (-0.3,0.5) {$\cdots$};\node[scale=0.8] at (0.3,0.5) {$\cdots$};
 \node[scale=0.8][right=2] at (o) {$\mu$};\node[scale=0.8][above=2] at (ol) {$\sigma^{-1}(1)\qquad$};\node[scale=0.8][above=2] at (ou) {$\sigma^{-1}(i)$};\node[scale=0.8][above=2] at (or) {$\qquad\sigma^{-1}(n)$};}, \quad  \sigma\in S_n, \mu\in\mscr{T}(M)(n).
 \mlabel{eq:acts}
\end{equation}
Then there is the following construction of free operads~\mcite{BreD}.
\begin{theorem}
The $\mathbb{S}$-module $\stt(M)$, together with the composition product $\gamma$ given by the grafting of rooted trees and the natural embedding $i:M\rightarrow \mscr{T}(M)$, is the free symmetric operad on an $\mathbb{S}$-module $M$.
\end{theorem}
\vsc
Any operad $\spp$ can be presented as the quotient of a free operad modulo an operadic
ideal:
$$\mscr{P} =\mscr{P}(M, R) := \mscr{T}(M)/\langle R\rangle,\vsc$$
where the $\mathbb{S}$-modules $M=\{0,M(1), \ldots, M(n), \ldots\}$ and $R=\{0,R(1), \ldots, R(n), \ldots\}$ are called the {\bf generators} and {\bf relations} of $\mscr{P}$, respectively.
The operad $\spp$ is called {\bf finitely generated} if $M(n)$ is finite dimensional for each $n \geq 1$.

For $(t,\dec)\in\mscr{T}(X)$, its {\bf weight} $|(t,\dec)|$ is defined to be the number of internal vertices of $t$.
Denote by $\mscr{T}(X)^{(m)}$ (resp. $\mscr{T}(X)^{(m)}_n$) the subset of $\mscr{T}(X)$ (resp. $\mscr{T}(X)_n$) of weight $m$. Then
$$\mscr{T}(M)^{(m)}:=\bfk \mscr{T}(X)^{(m)}=\Big\{0,\mscr{T}(M)^{(m)}(1),\ldots,\mscr{T}(M)^{(m)}(n),\ldots \Big\}\,\text{ where }\, \mscr{T}(M)^{(m)}(n):=\bfk T_n^m(X).\vsc$$
In particular, $\mscr{T}(M)^{(0)}=\bfk  T^0_1(X)$ and $\mscr{T}(M)^{(1)}=\Big\{0, \bfk T_1^1(X),\ldots,\bfk T_n^1(X),\ldots\Big\}=M.$

In this paper, we focus on locally homogeneous operads described as follows.
\begin{definition}\mlabel{de:ubo}
Let $\mscr{P} = \mscr{T}(M)/\langle R\rangle $ be an operad.
A relation in $R$ is called {\bf homogeneous} (of weight $m$) if it is in
$\mscr{T}(M)^{(m)}$ for some $m\geq 0$. If $R$ is in $\mscr{T}(M)^{(m)}$, then the operad $\spp$ is called {\bf homogeneous}. In particular, if $R$ is in $\mscr{T}(M)^{(2)}$, then $\spp$ is called {\bf quadratic}. If the operadic ideal $\langle R \rangle$ is generated by homogeneous relations, then the operad $\spp$ is called {\bf locally homogeneous}.
\end{definition}
A homogeneous  relation $r\in \mscr{T}(M)^{(m)}(n)=\bfk T_n^m(X)$ can be uniquely expressed as
$$r=\sum_{  (t,\dec)\in T_n^m(X)}\alpha_{t,\dec}\,(t,d).$$
In this case, denote $|r|:=m$, called the {\bf weight } of $r$.
\subsection{Polarizations for operads and linear compatibility}
\mlabel{ssec:polar}
\subsubsection{Polarizations for operads}
We next introduce the notion of quasipolarizations for operads.
In analog to the polarizations for a polynomial (see \S\,\mref{ss:polar}), the polarizations for a tree polynomial can be defined as the formal coefficients of a formal expansion.

Let $\Omega$ be a nonempty set and $X=\{X(1), \ldots, X(n),\ldots\}$ a family of sets.
For $\omega\in \Omega$, let $$X_\omega(n):=\{x_\omega\,|\,x\in X(n), \omega\in\Omega\}$$
be a replica of $X(n)$, labeled by $\omega$ for distinction.
Define
$$X_\omega:=\{X_\omega(1), \ldots, X_\omega(n),\ldots\}  \text{ and } X_\Omega:= \bigsqcup_{\omega\in\Omega} X_\omega=X\times \Omega.
\vsc$$
\begin{definition}
Let $M$ be an $\mathbb{S}$-module with a basis $X$.
\begin{enumerate}
\item An {\bf \owc} of $m\geq 1$ is a tuple $\fcw=(c_\omega)_{\omega\in \Omega}$ representing a map $\fc:\Omega\to \mathbb{Z}_{\ge 0}$ with $\sum_{\omega\in\Omega} \cw=m$.
Let $\mwc{m}$ be the set of \owcs of $m$.
\item Consider a decorated tree polynomial of homogeneous weight $m$
\begin{equation}
 f(X):=\sum_{  (t,\dec)\in T_n^m(X)}\alpha_{t,\dec} \,(t,\dec)\in \mscr{T}(M)^{(m)}(n).
 \mlabel{eq:opfun}
\end{equation}
Regarding $(\malpha_\omega)_{\omega\in\Omega}$ as formal variables, we have the expansion
\begin{equation}
f(X)\big|_{x\mapsto\sum_{\omega\in\Omega} \malpha_\omega x_\omega}=\sum_{\fc=(c_\omega)\in\mwc{m} } f_{\fc }\big(X_\Omega\big)\prod_{\omega\in \Omega} \malpha_\omega^{c_\omega}.
\mlabel{eq:polarhom}
\end{equation}
The tree polynomial $f_\fc(X_\Omega)$ is called the  {\bf \polar of $f(X)$ of type $\fc\in\mwc{m}$}.
A \polar is called a {\bf full polarization} if $\cw=0$ or $1$, for each $\omega\in\Omega$.
 \end{enumerate}
\mlabel{def:lrr}
\end{definition}
\vsb
\begin{remark}
If an operad is concentrated in arity 1 with commuting operators, then it is a commutative algebra. In this case, we recover the concept of a polynomial (full or quasi) polarization for $|\Omega|=m$ recalled in \S\,\mref{ss:polar}.
\mlabel{rem:nm}
\end{remark}
\vsb
To facilitate later computations, we next give an explicit form of quasipolarizations which requires some preliminary notations. We denote the projections
\begin{equation}
	\begin{split}
\pi:=\pi_X:&X_\Omega\to X, \quad x_\omega\mapsto x, \\
\pi_\Omega: & X_\Omega\to \Omega,\quad x_\omega \mapsto \omega, \quad x_\omega \in X_\Omega.
\end{split}
\mlabel{eq:proj}
\end{equation}

\begin{definition}
Let $\Omega$ be a nonempty set. Let $M$ be an $\mathbb{S}$-module with a basis $X$.
\begin{enumerate}
\item For a decorated tree $(t,\dec)\in T^{m}_n(X)$, an ($\Omega$-){\bf \lift} of $\dec$ is a coloring of $\dec$ on its vertices by elements of $\Omega$. More precisely, the \lift is a map $\deltal: V(t)\to X_\Omega$ such that $\deltal\circ \pi=\dec$:
\vsc
$$\xymatrix{
                &         X_\Omega \ar[d]^{\pi}     \\
  V(t) \ar@{-->}[ur]^{\deltal} \ar[r]_{\dec} & X ~
}
\vsb
$$
Denote by $\lif{\dec}$ the set of \lifts of $\dec$.

\item Let $(t,\dec)\in T^{m}_n(X)$ be a decorated tree and $\deltal$ a \lift of $\dec$. Define a map
\begin{equation} \mlabel{eq:type}
\fc: \Omega\to \mathbb{Z}_{\ge 0}, \quad \omega\mapsto c_\omega:=|(\pi_\Omega\circ \deltal)^{-1}(\omega)|,
\end{equation}
which will be denoted as a tuple $\fc=(\fc_\omega)$. With this notation, we say that the \lift $\deltal$ is of {\bf \type} $\fc=(\fc_\omega)$. Note that $\fc$ is in $\mwc{m}$.

\item For the map defined by Eq.~\eqref{eq:type}:
\begin{equation*} \mlabel{eq:lift}
T: \lif{\dec} \to \mwc{m}, \quad \deltal \mapsto \fc=(\fc_\omega),
\end{equation*}
let $\lif{\dec,\fc}$ be the inverse image of $\fc\in \mwc{m}$, i.e., the set of \lifts of $\dec$ of type $\fc$.
\end{enumerate}
\mlabel{defn:lift}
\end{definition}

\begin{remark}
Note that
	$|\lif{\dec, c}|=\tbinom{m}{\fc }:=\frac{m!}{\prod_{\omega\in\Omega} \cw!}$ for an \owc $c$.
\mlabel{remark:lif}
\end{remark}

We have the following constructions for operadic quasipolarizations.
\begin{proposition}\mlabel{prop:polar}
Let $M$ be an $\mathbb{S}$-module with a basis $X=\big\{X(1), \ldots, X(n),\ldots\big\}$ and $\Omega$ be a finite nonempty set.
Let $f(X)\in \mscr{T}(M)^{(m)}(n)$ be given in Eq.~\meqref{eq:opfun}.
\begin{enumerate}
\item  For an \owc $\fc\in\mwc{m}$, applying the formal partial derivatives, we have
$$f_{\fc }\big(X_\Omega \big)=\bigg(\frac{\partial^m}{\prod\limits_{\omega\in\Omega}\partial^{\cw}\malpha_\omega}\left(f(X)\big|_{x\mapsto\sum_{\omega\in\Omega}\malpha_\omega x_\omega}\right)\bigg)\bigg|_{\malpha_\omega=0},\,\text{ where }\, \malpha_\omega\in\bfk.$$
\mlabel{item:constp}
\vsc
\item For an \owc $\fc\in\mwc{m}$, we have
\begin{equation}
	f_{\fc}\big(X_\Omega \big)
	=\sum_{(t,\dec)\in T_n^m(X)} \alpha_{t,\dec} \sum_{\deltal\in \lif{\dec,c}} (t,\deltal).
\mlabel{eq:qpc}
\end{equation}
In particular, the quasipolarization of a tree monomial $(t,d)$ is given by
$$\qpl{t}{\fc}=\sum_{\deltal\in \lif{\dec,c}} (t,\deltal)=\sum_{|(\pi_\Omega\circ \deltal)^{-1}(\omega)|=c_\omega} (t,\deltal).$$
\mlabel{item:constpa}
\vsc
\item For an \owc $\fc\in\mwc{m}$, we have the {\bf restitution} of $f:$
\vsb
\begin{equation}
f_{\fc }\big(X_\Omega\big)\big|_{x_\omega\mapsto x}=\tbinom{m}{\fc} f(X).
\mlabel{eq:ores}
\vsb
\end{equation}
In particular, for a full polarization $f_{\fc }\big(X_\Omega \big)$, we have
$f_{\fc }\big(X_\Omega\big)\big|_{x_\omega\mapsto x}=m! f(X).$
\mlabel{item:replace}
\item A full polarization $f_{\fc }\big(X_\Omega\big)$ is multilinear in each variable,
and is symmetric in a variable set $X_\omega$ for which $c_\omega=1$.
\mlabel{item:fullp}
\end{enumerate}
\end{proposition}
\vsc
\begin{proof}
To represent a decorated tree with a specified decoration, we fix an ordering $(v_1,\ldots,v_m)$ of the vertices of the tree.
Then the decorated tree $(t,\dec)$ can be uniquely written as
\begin{equation*}
(t,\dec)=t(\dec(v_1),\ldots,\dec(v_m)),\text{ for }\, v_1,\ldots,v_m\in V(t).
\end{equation*}
\mref{item:constp} By Eq.\meqref{eq:polarhom}, we obtain
\vsa
{\small
$$
\frac{\partial^m}{\prod\limits_{\omega\in\Omega}\partial^{\cw}\malpha_\omega}\left(f(X)\big|_{x\mapsto\sum_{\omega\in\Omega}\malpha_\omega x_\omega}\right)\bigg|_{\malpha_\omega=0}
=\frac{\partial^m}{\prod\limits_{\omega\in\Omega}\partial^{\cw}\malpha_\omega}
\bigg( \sum_{\fc=(c_\omega)\in\mwc{m} } f_{\fc }\big(X_\Omega\big)\prod_{\omega\in \Omega} \malpha_\omega^{c_\omega}\bigg)\bigg|_{\malpha_\omega=0}
=
f_{\fc }\big(X_\Omega \big).
$$}
\vsc
\smallskip
\noindent
\mref{item:constpa}
By Eq.~\meqref{eq:opfun} and multi-linearity of decorated tree polynomials, we have
{\small
\begin{eqnarray*}
f(X)\big|_{x\mapsto\sum_{\omega\in\Omega}\malpha_\omega x_\omega}
&=&\sum_{  (t,\dec)\in T_n^m(X)}\alpha_{t,\dec} \,(t,\dec)\big|_{x\mapsto\sum_{\omega\in\Omega}\malpha_\omega x_\omega}\\
&=&\sum_{  (t,\dec)\in T_n^m(X)}\alpha_{t,\dec} \,t\left(\dec(v^t_1),\ldots, \dec(v^t_m)\right)\big|_{x\mapsto\sum_{\omega\in\Omega}\malpha_\omega x_\omega}\\
&=&\sum_{  (t,\dec)\in T_n^m(X)}\alpha_{t,\dec} \,t\left(\sum_{\omega\in\Omega}\malpha_\omega\dec(v^t_1)_\omega,\ldots, \sum_{\omega\in\Omega}\malpha_\omega\dec(v^t_m)_\omega\right)\\
&=&\sum_{  (t,\dec)\in T_n^m(X)}\alpha_{t,\dec} \,\sum_{\omega_1,\ldots,\omega_m\in\Omega}
\,t\left(\dec(v^t_1)_{\omega_1},\ldots, \dec(v^t_m)_{\omega_m}\right)\lambda_{\omega_1}\cdots\lambda_{\omega_m}\\
&=& \sum_{  (t,\dec)\in T_n^m(X)} \sum_{\deltal\in\lif{\dec,\fc}} \alpha_{t,\dec} \,(t,\deltal)\prod_{\omega\in \Omega} \malpha_\omega^{c_\omega}.
\end{eqnarray*}
}
Further by Eq.~\meqref{eq:polarhom}, the equality holds.

\noindent
\mref{item:replace}
By Eq.~\meqref{eq:qpc} and $|\lif{\dec,c}|=\mbin$, we have
\vsa
$$f_{\fc}\big(X_\Omega \big)\big|_{x_\omega\mapsto x}
=\sum_{(t,\dec)\in T_n^m(X)} \alpha_{t,\dec} \sum_{\deltal\in \lif{\dec,c}} (t,\deltal)\big|_{x_\omega\mapsto x}
=\sum_{(t,\dec)\in T_n^m(X)} \alpha_{t,\dec} \,\mbin\, (t,\dec)
=\tbinom{m}{\fc} f(X).
$$

\noindent
\mref{item:fullp}
By the definition of full polarizations, each decorative element in $X_\omega$ appears exactly once in $f_{\fc }\big(X_\Omega \big)$. Thus a full polarization is multilinear and symmetric in $X_\omega$ with $c_\omega=1$.
\end{proof}
\vsc

\subsubsection{Linear compatibility}
\mlabel{sec:linc}
We first recall the concept of linear compatibility of algebras. See~\mcite{ZGG23} and the references therein for further details.

\begin{definition}
Let $\Omega$ be a nonempty set. Let $\spp$ be an operad with generators $\mu_1, \ldots, \mu_k$. For each $\omega\in\Omega$, let $\mu_{\omega,1},\ldots,\mu_{\omega,k}$ be a replica of $\mu_1, \ldots, \mu_k$. Let $V$ be a vector space such that, for each $\omega\in \Omega$, the tuple
\vsb
$$A_\omega:=(V, \mu_{\omega,1},\ldots,\mu_{\omega,k})$$
is a $\spp$-algebra. Then $V$ is called an {\bf ($\Omega$-)linearly compatible $\spp$-algebra} if, for every linear combination $\nu_i:=\sum_{\omega\in\Omega}\malpha_\omega \mu_{\omega,i}\in \oplus_{\omega\in \Omega} \bfk u_{\omega,i}, i=1,\ldots,k,$ with finitely many coefficients nonzero, the tuple
\vsb
$$A:=A_{(\malpha_\omega)}: = (V, \nu_1, \ldots, \nu_k)$$
is still a $\spp$-algebra.
\mlabel{de:lina}
\end{definition}
\vsb

We next determine the operad governing linearly compatible $\mscr{P}$-algebras.

\begin{definition}
Let $\Omega$ be a nonempty set.
Let $M$ be an $\mathbb{S}$-module with a basis $X$ and let $\mscr{P}=\mscr{T}(M)/\langle R\rangle$ be a locally homogeneous operad.
We denote by $\lrr$ the set of \polars of all homogeneous polynomials $r(X)\in R$.
Explicitly,
\vsa
$$\lrr:=\lrr_{,\Omega}:=\bigg\{r_{\fc}\big(X_\Omega \big)\in \mscr{T}\Big(\bigoplus\limits_{\omega\in\Omega} M_\omega\Big)\, \bigg| \, r(X)\in  R, \fc\in\mwc{|r(X)|}\bigg\}.
\vsa
$$
We call the operad
\vsb
$$\lin{\mscr{P}}_\Omega:=\mscr{T}\Big(\bigoplus\limits_{\omega\in\Omega} M_\omega\Big)\Big/\Big\langle \lrr\Big\rangle \vsa
$$
the \textbf{linearly compatible operad} of $\spp$ with the parameter set $\Omega$.
\mlabel{de:comp1}
\end{definition}

We give an example from the commutative operad $\com$ of commutative algebras.
\begin{example}
\mlabel{ex:linas}
The commutative operad $\com=\mscr{T}(M)/\langle R_{\rm Com}\rangle$ is given by
\vsb
$$M=M_2=\bfk \bigg\{\treey{\cdlr{o}\node  at (ol) [above] {1}; \node  at (or) [above] {2};}=\treey{\cdlr{o}\node  at (ol) [above] {2};\node  at (or) [above] {1};}\bigg\}$$
and
\vsb
$$R_{\rm Com}=\bfk\Bigg\{r:=r(\cdot):=\treey{\cdlr{ol}\foreach \i/\j in { oll/1,olr/2,or/3} {\node at  (\i) [above] {\j};}}-\treey{\cdlr{ol}\foreach \i/\j in { oll/2,olr/3,or/1} {\node at  (\i) [above] {\j};}}
,s:=s(\cdot):=\treey{\cdlr{ol}\foreach \i/\j in { oll/1,olr/2,or/3} {\node at  (\i) [above] {\j};}}-\treey{\cdlr{ol}\foreach \i/\j in { oll/3,olr/1,or/2} {\node at  (\i) [above] {\j};}}\Bigg\} .$$
So $X=\{\cdot\}$.
Take $\Omega=\{\omega_1,\omega_2\}=\{\circ,\bullet\}$ and identify $X_\Omega=\{\cdot_\circ, \cdot_\bullet\}$ with $\Omega$.  We list all \polars $r_{i_\circ,i_\bullet}=r_{(i_\circ, i_\bullet)}(\circ, \bullet)$ and $s_{(i_\circ, i_\bullet)}=s_{(i_\circ, i_\bullet)}(\circ, \bullet)$ for $(i_\circ,i_\bullet)\in \ZZ_{\geq 0}^2$ with $i_\circ+i_\bullet=2$ as follows.
\begin{align*}
r_{(0,2)}&=\treey{\cdlr{ol}\foreach \i/\j in { oll/1,olr/2,or/3,o/$\bullet$,ol/$\bullet$} {\node at  (\i) [above] {\j};}}-\treey{\cdlr{ol}\foreach \i/\j in { oll/2,olr/3,or/1,o/$\bullet$,ol/$\bullet$} {\node at  (\i) [above] {\j};}},\quad
r_{(2,0)}=\treey{\cdlr{ol}\foreach \i/\j in { oll/1,olr/2,or/3,o/$\circ$,ol/$\circ$} {\node at  (\i) [above] {\j};}}-\treey{\cdlr{ol}\foreach \i/\j in { oll/2,olr/3,or/1,o/$\circ$,ol/$\circ$} {\node at  (\i) [above] {\j};}},\quad
r_{(1,1)}=
\treey{\cdlr{ol}\foreach \i/\j in { oll/1,olr/2,or/3,o/$\circ$,ol/$\bullet$} {\node at  (\i) [above] {\j};}}
-\treey{\cdlr{ol}\foreach \i/\j in { oll/2,olr/3,or/1,o/$\circ$,ol/$\bullet$} {\node at  (\i) [above] {\j};}}
+\treey{\cdlr{ol}\foreach \i/\j in { oll/1,olr/2,or/3,o/$\bullet$,ol/$\circ$} {\node at  (\i) [above] {\j};}}
-\treey{\cdlr{ol}\foreach \i/\j in { oll/2,olr/3,or/1,o/$\bullet$,ol/$\circ$} {\node at  (\i) [above] {\j};}},\\
s_{(0,2)}&=\treey{\cdlr{ol}\foreach \i/\j in { oll/1,olr/2,or/3,o/$\bullet$,ol/$\bullet$} {\node at  (\i) [above] {\j};}}-\treey{\cdlr{ol}\foreach \i/\j in { oll/3,olr/1,or/2,o/$\bullet$,ol/$\bullet$} {\node at  (\i) [above] {\j};}},\quad
s_{(2,0)}=\treey{\cdlr{ol}\foreach \i/\j in { oll/1,olr/2,or/3,o/$\circ$,ol/$\circ$} {\node at  (\i) [above] {\j};}}-\treey{\cdlr{ol}\foreach \i/\j in { oll/3,olr/1,or/2,o/$\circ$,ol/$\circ$} {\node at  (\i) [above] {\j};}},\quad
s_{(1,1)}=
\treey{\cdlr{ol}\foreach \i/\j in { oll/1,olr/2,or/3,o/$\circ$,ol/$\bullet$} {\node at  (\i) [above] {\j};}}
-\treey{\cdlr{ol}\foreach \i/\j in { oll/3,olr/1,or/2,o/$\circ$,ol/$\bullet$} {\node at  (\i) [above] {\j};}}
+\treey{\cdlr{ol}\foreach \i/\j in { oll/1,olr/2,or/3,o/$\bullet$,ol/$\circ$} {\node at  (\i) [above] {\j};}}
-\treey{\cdlr{ol}\foreach \i/\j in { oll/3,olr/1,or/2,o/$\bullet$,ol/$\circ$} {\node at  (\i) [above] {\j};}}.
\end{align*}
Then
\vsb
$$\lin{{\com}}_{\Omega}
=\mscr{T}\Big(M_\circ\oplus M_\bullet\Big)\Big/\Big\langle \lrr_{\rm Com,}\Big\rangle
=\mscr{T}\Big(\bfk \treey{\cdlr{o}\node  at (o) {$\circ$};\node  at (ol) [above] {1}; \node  at (or) [above] {2};}\oplus \bfk \treey{\cdlr{o}\node  at (o) {$\bullet$};\node  at (ol) [above] {1}; \node  at (or) [above] {2};}\Big)\Big/
\Big\langle r_{(0,2)},r_{(1,1)},r_{(2,0)}, s_{(0,2)},s_{(1,1)},s_{(2,0)}\Big\rangle.
$$
\end{example}

\begin{theorem}\mlabel{thm:comp}
Let $\Omega$ be a nonempty set and $\spp=\mscr{T}(M)/\langle R\rangle$ be a locally homogeneous operad.
A vector space $V$ is
a $\lin{\spp}_\Omega$-algebra if and only if it is an $\Omega$-linearly compatible $\spp$-algebra.
\end{theorem}

\begin{proof}
We only need to prove that, for any relation $r(X)\in R$, the space spanned by the \polars of $r(X)$ is equal to the space spanned by the linearly compatible relations of $r(X)$. Here the linearly compatible relations are induced by $r\big(X\big)\big|_{x\mapsto \sum_{\omega\in\Omega}\malpha_\omega x_{\omega}}$ for any choice of $\malpha_\omega\in\bfk$ with finitely many being nonzero.
Thus we are left to verify
$$\bfk\left\{r\big(X\big)\big|_{x\mapsto \sum_{\omega\in\Omega}\malpha_\omega x_{\omega}}\right\}=\bfk\left\{r_{\fc}\big(X_\Omega\big)\,\left|\,\fc\in\mwc{|r(X)|}\right.\right\},$$
which, by Eq.~\meqref{eq:polarhom} and the arbitrariness of $\malpha_\omega \in\bfk$, follows from
\vsb
$$\hspace{.5cm} \bfk\left\{r\big(X\big)\big|_{x\mapsto \sum_{\omega\in\Omega}\malpha_\omega x_{\omega}}\right\}
=\bfk \left\{\sum _{\fc\in\mwc{|r(X)|}} \Big(\prod{\malpha_\omega}^{\cw} r_{\fc}\big(X_\Omega\big)\Big)\right\}
=\bfk\bigg\{r_{\fc}\big(X_\Omega\big)\,\bigg|\,\fc\in\mwc{|r(X)|}\bigg\}. \hspace{.5cm} \qedhere
\vsb$$\end{proof}

An iteration of taking linear compatibility yields another linear compatibility as follows.
\begin{proposition}\mlabel{thm:ll}
For a nonempty set $\Omega$ and a locally homogeneous operad $\spp=\mscr{T}(M)/\langle R\rangle$,
\vsb
$$\lin{\Big(\lin{\spp}_\Omega\Big)}_\Omega=\lin{\spp}_{\Omega^2}.
\vsc
$$
\end{proposition}
\begin{proof}
By Definition~\mref{de:comp1}, the generators of the operads $\lin{\Big(\lin{\spp}_\Omega\Big)}_\Omega$ and $\lin{\spp}_{\Omega^2}$ are
identified:
$$ \bigoplus\limits_{\mu\in\Omega} \Big(\bigoplus\limits_{\nu\in\Omega} M_\nu\Big)_\mu = \bigoplus\limits_{(\mu,\nu)\in\Omega^2} M_{(\mu,\nu)}.$$
We also have the identification
$(X_\Omega)_\Omega \cong X\times \Omega \times \Omega \cong X_{\Omega^2}$.
So it suffices to prove that, for any relation $r(X)\in R^{(m)}$ for some $m\geq1$, the set of linearly compatible relations
$$\bfk\bigg\{\left.f_{\fcwp{d}{\nu}}\big(X_{\Omega^2}\big)\,\,\right|\,\,f(X_\Omega)\in \lrr, \sum_{\nu\in\Omega}\cwp{d}{\nu}=m\bigg\}
$$
of linearly compatible relations of $r(x)$ with parameter $\Omega$ coincides with the set of linearly compatible relations
\vsb
$$\bfk \bigg\{r_{\fcwp{c}{(\mu,\nu)}}\big(X_{\Omega^2}\big)\,\Big|\, \sum_{(\mu,\nu)\in\Omega^2}\cwp{c}{(\mu,\nu)}=m\bigg\}$$
of $r(X)$ with parameter $\Omega^2$.
Indeed, this follows from
{\small
\begin{eqnarray*}
\bfk\bigg\{f_{\fcwp{d}{\nu}}\big(X_{\Omega^2}\big)\,\,\Big|\,\,f(X_\Omega)\in \lrr, \sum_{\nu\in\Omega}\cwp{d}{\nu}=m\bigg\}
&=&\bfk\Bigg\{f_{\fcwp{d}{\nu}}\big(X_{\Omega^2}\big)\,\,\bigg|\,\,f(X_\Omega)\in \bfk\bigg\{r_{\fcwp{c}{\mu}}\big(X_\Omega\big)\,\Big|\, \sum_{\mu\in\Omega}\cwp{c}{\mu}=m \bigg\}, \sum_{\nu\in\Omega}\cwp{d}{\nu}=m\Bigg\}\\
&=&\bfk\Bigg\{\Big(r_{\fcwp{c}{\mu}}\big(X_\Omega\big)\Big)_{\fcwp{d}{\nu}}\big(X_{\Omega^2}\big)\,\bigg|\, \sum_{\mu\in\Omega}\cwp{c}{\mu}=m, \sum_{\nu\in\Omega}\cwp{d}{\nu}=m\Bigg\}\\
\hspace{4cm} &=&\bfk \bigg\{r_{\fcwp{c}{(\mu,\nu)}}\big(X_{\Omega^2}\big)\,\Big|\, \sum_{(\mu,\nu)\in\Omega^2}\cwp{c}{(\mu,\nu)}=m\bigg\}. \hspace{3.6cm}\qedhere
\end{eqnarray*}
}
\end{proof}
\vsb
\subsection{Foliations for operads and matching compatibilities}
\mlabel{ss:mc}
In this subsection, we consider a stronger compatibility, called the matching compatibility using the refined notion of foliation from polarization. See \S\,\mref{ss:folia} for the ``baby model'' for polynomials.

As above, let $\Omega$ be a nonempty set.
Let $M$ be an $\mathbb{S}$-module with a basis $X$ and
\begin{equation}
 f(X):=\sum_{ 0\leq i \leq s }\alpha_i \, \big(t_i, \dec_i\big)\in \mscr{T}(M)^{(m)}(n)
 \mlabel{eq:rel2}
 \vsc
\end{equation}
with $0\neq \alpha_i\in \bfk$. Fix $\fc\in\mwc{m}$.
For notational clarity, denote $$\sbin:=\sbin_\fc:=\mbin=|\lif{\dec_i,c}|.$$
By Eq.~\meqref{eq:ores}, for an \owc $\fc=\fcw\in\mwc{m}$, we have the {\bf restitution} of $f:$
$$f_{\fc }\big(X_\Omega\big)\big|_{x_\omega\mapsto x}=\tbinom{m}{\fc} f(X).$$
A {\bf foliation} of $f_{\fc}$ is a partition of $f_\fc$ into parts $f_{\fc,k}, 1\leq k\leq \sbin$, with the property that
$$f_{\fc,k}\big(X_\Omega\big)\big|_{x_\omega\mapsto x}=f(X).$$

We next give an explicit construction of the other foliations, each defines a matching compatibility. This process needs further notations, but we are awarded with a uniformly defined class of matching compatibilities which is wide enough to cover previously known compatibility constructions. By combining different foliations, we also obtain the newly introduced multi-Novikov algebra from stochastic PDEs\,\cite{BruD}. See Example\,\mref{ex:match}, especially Item\,\mref{it:matchnov}.

Applying Definition~\mref{defn:lift} and Remark~\mref{remark:lif}, list the set of \lifts of $d_i$ of type $\fc$ as
\begin{equation}
\lif{\dec_i, c}=\Big\{\deltal_{i,1},\deltal_{i,2},\ldots,\deltal_{i,\sbin}\Big\},\quad 0\leq i\leq s,
\mlabel{eq:colord}
\vsc
\end{equation}
which we abbreviate to column vectors $\vec{D}_i, 0\leq i \leq s .$ Likewise, let $(t_i,\vec{D}_i)$ denote the column vector $\big((t_i,D_{i,1}),\ldots,(t_i,D_{i,\sbin})\big)^T$.
Then $(t_i,\dec_i)_\fc$ is the sum of the entries of $(t_i,\vec{D}_i)$,
and the \polar of $f(X)$ of type $\fc$ is obtained from $f(X)$ in replacing $(t_i,d_i)$ in Eq.~\meqref{eq:rel2} by $(t_i,\dec_i)_\fc$:
$$f_\fc(X_\Omega)= \sum_{0\leq i \leq s } \alpha_i (t_i,\dec_i)_\fc= \sum_{0\leq i \leq s } \alpha_i \sum_{\deltal\in \lif{\dec_i,c}}(t_i,\deltal)
= \sum_{0\leq i \leq s } \alpha_i \sum_{1\leq j\leq \sbin}(t_i,\deltal_{i,j})
= \sum_{1\leq j\leq \sbin} \sum_{0\leq i \leq s } \alpha_i (t_i,\deltal_{i,j}).
$$
Thus the sum $f_\fc(X_\Omega)$ is naturally partitioned into the system
\begin{equation}
\Big\{\sum_{0\leq i \leq s } \alpha_i (t_i,\deltal_{i,j})\,\Big|\, 1\leq j\leq \sbin\Big\} = \left (\begin{array}{cccc}
	(t_0,D_{0,1}) & (t_1,D_{1,1}) & \cdots & (t_s,D_{s,1}) \\
	(t_0,D_{0,2}) & (t_1,D_{1,2}) & \cdots & (t_s,D_{s,2}) \\
	\vdots & \vdots & \ddots & \vdots \\
	(t_0,D_{0,\sbin}) & (t_1,D_{1,\sbin}) & \cdots & (t_s,D_{s,\sbin})
\end{array} \right ) \left(\begin{array}{c} \alpha_0\\ \alpha_1\\
\vdots \\ \alpha_s\end{array}\right).
\mlabel{eq:sfolia}
\end{equation}
We call this system the {\bf standard foliation} of $f_\fc(X_\Omega)$ (with respect to the ordering in Eq.~\meqref{eq:colord}).

Next for $\sigma \in S_\sbin$, the symmetric group on $\sbin$ letters, let $(t_i,\vec{D})^\sigma$ denote the permutation of the $i$-th column the matrix in Eq.~\meqref{eq:sfolia} by $\sigma$:
\vsa
$$ (t_i,\vec{D})^\sigma:=(t_i, \vec{D}^\sigma):=
\big((t_i,D_{i,\sigma(1)}),\ldots, (t_i,D_{i,\sigma(\sbin)})\big)^T.\vsb$$
Replacing the matrix $F$ in Eq.~\meqref{eq:sfolia} by the one with the $i$-column permuted by $\sigma_i, 1\leq i\leq s$, we obtain another foliation of $f_\fc(X_\Omega)$ shown below.
To keep track with multiple $f$ and $c$ which might share the same value $\sbin$, we denote $\pcos{\fc}{f(X)}:={S_\sbin}^{s}$.
\vsb
\begin{definition}
Let $f(X)=\sum\limits_{ 0\leq i \leq s  }\alpha_i \big(t_i, \dec_i\big)\in \mscr{T}(M)^{(m)}(n)$ and $\vec{\sigma}=(\sigma_1,\ldots,\sigma_{s})\in \pcos{\fc}{f(X)}:={S_\sbin}^{s}$.
Define the {\bf \spli of the \polar $f_{\fc}\big(X_\Omega\big)$ \tw $\vec{\sigma}$}
to be the set
\vsb
\begin{equation*}
\prc{f(X)}{\vec{\sigma}}
	:=\alpha_0(t_0,\vec{D}_0)+\sum_{1\leq i\leq s} \alpha_i(t_i,\vec{D}_i^{\sigma_i})
=\Big\{\alpha_0 (t_0, \deltal_{0,j})+
	\sum_{ 1\leq i \leq s }\alpha_i (t_i, \deltal_{i,\sigma_i(j)})
	\,\Big|\, 1\leq j\leq \sbin\Big\}.
\mlabel{eq:split2}
\end{equation*}
\mlabel{defn:partm}
\vsd
\end{definition}

\begin{remark}\mlabel{remark:p}
Let $f_{\fc}\big(X_\Omega \big)$ be a  \polar of $f(X)\in \mscr{T}(M)^{(m)}(n)$ of type $\fc\in\mwc{m}$.
\begin{enumerate}
\item
From $|{S_{\sbin}}^{s}|=(\sbin!)^{s}$, there are $(\sbin!)^{s}$ \splis of the \polar $f_{\fc}\big(X_\Omega\big)$.
\item
By Eq.~\meqref{eq:acts}, we have
$\prc{f(X)}{\vec{\sigma}}\cdot\rho
=\alpha_0(t_0\cdot\rho,\vec{D}_0)+\sum_{1\leq i\leq s} \alpha_i(t_i\cdot\rho,\vec{D}_i^{\sigma_i})
=\prc{\big(f(X)\cdot\rho\big)}{\vec{\sigma}}
$ for $\rho\in S_n$.
\mlabel{item:symc}
\end{enumerate}
\vsc
\end{remark}
\vsc
Let $ f(X)\in\mscr{T}(M)$ be given in Eq.~\meqref{eq:opfun}.
The {\bf support} ${\rm Supp}(f(X))$ of $f(X)$ is defined to be the set of all decorated trees $(t,\dec)$ with $\alpha_{t,\dec}\neq0$.

\begin{definition}
Let $\Omega$ be a nonempty set.
Let $M$ be an $\mathbb{S}$-module with a basis $X$ and $\mscr{P}=\mscr{T}(M)/\langle R\rangle$ be a locally homogeneous operad.
\begin{enumerate}
\item
Let $r=r(X)\in \mscr{T}(M)^{(m)}$. For better distinction, for $c\in \mwc{m}$, denote $\pcos{\fc}{r}={S_{\sbin}}^{|{\rm Supp}(r)|-1}$
even though the set is independent of $c$. Define
$$\pr{r}:=\prod_{c\in\mwc{m}}\pcos{\fc}{r}=\prod_{c\in\mwc{m}}{S_{\sbin}}^{|{\rm Supp}(r)|-1}=\Big({S_{\sbin}}^{|{\rm Supp}(r)|-1)}\Big)^{|\mwc{m}|}.$$
Given $\sigmab:=(\vec{\sigma}_c)_{c\in\mwc{m}}\in \pr{r}$ with $\vec{\sigma}_c\in\pcos{\fc}{r}$, the {\bf \mpr of $r(X)$ with respect to ${\sigmab}$ } is the union of the \splis
\vsb
$$r(X)\ma{\sigmab}:=\bigcup_{\fc\in\mwc{m}}\prc{r(X)}{\vec{\sigma}_\fc}.
\vsd$$
\item Let $R\subseteq\mscr{T}(M)$. Define
\begin{equation*}
\Lambda(R):=\prod_{r\in  R}\pr{r}= \prod_{r\in  R}\prod_{c\in\mwc{|r|}}\pcos{\fc}{r}=\prod_{r\in  R}\prod_{c\in\mwc{|r|}}{S_{\sbin}}^{|{\rm Supp}(r)|-1}
\vsb
\end{equation*}
and
\vsb
\begin{equation}
\pc{R}:=\Big\{{\sigmab}\in\Lambda(R)\,\big|\,r\in R, \rho\in S_n,
\left(r(X)\ma{\sigmab}\right)\cdot\rho=\left(r(X)\cdot\rho\right)\ma{\sigmab}\Big\}.
\mlabel{eq:matset}
\end{equation}
The {\bf \mpr of $R$ with respect to $\sigmab\in\pc{R}$ } is the set
$$R\ma{\sigmab}:=\bigcup_{r(X)\in R}r(X)\ma{\sigmab}=\bigcup_{r(X)\in R;~\fc\in\mwc{m}}\prc{r(X)}{\vec{\sigma}_{\fc,r}}.$$
Elements of $R\ma{\sigmab}$ are called the {\bf matching relations } of $R$ with respect to $\sigmab$.
\end{enumerate}
\mlabel{de:spr}
\end{definition}

\begin{remark}
\begin{enumerate}
\item The matching compatibility from the standard foliation is precisely the one introduced in ~\mcite{ZGG23} for nonsymmetric operads.
\item
Since $R\ma{\sigmab}\neq R\ma{\pmd}$ as long as $\pmc\neq \pmd\in\pc{R}$, there are $|\pc{R}|$ \mprs of $R$.
\item
Let $\bar{R}$ be an $\mathbb{S}$-basis of $R$.
From  Eq.~\meqref{eq:matset} and $|{S_{\sbin_c}}^{|{\rm Supp}(r)|-1}|=(\sbin_c!)^{|{\rm Supp}(r)|-1}$,
there are  $|\pc{R}|=|\pc{\bar{R}}|=\prod_{r\in \bar{R}} \prod_{c\in \mwc{|r|}}(\sbin_c!)^{|{\rm Supp}(r)|-1}$ \mprs of $R$.
\mlabel{item:rer}
\item
The condition in Eq.~\meqref{eq:matset} is imposed so that the concept of matching algebras is well-defined.
Indeed, let $A$ be a $\spp$-algebra. For any $n$-arity operations $r\in R$, $\rho\in S_n$ and $a_1,\ldots,a_n\in A$,
we have
\vsc
$$r(a_1,\ldots,a_n)=0\Leftrightarrow r\cdot\rho(a_1,\ldots,a_n)=0,
$$
by the arbitrariness of $a_1,\ldots,a_n\in A$. Hence
\vsc
$$ \left(r\ma{\sigmab}\right)\cdot\rho(a_1,\ldots,a_n)=0\Leftrightarrow \left(r\ma{\sigmab}\right)(a_1,\ldots,a_n)=0\Leftrightarrow \left(r\cdot\rho\right)\ma{\sigmab}(a_1,\ldots,a_n)=0.$$
Notice that $ \left(r\ma{\sigmab}\right)\cdot\rho$ and $\left(r\cdot\rho\right)\ma{\sigmab}$ have the same underlying tree. So we expect the invariant relation
\vsc
$$ \left(r\ma{\sigmab}\right)\cdot\rho=\left(r\cdot\rho\right)\ma{\sigmab}.
\vsc$$
\end{enumerate}
\mlabel{remark:matc}
\end{remark}
\vsb
\begin{example}
Let $X=\{\prec,\succ\}$. Let $R=\{r(X)\cdot\rho \,|\,\rho\in S_3\}$ with
$r(X)=(t_0,\dec_0)-(t_1,\dec_1)-(t_2,\dec_2)$
where
\vsc
$$(t_0,\dec_0)=\treey{\cdlr{ol}\foreach \i/\j in { oll/1,olr/2,or/3} {\node at  (\i) [above]  [scale=0.8]{\j};}\foreach \i/\j in { o/$\prec$,ol/$\prec$} {\node at  (\i) [left] [scale=0.8]{\j};}},\quad
(t_1,\dec_1)=\treey{\cdlr{or}\foreach \i/\j in { ol/3,orl/1,orr/2} {\node at  (\i) [above]  [scale=0.8]{\j};}\foreach \i/\j in { o/$\prec$,or/$\prec$} {\node at  (\i) [right] [scale=0.8]{\j};}},\quad
(t_2,\dec_2)=\treey{\cdlr{or}\foreach \i/\j in { ol/3,orl/1,orr/2} {\node at  (\i) [above]  [scale=0.8]{\j};}\foreach \i/\j in { o/$\prec$,or/$\succ$} {\node at  (\i) [right] [scale=0.8]{\j};}}.$$
Take $\Omega=\{\circ,\bullet\}$.
First for  $\fc=(\fc_\circ,\fc_\bullet)=(1,1)$, we have
$\sbin_c=\tbinom{2}{1}=2$. So for each $(t_i,d_i),0\leq i\leq 2$, we have
\vsc
$$ (t_i,\vec{D}_i)=\big((t_i,D_{i,1}), (t_i,D_{i,2})\big).$$
Then there are four \splis of $r_c(X_\Omega)$, determined by the four pairs of permutations $(\sigma_1,\sigma_2)\in\pcos{\fc}{r(X)}= {S_2}^2$:
\vsc
$$r(X)^{(\sigma_1,\sigma_2)}:=
(t_0,\vec{D}_0)-(t_1,\vec{D}_1^{\sigma_1})-(t_2,\vec{D}_2^{\sigma_2}).$$
One of the \splis is
\vsc
$$\treey{\cdlr{ol}\foreach \i/\j in { oll/1,olr/2,or/3} {\node at  (\i) [above]  [scale=0.8]{\j};}\foreach \i/\j in { o/$\prec_{\circ}$,ol/$\prec_{\bullet}$} {\node at  (\i) [left] [scale=0.8]{\j};}}
-\treey{\cdlr{or}\foreach \i/\j in { ol/3,orl/1,orr/2} {\node at  (\i) [above]  [scale=0.8]{\j};}\foreach \i/\j in { o/$\prec_{\circ}$,or/$\prec_{\bullet}$} {\node at  (\i) [right] [scale=0.8]{\j};}}
-\treey{\cdlr{or}\foreach \i/\j in { ol/3,orl/1,orr/2} {\node at  (\i) [above]  [scale=0.8]{\j};}\foreach \i/\j in { o/$\prec_{\bullet}$,or/$\succ_{\circ}$} {\node at  (\i) [right] [scale=0.8]{\j};}}.$$
Next for $c=(2,0)$ or $(0,2)$, we have $|\pcos{c}{r(X)}|=1$. So they only give one \spli $r(X_\circ)$ or $r(X_\bullet)$ of $r_{c}(X_\circ, X_\bullet)$.
Therefore we have the following four \mprs of $r(X)$
\begin{eqnarray*}
\big\{r(X_\circ),r(X_\bullet)\big\}\cup \big\{r(X)^{(\sigma_1,\sigma_2)} \big\}, \text{ for }\,(\sigma_1,\sigma_2)\in {S_2}^2.
\end{eqnarray*}

Moreover, for $\rho\in S_3$, we have
\vsd
$$r(X)\cdot\rho=(t_0\cdot\rho,\dec_0)-(t_1\cdot\rho,\dec_1) -(t_2\cdot\rho,\dec_2)
=\treey{\cdlr{ol}\foreach \i/\j in { oll/$\rho^{-1}(1)$,olr/$\rho^{-1}(2)$,or/$\rho^{-1}(3)$} {\node at  (\i) [above]  [scale=0.6]{\j};}\foreach \i/\j in { o/$\prec$,ol/$\prec$} {\node at  (\i) [left] [scale=0.8]{\j};}}-
\treey{\cdlr{or}\foreach \i/\j in { ol/$\rho^{-1}(3)$,orl/$\rho^{-1}(1)$,orr/$\rho^{-1}(2)$} {\node at  (\i) [above]  [scale=0.6]{\j};}\foreach \i/\j in { o/$\prec$,or/$\prec$} {\node at  (\i) [right] [scale=0.8]{\j};}}-
\treey{\cdlr{or}\foreach \i/\j in { ol/$\rho^{-1}(3)$,orl/$\rho^{-1}(1)$,orr/$\rho^{-1}(2)$} {\node at  (\i) [above]  [scale=0.6]{\j};}\foreach \i/\j in { o/$\prec$,or/$\succ$} {\node at  (\i) [right] [scale=0.8]{\j};}}.$$
By Remark~\mref{remark:matc}\mref{item:rer},
$|\pc{R}|=1\times 1\times 2^2=4.$
Thus there are four \mprs of $R$:
$$\big\{r(X_\circ)\cdot\rho,r(X_\bullet)\cdot\rho\,|\,\rho\in S_3\big\}\cup \big\{\big(r(X)^{(\sigma_1,\sigma_2)} \big)\cdot\rho\,|\,\rho\in S_3\big\}, \text{ for }\,(\sigma_1,\sigma_2)\in {S_2}^2.$$
\end{example}

\begin{definition}
Let $\Omega$ be a nonempty set.
Let $\mscr{P}=\mscr{T}(M)/\langle R\rangle$ be a locally homogeneous operad and  let $\pmc\in\pc{R}$ be as given in Eq.~\meqref{eq:matset}.
\begin{enumerate}
\item We call the operad
\vsc
$$\mat{\mscr{P}}_\Omega:=\mscr{T}\Big(\bigoplus\limits_{\omega\in\Omega} M_\omega\Big)\Big/\Big\langle R\ma{\sigmab}\Big\rangle
\vsa
$$
the \textbf{matching operad} of $\spp$ with parameters $\Omega$ and $\pmc$.
\mlabel{item:mat}

\item Let $\pordq$ be a total order of the vertices of each planar rooted tree $t$. Let $\fc\in\mwc{m}$ and $r(X)=\sum_{ 0\leq i \leq s }\alpha_i \big(t_i, \dec_i\big)\in R^{(m)}$. A matching relation
$$\sum_{ 0\leq i \leq s }\alpha_i \,t_i\big(\deltal_{i}(v_{i,1}), \cdots,\deltal_{i}(v_{i,m}) \big)\in R\ma{\sigmab},
\quad v_{i,1}\prord\,\cdots\,\prord v_{i,m}\in V(t_i)$$
of  $r(X)$ is called {\bf \lev} (with respect to $\pordq$) if for any $0\leq k,\ell\leq s $ and any $1<j<m$,
we have $\pi_\Omega\circ \deltal_{k} (v_{k, j})=\pi_\Omega\circ\deltal_{\ell}(v_{\ell,j})$,
where $\pi_\Omega$ is the projection in Eq.~\meqref{eq:proj}.

\item A \mpr $R\ma{\sigmab}$ of $R$ is called {\bf \lev} (with respect to $\pordq$) if each matching relation is \lev. Let $\lmr$ denote this \lev matching compatibility of $R$.

\item A matching operad
$\mat{\mscr{P}}_\Omega=\mscr{T}\Big(\bigoplus\limits_{\omega\in\Omega} M_\omega\Big)\Big/\Big\langle R\ma{\sigmab}\Big\rangle $
is called {\bf \lev} (with respect to $\pordq$) if
 $\mrr$ is \lev. Denote by $\lmat{\spp}_\Omega$ the \lev matching operad (with respect to $\pordq$).
\mlabel{item:lmat}
\end{enumerate}
\mlabel{de:mat}
\end{definition}

\begin{remark}
For a given total order $\pordq$ of the vertices, there is unique \lev matching operad.
As an example, we consider a concrete order $\pordq$, namely, the (depth-first) preorder traversal of  a tree $t$, by first visiting the root vertex and then the subtrees from left to right~\mcite{CLRS}: $V(t)=(v_1,\ldots,v_m)$. Then the pair $(t,\dec)$ can be uniquely written as
\begin{equation}
(t,\dec)=t(\dec(v_1),\ldots,\dec(v_m)).
\mlabel{eq:order}
\end{equation}
For convenience, we work with the \lev matching operad with respect to this preorder in the following, even though the results remain true for other total orders.
\end{remark}

In practice, we can quite easily list all matching compatibilities of a locally homogeneous operad, for any coloring set $\Omega$. We take the dendriform algebra as an example.

\begin{example} \mlabel{eg:dend}
For the dendriform algebra, the generator set is $X=\{\prec, \succ\}$ and the relation set is
\vsb
\begin{eqnarray*}
&r_1:=(x\prec y)\prec z-x\prec (y\prec z)-x\prec (y\succ z),& \\
& r_2:=(x\prec y)\succ z - x\succ (y\prec z), &\\
&r_3:=(x\prec y)\succ z+(x\succ y)\succ z - x\succ (y\succ z).&
\end{eqnarray*}
Let $\Omega$ with $|\Omega|=2$.
For $\fc=(c_{\omega_1},c_{\omega_2})=(1,1), \omega_1\neq \omega_2$,
there are four (resp. four, resp. two) \splis of $r_1$ (resp. $r_2$, resp. $r_3$), determined by the pairs of permutations $(\sigma_1,\sigma_2)\in\pcos{\fc}{r_1}= {S_2}^2$
(resp. $\sigma\in\pcos{\fc}{r_2}= {S_2}$, resp. $(\tau_1,\tau_2)\in\pcos{\fc}{r_3}= {S_2}^2$):
\vsa
\begin{eqnarray}
r_1^{(\sigma_1,\sigma_2)}
&=&\left\{(x\prec_{\omega_1} y)\prec_{\omega_2} z-x\prec_{\omega_{\sigma_1(1)}} (y\prec_{\omega_{\sigma_1(2)}} z)-x\prec_{\omega_{\sigma_2(1)}} (y\succ_{\omega_{\sigma_2(2)}} z)\,|\,\omega_1\neq \omega_2\in \Omega\right\},\nonumber\\
r_2^{\sigma}
&=&\left\{(x\succ_{\omega_1} y)\prec_{\omega_2} z-x\succ_{\omega_{\sigma(1)}} (y\prec_{\omega_{\sigma(2)}} z)\,|\,\omega_1\neq \omega_2\in \Omega\right\},\mlabel{eq:dendm}\\
r_1^{(\tau_1,\tau_2)}
&=&\left\{(x\prec_{\omega_1} y)\succ_{\omega_2} z-(x\succ_{\omega_{\tau_1(1)}} y)\succ_{\omega_{\tau_1(2)}} z-x\succ_{\omega_{\tau_2(1)}} (y\succ_{\omega_{\tau_2(2)}} z)\,|\,\omega_1\neq \omega_2\in \Omega\right\}.\nonumber
\end{eqnarray}
Notice that the condition in Eq.~\meqref{eq:matset} can be obtained by the arbitrariness of $x,y,z$. Therefore there exist
$32=2^2\times 2 \times2^2$ types of the matching dendriform algebras with parameter $\Omega$.
Moreover, the unique \lev matching dendriform algebra (with respect to the preorder given in Eq.~\meqref{eq:order}) is determined by the relations
$r_1^{(e, e)},r_2^{e}, r_3^{(e, e)}$, where $e$ is the unit of  $S_2$. It first appeared in~\cite{ZGG23}.
\end{example}

\begin{remark}
We pay give special attention to the naturally arisen matching dendriform algebra introduced in ~\cite{ZGG20} from Volterra integral operators~\cite{GGL}. Let $A$ be the $\mathbb{R}$-algebra of continuous functions on $\mathbb{R}$. For $k_\omega(x)\in A, \omega \in \Omega$, consider the Volterra integral operators with kernel $k_\omega$:
$$I_\omega(f)(x):=\int_0^x k_\omega(t)f(t)dt, \quad f\in A.$$
Define
\vsb
$$ f \prec_\omega g:=f(x)\int_0^x k_\omega(t) g(t)dt, \quad f\succ_\omega g:=\Big(\int_0^x k_\omega(t) f(t)dt\Big) g(x).$$
Then $(A,\prec_\omega, \succ_\omega |\omega \in \Omega)$ is a matching dendriform algebra in the sense of \cite{ZGG20}. Another importance of this construction is that, by antisymmetrization, this matching dendriform algebra gives rise to the matching (multiple) pre-Lie algebra~\cite{Foi,ZGG20} from the landmark work~\mcite{BHZ} on regularity structures. Unfortunately, this construction does not correspond to the matching compatibility defined in~\cite{ZGG23}. Being able to cover this natural construction of compatibility is a primary motivation for our foliation approach to a more general notion of matching compatibilities. Now this construction corresponds to the matching compatibility $r_1^{(e, (12))},r_2^{e}, r_3^{((12), (12))}$ in Eq.~\meqref{eq:dendm}.
\mlabel{rk:mdend}
\end{remark}

We give more examples for later applications and to show that the other previous instances of matching compatibilities are special cases of our general notion.
\begin{example}\mlabel{ex:aspolar}\label{ex:match}
Let $\Omega$ be a set with $|\Omega|=2$.
\begin{enumerate}
\item
The \lev matching compatible operad (with respect to the preorder given in Eq.~\meqref{eq:order})
$$\lmat{\com}_\Omega=\stt\Big(\bigoplus_{\omega\in \Omega} M_\omega \Big)\Big/\Big\langle \lmr\Big\rangle
\vsb$$
of the commutative operad is given by
\begin{gather*}
M_\omega=\bfk \left\{\treey{\cdlr{o}\node  at (ol) [above] {1}; \node  at (or) [above] {2};\node at  (o)[left=0.2] [scale=0.8]{$\omega$};}=\treey{\cdlr{o}\node  at (ol) [above] {2};\node  at (or) [above] {1};\node at  (o)[left=0.2] [scale=0.8]{$\omega$};}\right\},\quad\omega\in\Omega\\
\lmr=~\Biggr\{\treey{\cdlr{ol}\node at (oll)[above=0.2] [scale=0.8]{$1$};\node at (olr)[above=0.2] [scale=0.8]{$2$};\node at (or)[above=0.2] [scale=0.8]{$3$};
\node at  (ol)[left=0.2] [scale=0.8]{$\mu$};\node at  (o)[left=0.2] [scale=0.8]{$\mu$};}
-\treey{\cdlr{ol}\node at (oll)[above=0.2] [scale=0.8]{$3$};\node at (olr)[above=0.2] [scale=0.8]{$1$};\node at (or)[above=0.2] [scale=0.8]{$2$};
\node at  (ol)[left=0.2] [scale=0.8]{$\mu$};\node at  (o)[left=0.2] [scale=0.8]{$\mu$};},\,
\treey{\cdlr{ol}\node at (oll)[above=0.2] [scale=0.8]{$1$};\node at (olr)[above=0.2] [scale=0.8]{$2$};\node at (or)[above=0.2] [scale=0.8]{$3$};
\node at  (ol)[left=0.2] [scale=0.8]{$\mu$};\node at  (o)[left=0.2] [scale=0.8]{$\mu$};}
-\treey{\cdlr{ol}\node at (oll)[above=0.2] [scale=0.8]{$2$};\node at (olr)[above=0.2] [scale=0.8]{$3$};\node at (or)[above=0.2] [scale=0.8]{$1$};
\node at  (ol)[left=0.2] [scale=0.8]{$\mu$};\node at  (o)[left=0.2] [scale=0.8]{$\mu$};},\,
\treey{\cdlr{ol}\node at (oll)[above=0.2] [scale=0.8]{$1$};\node at (olr)[above=0.2] [scale=0.8]{$2$};\node at (or)[above=0.2] [scale=0.8]{$3$};
\node at  (ol)[left=0.2] [scale=0.8]{$\mu$};\node at  (o)[left=0.2] [scale=0.8]{$\nu$};}
-\treey{\cdlr{ol}\node at (oll)[above=0.2] [scale=0.8]{$3$};\node at (olr)[above=0.2] [scale=0.8]{$1$};\node at (or)[above=0.2] [scale=0.8]{$2$};
\node at  (ol)[left=0.2] [scale=0.8]{$\mu$};\node at  (o)[left=0.2] [scale=0.8]{$\nu$};},\,
\treey{\cdlr{ol}\node at (oll)[above=0.2] [scale=0.8]{$1$};\node at (olr)[above=0.2] [scale=0.8]{$2$};\node at (or)[above=0.2] [scale=0.8]{$3$};
\node at  (ol)[left=0.2] [scale=0.8]{$\mu$};\node at  (o)[left=0.2] [scale=0.8]{$\nu$};}
-\treey{\cdlr{ol}\node at (oll)[above=0.2] [scale=0.8]{$2$};\node at (olr)[above=0.2] [scale=0.8]{$3$};\node at (or)[above=0.2] [scale=0.8]{$1$};
\node at  (ol)[left=0.2] [scale=0.8]{$\mu$};\node at  (o)[left=0.2] [scale=0.8]{$\nu$};}
\,\Biggr|\,
\mu,\nu\in\Omega
\Biggr\}.
\vsa
\end{gather*}

\item Consider the associative operad $\as$ of associative algebras defined by the associative relation
\vsb
$r=\treey{\cdlr{ol}\node at (oll)[above=0.2] [scale=0.8]{$1$};\node at (olr)[above=0.2] [scale=0.8]{$2$};\node at (or)[above=0.2] [scale=0.8]{$3$};}-
\treey{\cdlr{or}\node at (ol)[above=0.2] [scale=0.8]{$1$};\node at (orl)[above=0.2] [scale=0.8]{$2$};\node at (orr)[above=0.2] [scale=0.8]{$3$};}.$
We get two \mprs of $r$, for $\sigma\in S_2$:
\begin{align*}
r_{{\rm MT}^\sigma}=~\Biggr\{\treey{\cdlr{ol}\node at (oll)[above=0.2] [scale=0.8]{$1$};\node at (olr)[above=0.2] [scale=0.8]{$2$};\node at (or)[above=0.2] [scale=0.8]{$3$};\foreach \i/\j in { o/$\omega_1$,ol/$\omega_2$} {\node at  (\i) [left] [scale=0.8]{\j};}}
-
\treey{\cdlr{or}\node at (ol)[above=0.2] [scale=0.8]{$1$};\node at (orl)[above=0.2] [scale=0.8]{$2$};\node at (orr)[above=0.2] [scale=0.8]{$3$};\foreach \i/\j in { o/$\omega_{\sigma(1)}$,or/$\omega_{\sigma(2)}$} {\node at  (\i) [right] [scale=0.8]{\j};}}\,\Biggr|\,
\omega_1,\omega_2\in\Omega\Biggr\}.
\end{align*}
The matching associative algebra with relations $r_{{\rm MT}^{(12)}}$ was studied in~\mcite{ZBG, ZGG23}. The matching relation $r_{{\rm MT}^e}$ is \lev. 

\item Let $\mscr{P}re\mscr{L}=\mscr{T}(M)/\langle R_{\rm PL}\rangle $ be the operad of (left) pre-Lie algebras with relation
$$R_{\rm PL}=\Biggr\{\treey{\cdlr{ol}\foreach \i/\j in { oll/$\rho(1)$,olr/$\rho(2)$,or/$\rho(3)$} {\node at  (\i) [above] [scale=0.8]{\j};}}
-\treey{\cdlr{or}\foreach \i/\j in { ol/$\rho(1)$,orl/$\rho(2)$,orr/$\rho(3)$} {\node at  (\i) [above] [scale=0.8]{\j};}}
-\treey{\cdlr{ol}\foreach \i/\j in { oll/$\rho(2)$,olr/$\rho(1)$,or/$\rho(3)$} {\node at  (\i) [above] [scale=0.8]{\j};}}
+\treey{\cdlr{or}\foreach \i/\j in { ol/$\rho(2)$,orl/$\rho(1)$,orr/$\rho(3)$} {\node at  (\i) [above] [scale=0.8]{\j};}}\,\Biggr|\, \rho\in S_3 \Biggr\}.$$
Then there are eight \mprs of $R_{\rm PL}$ parameterized by ${S_2}^3$. More precisely, for $\sigmab=(\sigma_1,\sigma_2,\sigma_3)\in {S_2}^3$, take
$$R_{{\rm PL}, {\rm MT}^{\sigmab}}:=\Biggr\{\treey{\cdlr{ol}\foreach \i/\j in { oll/$\rho(1)$,olr/$\rho(2)$,or/$\rho(3)$} {\node at  (\i) [above] [scale=0.8]{\j};}\foreach \i/\j in { ol/$\omega_{1}$,o/$\omega_{2}$} {\node at  (\i) [left] [scale=0.8]{\j};}}
-\treey{\cdlr{or}\foreach \i/\j in { ol/$\rho(1)$,orl/$\rho(2)$,orr/$\rho(3)$} {\node at  (\i) [above] [scale=0.8]{\j};}\foreach \i/\j in { o/$\omega_{\sigma_1(1)}$,or/$\omega_{\sigma_1(2)}$} {\node at  (\i) [right] [scale=0.8]{\j};}}
-\treey{\cdlr{ol}\foreach \i/\j in { oll/$\rho(2)$,olr/$\rho(1)$,or/$\rho(3)$} {\node at  (\i) [above] [scale=0.8]{\j};}\foreach \i/\j in { ol/$\omega_{\sigma_2(1)}$,o/$\omega_{\sigma_2(2)}$} {\node at  (\i) [left] [scale=0.8]{\j};}}
+\treey{\cdlr{or}\foreach \i/\j in { ol/$\rho(2)$,orl/$\rho(1)$,orr/$\rho(3)$} {\node at  (\i) [above] [scale=0.8]{\j};}\foreach \i/\j in { o/$\omega_{\sigma_3(1)}$,or/$\omega_{\sigma_3(2)}$} {\node at  (\i) [right] [scale=0.8]{\j};}}\,\Biggr|\, \omega_1,\omega_2\in\Omega, \rho\in S_3 \Biggr\}.$$
The matching pre-Lie algebra with  relation $R_{{\rm PL, MT}^\sigma}$ for $\sigmab=(\sigma_1,\sigma_2,\sigma_3)=\big(e,(12),(12)\big)\in {S_2}^3$, first appeared in~\cite{BHZ} and further studied in~\cite{Foi,ZGG20}.

\item The operad $\mscr{N}{\hspace{-0.2cm}}ov$ of Novikov algebras is given by a binary operation and the Novikov relations
$R_{\rm Nov}:=R_{\rm PL} \cup \Biggr\{\treey{\cdlr{ol}\foreach \i/\j in { oll/$\rho(1)$,olr/$\rho(2)$,or/$\rho(3)$} {\node at  (\i) [above] [scale=0.8]{\j};}}
-\treey{\cdlr{ol}\foreach \i/\j in { oll/$\rho(1)$,olr/$\rho(3)$,or/$\rho(2)$} {\node at  (\i) [above] [scale=0.8]{\j};}}\,\Biggr|\, \omega_1,\omega_2\in\Omega, \rho\in S_3 \Biggr\}.$
Thus the \mprs of $R_{\rm Nov}$ are parameterized by    $\sigmab=\big((\sigma_1,\sigma_2,\sigma_3),\sigma_4\big)\in {S_2}^3\times S_2$:
$$R_{\rm Nov, {\rm MT}^{\sigmab}}:= \Biggr\{
\treey{\cdlr{ol}\foreach \i/\j in { oll/$\rho(1)$,olr/$\rho(2)$,or/$\rho(3)$} {\node at  (\i) [above] [scale=0.8]{\j};}\foreach \i/\j in { ol/$\omega_{1}$,o/$\omega_{2}$} {\node at  (\i) [left] [scale=0.8]{\j};}}
-\treey{\cdlr{or}\foreach \i/\j in { ol/$\rho(1)$,orl/$\rho(2)$,orr/$\rho(3)$} {\node at  (\i) [above] [scale=0.8]{\j};}\foreach \i/\j in { o/$\omega_{\sigma_1(1)}$,or/$\omega_{\sigma_1(2)}$} {\node at  (\i) [right] [scale=0.8]{\j};}}
-\treey{\cdlr{ol}\foreach \i/\j in { oll/$\rho(2)$,olr/$\rho(1)$,or/$\rho(3)$} {\node at  (\i) [above] [scale=0.8]{\j};}\foreach \i/\j in { ol/$\omega_{\sigma_2(1)}$,o/$\omega_{\sigma_2(2)}$} {\node at  (\i) [left] [scale=0.8]{\j};}}
+\treey{\cdlr{or}\foreach \i/\j in { ol/$\rho(2)$,orl/$\rho(1)$,orr/$\rho(3)$} {\node at  (\i) [above] [scale=0.8]{\j};}\foreach \i/\j in { o/$\omega_{\sigma_3(1)}$,or/$\omega_{\sigma_3(2)}$} {\node at  (\i) [right] [scale=0.8]{\j};}},\,
\treey{\cdlr{ol}\foreach \i/\j in { oll/$\rho(1)$,olr/$\rho(2)$,or/$\rho(3)$} {\node at  (\i) [above] [scale=0.8]{\j};}\foreach \i/\j in { ol/$\omega_{1}$,o/$\omega_{2}$} {\node at  (\i) [left] [scale=0.8]{\j};}}
-\treey{\cdlr{ol}\foreach \i/\j in { oll/$\rho(1)$,olr/$\rho(3)$,or/$\rho(2)$} {\node at  (\i) [above] [scale=0.8]{\j};}\foreach \i/\j in { ol/$\omega_{\sigma_4(1)}$,o/$\omega_{\sigma_4(2)}$} {\node at  (\i) [left] [scale=0.8]{\j};}}\,\Biggr|\, \omega_1,\omega_2\in\Omega, \rho\in S_3 \Biggr\}.$$
Furthermore, taking the union $R_{{\rm Nov, MT}^\sigma}\cup R_{{\rm Nov, MT}^\tau}$ for
$$\sigmab=\big((\sigma_1,\sigma_2,\sigma_3),\sigma_4\big)=\left(\big(e,(12),(12)\big), (12)\right)\in {S_2}^3\times S_2,$$
and
$${\boldsymbol{ \tau}}=\big((\tau_1,\tau_2,\tau_3),\tau_4\big)=\left(\big(e,e,e\big), (12)\right)\in {S_2}^3\times S_2,$$
we obtain the multi-Novikov relation newly introduced in the context of singular stochastic partial
differential equations\,\mcite{BruD}.
\label{it:matchnov}

\item It can be verified that the matching Rota-Baxter algebra introduced in~\cite{GGL,ZGG20} from Volterra integral equations is also one of our matching structures from a foliation.
\end{enumerate}
\end{example}

The following result shows that a matching operad is a linearly compatible operad.
\begin{proposition}\mlabel{prop:matlin0}
Let $\Omega$ be a nonempty set and $\spp$ be a locally homogeneous operad.
Then for each \mpr $\mrr$ with $\pmc\in \pc{R}$, there is an operad epimorphism
$\lin{\spp}_\Omega \longrightarrow \mat{\spp}_\Omega.$
\end{proposition}
\begin{proof}
By the definition of foliation, we have $\bfk\{\lrr\} \subseteq \bfk\{ \mrr\}$ and so the result follows.
\end{proof}

\begin{theorem}
Let $\Omega$ be a nonempty set. Let $\spp=\mscr{T}(M)/\langle R\rangle$ be a locally homogeneous operad. Then
\vsb
$$\lmat{\Big(\lin{\spp}_\Omega\Big)}_\Omega=\lin{\Big(\lmat{\spp}_\Omega\Big)}_\Omega \,\text{ and }\,\lmat{\Big(\lmat{\spp}_\Omega\Big)}_\Omega =\lmat{\spp}_{\Omega^2}.$$
\end{theorem}

\begin{proof}
We just check the first one, as the second one is similar to the proof of Proposition~\mref{thm:ll}.

Since the generators of both the operads $\lmat{\Big(\lin{\spp}_\Omega\Big)}_\Omega$ and $\lin{\Big(\lmat{\spp}_\Omega\Big)}_\Omega$ are
$\bigoplus_{(\mu,\nu)\in\Omega^2} M_{(\mu,\nu)},$
we only need to prove that for any relation $r(X)=\sum_{1\leq i \leq s} \alpha_i\,(t_i,\dec_i)\in R^{(m)}$, the set of \lev matching relations
\vsb
$$\left\{r_{\fcwp{c}{\mu}}\big(X_\Omega \big)\,\left|\, \fcwp{c}{\mu}\in\mwc{m} \right.\right\}_{\rm LMT}$$
of the linearly compatible relations of $r(X)$ with parameter $\Omega$
coincides with the set of linearly compatible relations
\vsb
$$\left\{f_{\fcwp{b}{\nu}}\big(X_{\Omega^2}\big)\,\left|\,f\in r(X)_{\rm LMT}, \fcwp{b}{\nu}\in\mwc{m} \right.\right\}$$
of the \lev matching relations of $r(X)$ with parameter $\Omega$.
For this purpose, we check
\begin{equation}\notag
\resizebox{\textwidth}{!}
{$
\begin{split}
&\left\{r_{\fcwp{c}{\mu}}\big(X_\Omega \big)\,\left|\, \fcwp{c}{\mu}\in\mwc{m} \right.\right\}_{\rm LMT}\\
=&\bigg\{\sum_{1\leq i\leq s}\alpha_{i}\, (t_i, d_i){\fcwp{c}{\mu}}\,\Big|\, \fcwp{c}{\mu}\in\mwc{m} \bigg\}_{\rm LMT}\\ \smallskip
=&\bigg\{\sum_{1\leq i\leq s}\alpha_{i} \sum_{1\leq k\leq |\lif{\dec_i,{\fcwp{c}{\mu}}}|} ({t_i},\deltal_{i,k})\,\Big|\, \fcwp{c}{\mu}\in\mwc{m}, \deltal_{i,k}\in \lif{\dec_i,{\fcwp{c}{\mu}}} \bigg\}_{\rm LMT}\\
=&\left\{\left.\bigsqcup_{1\leq p \leq\tbinom{m}{(b_\nu)}}\left\{\sum_{1\leq i\leq s}\alpha_{i} \sum_{1\leq k\leq \tbinom{m}{(c_\mu)}} ({t_i},\deltal_{i,k,p})\right\}
\,\right|\,
\begin{split}
\fcwp{c}{\mu},\fcwp{b}{\nu} \in\mwc{m},\\
\deltal_{i,k}\in \lif{\dec_i,{\fcwp{c}{\mu}}}
\end{split}
\right\},
\begin{split}
&\bigsqcup_{1 \leq p\leq\tbinom{m}{(b_\nu)}}\deltal_{i,k,p}= \lif{\deltal_{i,k}, \fcwp{b}{\nu}}, \\
&\pi_\Omega \big(\deltal_{\ell,q,p}(v_{i,j})\big)=\pi_\Omega \big(\deltal_{\ell',q',p}(v_{i,j})\big),~~v_{i,1},\ldots, v_{i,m}\in V(t_i)
\end{split}
\\
& \\
=&\left\{\left.
\bigsqcup_{1\leq p \leq \tbinom{m}{(b_\nu)}}\left\{\sum_{ 1\leq i \leq s }\alpha_i
\sum_{1\leq k \leq \tbinom{m}{(c_\mu)}}(t_i, \deltal_{i,k,p})\right\}\,\right|\,
\begin{split}
&\fcwp{c}{\mu},\fcwp{b}{\nu} \in\mwc{m}\\
&\deltal_{i,k,p}\in\lif{\deltal_{i,p}, (c_\nu)}
\end{split}
\right\},
\begin{split}
&\bigsqcup_{1 \leq p\leq\tbinom{m}{(b_\nu)}}\deltal_{i,p}= \lif{\dec_i, (b_\nu)}\\
&\pi_\Omega \big(\deltal_{\ell,p}(v_{i,j})\big)=\pi_\Omega \big(\deltal_{\ell',p}(v_{i,j})\big),~~v_{i,1},\ldots, v_{i,m}\in V(t_i)
\end{split}\\
& \\
=&\left\{f_{\fcwp{c}{\mu}}\big(X_{\Omega^2}\big)\,\left|\,f(X_\Omega)\in
\left\{\bigsqcup_{1\leq p \leq \tbinom{m}{(b_\nu)}}\sum_{ 1\leq i \leq s }\alpha_i \,\big(t_i, \deltal_{i,p}\big)
\right\}, \fcwp{c}{\mu}\in\mwc{m} \right.\right\},
\begin{split}
&\bigsqcup_{1 \leq p\leq\tbinom{m}{(b_\nu)}}\deltal_{i,p}= \lif{\dec_i, (b_\nu)},\\
&\pi_\Omega \big(\deltal_{q,p}(v_{i,j})\big)=\pi_\Omega \big(\deltal_{q',p}(v_{i,j})\big),~~v_{i,1},\ldots, v_{i,m}\in V(t_i)                                                                                                                                        \end{split}
\\
& \\
=&\left\{f_{\fcwp{c}{\mu}}\big(X_{\Omega^2}\big)\,\left|\,f(X_\Omega)\in r(X)_{\rm LMT}, \fcwp{c}{\mu}\in\mwc{m} \right.\right\},
\end{split}
$
}
\end{equation}
as required.
\end{proof}

\begin{remark}
The above theorem might not hold without the \lev matching compatibility, because the fifth equation in the above proof is not valid for general \mprs.
\end{remark}

\subsection{Total compatibility}\mlabel{ss:totc}
In this subsection, we introduce the third compatibility of an operad, the total compatibility. It is obtained by adding more relations to the matching compatibility, in analogue to the \ufo of polynomials in \S\,\ref{ss:folia}.

Let $R$ be a set of relations of an operad. Write
$$r(X)=\sum_{  (t,\dec)\in T_n^m(X)}\alpha_{t,\dec} \,(t,\dec)\in R^{(m)}.
\vsb
$$
We define
\vsb
$$r_{{\rm TC}, \fc}(X):=\Big\{(t,\delta)-(t,\tau)\,|\, t\in  {\rm Supp}(r(X))\,\text{ and }\, \delta,\tau\in \lif{\dec, \fc}\Big\}\,\text{ for }\,\fc\in \mwc{m},$$
\vsa
$$\trr:=\bigcup_{r(X)\in R, \fc\in \mwc{|r(X)|}}\left\{r_{{\rm TC},\fc}(X) \right\}.$$
For example, consider $\Omega=\{1,2,3\}$,  $c_1=c_2=c_3=1$ and
$r(X)={\treey{\cdlr[0.8]{o}\cdl{ol}\cdr{or}\node at (ol) [scale=0.6]{$\bullet$};\node at (or) [scale=0.6]{$\bullet$};\node at (0.1,-0.1) {$x$};
\node at (-0.4,0.2) {$y$};\node at (0.4,0.2) {$z$};}}$.
We have
$$\left\{(t,\deltal)\,|\,\deltal\in\lif{\dec, \fc}\right\}=\bigg\{\treey{\cdlr[0.8]{o}\cdl{ol}\cdr{or}\node at (ol) [scale=0.6]{$\bullet$};\node at (or) [scale=0.6]{$\bullet$};\node at (0.2,-0.1) {$x_3$};
\node at (-0.5,0.2) {$y_1$};\node at (0.5,0.2) {$z_2$};}, \treey{\cdlr[0.8]{o}\cdl{ol}\cdr{or}\node at (ol) [scale=0.6]{$\bullet$};\node at (or) [scale=0.6]{$\bullet$};\node at (0.2,-0.1) {$x_2$};
\node at (-0.5,0.2) {$y_1$};\node at (0.5,0.2) {$z_3$};}, \treey{\cdlr[0.8]{o}\cdl{ol}\cdr{or}\node at (ol) [scale=0.6]{$\bullet$};\node at (or) [scale=0.6]{$\bullet$};\node at (0.2,-0.1) {$x_1$};
\node at (-0.5,0.2) {$y_3$};\node at (0.5,0.2) {$z_2$};}, \treey{\cdlr[0.8]{o}\cdl{ol}\cdr{or}\node at (ol) [scale=0.6]{$\bullet$};\node at (or) [scale=0.6]{$\bullet$};\node at (0.2,-0.1) {$x_2$};
\node at (-0.5,0.2) {$y_3$};\node at (0.5,0.2) {$z_1$};}, \treey{\cdlr[0.8]{o}\cdl{ol}\cdr{or}\node at (ol) [scale=0.6]{$\bullet$};\node at (or) [scale=0.6]{$\bullet$};\node at (0.2,-0.1) {$x_3$};
\node at (-0.5,0.2) {$y_2$};\node at (0.5,0.2) {$z_1$};}, \treey{\cdlr[0.8]{o}\cdl{ol}\cdr{or}\node at (ol) [scale=0.6]{$\bullet$};\node at (or) [scale=0.6]{$\bullet$};\node at (0.2,-0.1) {$x_1$};
\node at (-0.5,0.2) {$y_2$};\node at (0.5,0.2) {$z_3$};} \bigg\}.$$
So some elements of $r_{{\rm TC}, \fc}(X)$ are
$$\treey{\cdlr[0.8]{o}\cdl{ol}\cdr{or}\node at (ol) [scale=0.6]{$\bullet$};\node at (or) [scale=0.6]{$\bullet$};\node at (0.2,-0.1) {$x_3$};
\node at (-0.5,0.2) {$y_1$};\node at (0.5,0.2) {$z_2$};}-\treey{\cdlr[0.8]{o}\cdl{ol}\cdr{or}\node at (ol) [scale=0.6]{$\bullet$};\node at (or) [scale=0.6]{$\bullet$};\node at (0.2,-0.1) {$x_1$};
\node at (-0.5,0.2) {$y_3$};\node at (0.5,0.2) {$z_2$};},\,
\treey{\cdlr[0.8]{o}\cdl{ol}\cdr{or}\node at (ol) [scale=0.6]{$\bullet$};\node at (or) [scale=0.6]{$\bullet$};\node at (0.2,-0.1) {$x_3$};
\node at (-0.5,0.2) {$y_1$};\node at (0.5,0.2) {$z_2$};}-\treey{\cdlr[0.8]{o}\cdl{ol}\cdr{or}\node at (ol) [scale=0.6]{$\bullet$};\node at (or) [scale=0.6]{$\bullet$};\node at (0.2,-0.1) {$x_2$};
\node at (-0.5,0.2) {$y_1$};\node at (0.5,0.2) {$z_3$};},\,
\treey{\cdlr[0.8]{o}\cdl{ol}\cdr{or}\node at (ol) [scale=0.6]{$\bullet$};\node at (or) [scale=0.6]{$\bullet$};\node at (0.2,-0.1) {$x_3$};
\node at (-0.5,0.2) {$y_1$};\node at (0.5,0.2) {$z_2$};}-\treey{\cdlr[0.8]{o}\cdl{ol}\cdr{or}\node at (ol) [scale=0.6]{$\bullet$};\node at (or) [scale=0.6]{$\bullet$};\node at (0.2,-0.1) {$x_3$};
\node at (-0.5,0.2) {$y_2$};\node at (0.5,0.2) {$z_1$};}.$$

\begin{proposition}\mlabel{prop:meqt}
Let $\Omega$ be a nonempty set. For two matching operads of a locally homogeneous operad $\mscr{P}$,
\vsb
$$\mat{\mscr{P}}_\Omega:=\mscr{T}\Big(\bigoplus\limits_{\omega\in\Omega} M_\omega\Big)\Big/\Big\langle R_{{\rm MT}^\pmc}\Big\rangle\,\text{ and }\, {{\mscr{P}}_\Omega^{{\rm MT}^\pmd}}:=\mscr{T}\Big(\bigoplus\limits_{\omega\in\Omega} M_\omega\Big)\Big/\Big\langle R_{{\rm MT}^\pmd}\Big\rangle,$$
with $\pmc,\pmd\in\pc{R}$.
We have
\vsb
$$\tot{\mscr{P}}_\Omega=\mscr{T}\Big(\bigoplus\limits_{\omega\in\Omega} M_\omega\Big)\Big/\Big\langle R_{{\rm MT}^\pmc}\cup\trr\Big\rangle=\mscr{T}\Big(\bigoplus\limits_{\omega\in\Omega} M_\omega\Big)\Big/\Big\langle R_{{\rm MT}^\pmd}\cup\trr\Big\rangle.$$
\end{proposition}

\begin{proof}
Let $r(X)=\sum_{ 0\leq i \leq s }\alpha_i \big(t_i, \dec_i\big)\in R^{(m)}$ and  $r_{\fc}\big(X_\Omega\big)$ be a \polar of $r(X)$ of type ${\fc}\in\mwc{m}$. By Definition~\mref{defn:partm}, we have
\begin{eqnarray*}
\prc{r(X)}{\vec{\pmc}}
&=&\bigg\{
\alpha_0 \left(t_0, \deltal_{0,j}\right)+\sum_{ 1\leq i \leq s }\alpha_i \left(t_i, \deltal_{i,\pmc_{i}(j)}\right)
\,\Big|\, 1\leq j\leq \sbin\bigg\}\\
\prc{r(X)}{\vec{\pmd}}
&=&\bigg\{
\alpha_0 \left(t_0, \deltal_{0,j}\right)+\sum_{ 1\leq i \leq s }\alpha_i \left(t_i, \deltal_{i,\pmd_{i}(j)}\right)
\,\Big|\, 1\leq j\leq \sbin\bigg\}
\end{eqnarray*}
for $\sigma_{i},\tau_{i}\in S_\sbin$ and $\lif{\dec_i, c}=\{\deltal_{i,1},\deltal_{i,2},\ldots,\deltal_{i,\sbin}\}$ with $0\leq i \leq s$.
By the definition of totally compatible relations $r_{{\rm TC}, \fc}(X),$
the relations $ \big(t_i, \deltal_{i,j}\big)- \big(t_i,  \deltal_{i,j'}\big)$ are in $r_{{\rm TC}, \fc}(X)$
for $1 \leq j,j'\leq\sbin$ and $0\leq i\leq s$.
Thus,
$$\bfk\big\{\prc{r(X)}{\vec{\pmc}}, r_{{\rm TC}, \fc}(X)\big\}=\bfk\big\{\prc{r(X)}{\vec{\pmd}}, r_{{\rm TC}, \fc}(X)\big\},\quad \fc\in\mwc{m}$$
and so $\bfk\Big\{ R_{{\rm MT}^\pmc}\cup\trr\Big\}=\bfk\Big \{ R_{{\rm MT}^\pmd}\cup\trr\Big\}.$
\end{proof}

Now we are ready to introduce the total compatibility.

\begin{definition}
Let $\Omega$ be a nonempty set. Let $\mscr{P} = \stt(M)/\langle R\rangle$ be a locally homogeneous operad and $R_{{\rm MT}^\pmc}$ a matching compatibility with any $\pmc\in\pc{R}$.
We call
\begin{equation*}
 \tot{\mscr{P}}_\Omega :=\mscr{T}\Big(\bigoplus\limits_{\omega\in\Omega} M_\omega\Big)\Big/\Big\langle R_{{\rm MT}^\pmc}\cup\trr\Big\rangle
\end{equation*}
the \textbf{totally compatible operad } of $\mscr{P}$ with parameter $\Omega$.
\mlabel{de:totcomp1}
\end{definition}

By Proposition~\mref{prop:meqt}, the totally compatible operad is independent  of the choice of $R_{{\rm MT}^\pmc}$.
\vsa
\begin{example}
Let $\Omega$ be a set with $|\Omega|=2$.
\begin{enumerate}
\item For the operad $\as$ give in Example~\mref{ex:aspolar}, we have
$\trr=\bigg\{\treey{\cdlr{ol}\node at (oll)[above=0.2] [scale=0.8]{$1$};\node at (olr)[above=0.2] [scale=0.8]{$2$};\node at (or)[above=0.2] [scale=0.8]{$3$};\foreach \i/\j in { o/$\omega_1$,ol/$\omega_2$} {\node at  (\i) [left] [scale=0.8]{\j};}}
-
\treey{\cdlr{ol}\node at (oll)[above=0.2] [scale=0.8]{$1$};\node at (olr)[above=0.2] [scale=0.8]{$2$};\node at (or)[above=0.2] [scale=0.8]{$3$};\foreach \i/\j in { o/$\omega_2$,ol/$\omega_1$} {\node at  (\i) [left] [scale=0.8]{\j};}}
,\treey{\cdlr{or}\node at (ol)[above=0.2] [scale=0.8]{$1$};\node at (orl)[above=0.2] [scale=0.8]{$2$};\node at (orr)[above=0.2] [scale=0.8]{$3$};\foreach \i/\j in { o/$\omega_1$,or/$\omega_2$} {\node at  (\i) [right] [scale=0.8]{\j};}}
-
\treey{\cdlr{or}\node at (ol)[above=0.2] [scale=0.8]{$1$};\node at (orl)[above=0.2] [scale=0.8]{$2$};\node at (orr)[above=0.2] [scale=0.8]{$3$};\foreach \i/\j in { o/$\omega_2$,or/$\omega_1$} {\node at  (\i) [right] [scale=0.8]{\j};}}
\bigg\}$
and the two \mprs of $R$ are
\begin{align*}
R_{{\rm MT}^\sigma}=~\Biggr\{\treey{\cdlr{ol}\node at (oll)[above=0.2] [scale=0.8]{$1$};\node at (olr)[above=0.2] [scale=0.8]{$2$};\node at (or)[above=0.2] [scale=0.8]{$3$};\foreach \i/\j in { o/$\omega_1$,ol/$\omega_2$} {\node at  (\i) [left] [scale=0.8]{\j};}}
-
\treey{\cdlr{or}\node at (ol)[above=0.2] [scale=0.8]{$1$};\node at (orl)[above=0.2] [scale=0.8]{$2$};\node at (orr)[above=0.2] [scale=0.8]{$3$};\foreach \i/\j in { o/$\omega_{\sigma(1)}$,or/$\omega_{\sigma(2)}$} {\node at  (\i) [right] [scale=0.8]{\j};}}\,\Big|\,
\omega_1,\omega_2\in\Omega\Biggr\}, \quad  \sigma\in S_2.
\end{align*}
Thus $\bfk\Big\{ R_{{\rm MT}^{e}}\cup\trr\Big\}=\bfk\Big \{ R_{{\rm MT}^{(12)}}\cup\trr\Big\}$
and so
$$\tot{\as}_\Omega=\mscr{T}\Big(\bigoplus\limits_{\omega\in\Omega} M_\omega\Big)\Big/\Big\langle R_{{\rm MT}^e}\cup\trr\Big\rangle=\mscr{T}\Big(\bigoplus\limits_{\omega\in\Omega} M_\omega\Big)\Big/\Big\langle R_{{\rm MT}^{(12)}}\cup\trr\Big\rangle.$$

\item For the dendriform operad in Example~\mref{eg:dend},
its totally compatible relations  are
$$\trr=\left\{\begin{split}
           (x\circ_{\omega_1} y)\bullet_{\omega_2} z-&(x\circ_{\omega_2} y)\bullet_{\omega_1} z,\\
x\circ_{\omega_1} (y\bullet_{\omega_2} z)-&x\circ_{\omega_2} (y\bullet_{\omega_1} z)
         \end{split}
\,\Biggr|\, \omega_1,\omega_2\in \Omega,
(\circ,\bullet)\in\{(\prec, \prec), (\prec, \succ), (\succ,\succ),(\succ,\prec)\}
\right\}.
$$
\end{enumerate}
\end{example}

We evidently have
\vsb
\begin{proposition}\mlabel{prop:totmat0}
Let $\Omega$ be a nonempty set. Let $\spp$ be a locally homogeneous operad.
Then there is an epimorphism of operads
$\mat{\spp}_\Omega \longrightarrow \tot{\spp}_\Omega.$
\vsc
\end{proposition}

\section{Koszul duality and Manin products of compatible operads}
\mlabel{sec:km}
We first establish the Koszul duality between the linear compatibility and total compatibility.
We then present a self-duality among the matching compatibilities.
We finally show that compatibilities for binary quadratic operads can be achieved by taking the Manin products.
 \vsa
\subsection{Koszul duality between the linear compatibility and total compatibility}
\mlabel{ss:ltdual}

Denote by ${\rm sgn}:=\big\{0, \text{sgn}_1 ,\text{sgn}_2,0,\cdots\big\}$,
where ${\rm sgn}_1$ and ${\rm sgn}_2$ are the sign representations of $S_1$ and $S_2$, respectively.
{Recall that a unary binary quadratic operad is a quadratic operad with generators concentrated in arities 1 and 2.}

\begin{lemma}\mlabel{thm:ubdual}
~\cite{Val08}$($Koszul dual operad of a unary binary quadratic operad$)$
Let $\spp = \spp(M,R)$ be a unary binary quadratic operad, generated by a reduced $\mathbb{S}$-module $M$ that is
finite-dimensional in each arity. Then the Koszul dual operad $\spp^!$ admits a quadratic presentation of the form
\vsa
\begin{equation*}
\spp^!=\spp(M^\vee, R^\perp),
\vsa
\end{equation*}
where $M^\vee = M^* \underset{ {\rm H} }{\otimes} {\rm sgn}$
and $R^\perp$ is the orthogonal subspace of $R$ in $\stt(M^\vee)^{(2)}$.
\end{lemma}
\vsb
Strohmayer~\cite[Proposition 1.14]{St08} proved that $\Big({\lin{\mscr{P}}_\Omega}\Big)^!=~\tot{\Big(\mscr{P}^!\Big)}_\Omega$ and $\Big(~\tot{\mscr{P}}_\Omega\Big)^!=\lin{\Big(\mscr{P}^!\Big)}_\Omega$ for $|\Omega|=2$ and
a binary quadratic symmetric operad $\mscr{P}$. The case of arbitrary nonempty finite set $\Omega$ and unary binary quadratic  nonsymmetric operad $\mscr{P}$ was obtained in~\mcite{ZGG23}.
We generalize it to  arbitrary unary binary quadratic symmetric operads, for any nonempty finite set $\Omega$.

\begin{theorem}\mlabel{thm:dul}
Let $\Omega$ be a nonempty finite set.
Let $\mscr{P}=\mscr{T}(M)/\langle R\rangle $  be a unary binary quadratic symmetric operad. If $M(1)$ and $M(2)$ are finitely dimensional, then
$\big({\lin{\mscr{P}}_\Omega}\big)^!=~\tot{\big(\mscr{P}^!\big)}_\Omega$ and $\big(~\tot{\mscr{P}}_\Omega\big)^!=\lin{\big(\mscr{P}^!\big)}_\Omega$.
\end{theorem}
\vsb
\begin{proof}
We first prove $\big({\lin{\mscr{P}}_\Omega}\big)^!=~\tot{\big(\mscr{P}^!\big)}_\Omega$.
Let $M=\big\{\,0, M(1), M(2), 0,\ldots, 0 ,\ldots\,\big\},$
where $ M(1)$ is spanned by
\vsb
\begin{equation*}
	\biggr\{
\treeyy[scale=0.8]{\cdu{o} \node  at (0,0) [scale=0.6]{$\bullet$};\node at (0.3,0) [scale=0.6]{$P_1$};},
\treeyy[scale=0.8]{\cdu{o} \node  at (0,0) [scale=0.6]{$\bullet$};\node at (0.3,0)[scale=0.6] {$P_2$};},~\cdots~,
\treeyy[scale=0.8]{\cdu{o} \node  at (0,0) [scale=0.6]{$\bullet$};\node at (0.3,0) [scale=0.6]{$P_t$};}
\biggr\}
\end{equation*}
for some $t\in \mathbb{Z}_{\ge 0}$, and $ M(2)$ is spanned by
\vsa
\begin{equation*}
	\left\{
\treey{\cdlr{o}\node  at (0,0.25)[scale=0.6] {$\mu_1$};},
\treey{\cdlr{o}\node  at (0,0.25) [scale=0.6]{$\mu_2$};},~\ldots ~,
\treey{\cdlr{o}\node  at (0,0.25) [scale=0.6]{$\mu_s$};}
\right\}
\end{equation*}
for some $s\in \mathbb{Z}_{\ge 0}$. Write $R:=R(1)\cup R(2)\cup R(3)$, where
\begin{equation}
R(1)= ~\biggr\{r_1^n(P_k,P_\ell):=\sum_{1\leq k,\ell\leq t}\alpha_{k,\ell}^{n}~~
\treeyy{\cdu{o} \node  at (0,-0.2) [scale=0.6]{$\bullet$};\node at (0.3,-0.2) [scale=0.6]{$P_k$};
\node  at (0,0.2) [scale=0.6]{$\bullet$};\node at (0.3,0.2)[scale=0.6] {$P_\ell$};}
 \quad\Bigg|\,{1\leq n\leq m_1}\biggr\},
\mlabel{eq:pr1}
\end{equation}
\vsb
\begin{equation*}
R(2)= ~\biggr\{r_{2}^n(P_k,\mu_i):=\sum_{1\leq i\leq s, \atop 1\leq k\leq t}\kb{k}{i}{1} \treey{\cdlr[0.8]{o} \cdl{ol}\cdr{or}\zhds{o/b}
	\node at (0.25,-0.2) [scale=0.6]{$P_k$};\node at (0,0.2) [scale=0.6]{$\mu_i$};}
+\kb{k}{i}{2}\treey{\cdlr[0.8]{o} \cdl{ol}\cdr{or}
	\node at (ol) [scale=0.6]{$\bullet$};\node at (-0.4,0.1)[scale=0.6] {$P_k$};\node at (0,0.2) [scale=0.6]{$\mu_i$};}
+\kb{k}{i}{3}\treey{\cdlr[0.8]{o} \cdl{ol}\cdr{or}
	\node at (or) [scale=0.6]{$\bullet$};\node at (0.4,0.1) [scale=0.6]{$P_k$};\node at (-0,0.2) [scale=0.6]{$\mu_i$};}
\quad\Bigg|\,{1\leq n\leq m_2}\biggr\},
\end{equation*}
\vsc
\begin{equation*}
R(3)= ~\biggr\{r_{3}^n(\mu_i,\mu_j):=\sum_{1\leq i,j\leq s}
\ka{i}{j}{1}\treey{\cdlr{ol}\node at (oll)[above=0.2] [scale=0.6]{$1$};\node at (olr)[above=0.2] [scale=0.6]{$2$};\node at (or)[above=0.2] [scale=0.6]{$3$};
\node at  (ol)[left=0.2] [scale=0.6]{$\mu_j$};\node at  (o)[left=0.2] [scale=0.6]{$\mu_i$};}
+\ka{i}{j}{2}\treey{\cdlr{ol}\node at (oll)[above=0.2] [scale=0.6]{$2$};\node at (olr)[above=0.2] [scale=0.6]{$3$};\node at (or)[above=0.2] [scale=0.6]{$1$};
\node at  (ol)[left=0.2] [scale=0.6]{$\mu_j$};\node at  (o)[left=0.2] [scale=0.6]{$\mu_i$};}
+\ka{i}{j}{3}\treey{\cdlr{ol}\node at (oll)[above=0.2] [scale=0.6]{$3$};\node at (olr)[above=0.2] [scale=0.6]{$1$};\node at (or)[above=0.2] [scale=0.6]{$2$};
\node at  (ol)[left=0.2] [scale=0.6]{$\mu_j$};\node at  (o)[left=0.2] [scale=0.6]{$\mu_i$};}
\quad \Bigg| \,{1\leq n\leq m_{3}}\biggr\}.
\end{equation*}
By Definitions~\mref{de:comp1}, ~\mref{de:totcomp1} and Lemma~\mref{thm:ubdual}, we have
$$\Big(\lin{\mscr{P}}_\Omega\Big)^!=\mscr{T}\Big((\bigoplus\limits_{\omega\in\Omega} M_\omega)^\vee\Big)\Big/\Big\langle (\lrr)^\perp\Big\rangle,\quad
 \tot{\mscr{\Big(P^!\Big)}}_\Omega =\mscr{T}\Big(\bigoplus\limits_{\omega\in\Omega} M^\vee_\omega\Big)\Big/\Big\langle (R^\perp)_{\rm MT}\cup(R^\perp)_{\rm TC}\Big\rangle.
$$
By $\Big(\bigoplus\limits_{\omega\in\Omega} M_\omega\Big)^\vee\cong \bigoplus\limits_{\omega\in\Omega} M^\vee_\omega$, we can identify $\mscr{T}\Big((\bigoplus\limits_{\omega\in\Omega} M_\omega)^\vee\Big)$ with $\mscr{T}\Big(\bigoplus\limits_{\omega\in\Omega} M^\vee_\omega\Big)$. So we only need to show
\vsb
$$\Big\langle (\lrr)^\perp\Big\rangle=\Big\langle (R^\perp)_{\rm MT}\cup(R^\perp)_{\rm TC}\Big\rangle,$$
which follows from the equality
\vsb
\begin{equation}
\bfk(\lrr)^\perp=\bfk\Big( (R^\perp)_{\rm MT}\cup(R^\perp)_{\rm TC}\Big).
\mlabel{eq:lrd}
\end{equation}
To check Eq.~\meqref{eq:lrd}, on the one hand,
\vsb
\begin{eqnarray*}
\bfk(\lrr)^\perp
&=&\bfk\left\{r_1^n(P_k^\omega,P_\ell^\tau)+r_1^n(P_k^\tau,P_\ell^\omega), r_2^n(P_k^\omega,\mu_i^\tau)+r_2^n(P_k^\tau,\mu_i^\omega)
,r_3^n(\mu_i^\omega,\nu_j^\tau)+r_3^n(\mu_i^\tau,\nu_j^\omega)\right.\\
&&\qquad\,\left.\left|\, \omega,\tau\in\Omega, r^n_p\in R(p), p=1,2,3 \right.\right\}^\perp\\
&=&
\bfk\left\{\left\{r_1^n(P_k^\omega,P_\ell^\tau)+r_1^n(P_k^\tau,P_\ell^\omega)\,\left|\, \omega,\tau\in\Omega, r^n_1\in R(1)\right.\right\}^\perp\right.\\
&&\qquad\cup \left\{r_2^n(P_k^\omega,\mu_i^\tau)+r_2^n(P_k^\tau,\mu_i^\omega)\,\left|\, \omega,\tau\in\Omega, r^n_2\in R(2)\right.\right\}^\perp\\
&&\qquad\cup \left.\left\{r_3^n(\mu_i^\omega,\nu_j^\tau)+r_3^n(\mu_i^\tau,\nu_j^\omega)\,\left|\, \omega,\tau\in\Omega, r^n_3\in R(3)\right.\right\}^\perp
\right\}.
\vsb
\end{eqnarray*}
On the other hand,
\vsa
\begin{eqnarray*}
&&\bfk\Big( (R^\perp)_{\rm MT}\cup(R^\perp)_{\rm TC}\Big)\\
&=& \bfk\bigg( \Big(R(1)^\perp\Big)_{\rm MT}\cup\Big(R(1)^\perp\Big)_{\rm TC}\cup \Big(R(2)^\perp\Big)_{\rm MT}\cup\Big(R(2)^\perp\Big)_{\rm TC}\cup \Big(R(3)^\perp\Big)_{\rm MT}\cup\Big(R(3)^\perp\Big)_{\rm TC}\bigg).
\end{eqnarray*}
So Eq.~\meqref{eq:lrd} is equivalent to
\vsa
\begin{equation}
\begin{split}
\bfk\Big\{r_1^n(P_k^\omega,P_\ell^\tau)+r_1^n(P_k^\tau,P_\ell^\omega)\,\big|\, \omega,\tau\in\Omega, r^n_1\in R(1)\Big\}^\perp
&=\bfk\Big(\big(R(1)^\perp\big)_{\rm MT}\cup\big(R(1)^\perp\big)_{\rm TC}\Big),\\
\bfk \Big\{r_2^n(P_k^\omega,\mu_i^\tau)+r_2^n(P_k^\tau,\mu_i^\omega)\,\big|\, \omega,\tau\in\Omega, r^n_2\in R(2)\Big\}^\perp
&=\bfk\Big(\big(R(2)^\perp\big)_{\rm MT}\cup\big(R(2)^\perp\big)_{\rm TC}\Big),\\
\bfk\Big\{r_3^n(\mu_i^\omega,\nu_j^\tau)+r_3^n(\mu_i^\tau,\nu_j^\omega)\,\big|\, \omega,\tau\in\Omega, r^n_3\in R(3)\Big\}^\perp
&=\bfk\Big(\big(R(3)^\perp\big)_{\rm MT}\cup\big(R(3)^\perp\big)_{\rm TC}\Big).\\
\end{split}
\mlabel{eq:lr123}
\end{equation}
We just check the first one in Eq.~\meqref{eq:lr123}, as the other two are similarly verified.
Writing explicitly,
\vsb
\begin{equation}
R(1)^\perp= ~\biggr\{\sum_{1\leq k,\ell\leq t}\lambda_{k,\ell}^{n}~~
\treeyy{\cdu{o} \node  at (0,-0.2) [scale=0.6]{$\bullet$};\node at (0.3,-0.2) [scale=0.6]{$P_k^\ast$};
\node  at (0,0.2) [scale=0.6]{$\bullet$};\node at (0.3,0.2)[scale=0.6] {$P_\ell^\ast$};}
 \quad\,\Bigg|\,{1\leq n\leq p_1}\biggr\}.
\mlabel{eq:ppr1}
\end{equation}
Then the first one in Eq.~\meqref{eq:lr123} follows from
\vsa
\begin{eqnarray*}
&&\bfk\bigg\{r_1^n(P_k^\omega,P_\ell^\tau)+r_1^n(P_k^\tau,P_\ell^\omega)\,\big|\, \omega,\tau\in\Omega, r^n_1\in R(1)\bigg\}^\perp\\
&=&\bfk\bigg\{\alpha_{k,\ell}^{n}~~
\treeyy{\cdu{o} \node  at (0,-0.2) [scale=0.6]{$\bullet$};\node at (0.3,-0.2) [scale=0.6]{$P_k^\omega$};
\node  at (0,0.2) [scale=0.6]{$\bullet$};\node at (0.3,0.2)[scale=0.6] {$P_\ell^\tau$};}
+\alpha_{k,\ell}^{n}~~
\treeyy{\cdu{o} \node  at (0,-0.2) [scale=0.6]{$\bullet$};\node at (0.3,-0.2) [scale=0.6]{$P_k^\tau$};
\node  at (0,0.2) [scale=0.6]{$\bullet$};\node at (0.3,0.2)[scale=0.6] {$P_\ell^\omega$};}\,\Bigg|\, \omega,\tau\in\Omega\bigg\}^\perp
\quad\quad (\text{by Eq.~\meqref{eq:pr1}})\\
&=&\bfk\biggr\{\sum_{1\leq k,\ell\leq t}\lambda_{k,\ell}^{n}~~
\treeyy{\cdu{o} \node  at (0,-0.2) [scale=0.6]{$\bullet$};\node at (0.3,-0.2) [scale=0.6]{$(P_k^\omega)^\ast$};
\node  at (0,0.2) [scale=0.6]{$\bullet$};\node at (0.3,0.2)[scale=0.6] {$(P_\ell^\tau)^\ast$};}
,\,\, \treeyy{\cdu{o} \node  at (0,-0.2) [scale=0.6]{$\bullet$};\node at (0.3,-0.2) [scale=0.6]{$(P_k^\omega)^\ast$};
\node  at (0,0.2) [scale=0.6]{$\bullet$};\node at (0.3,0.2)[scale=0.6] {$(P_\ell^\tau)^\ast$};}-
\treeyy{\cdu{o} \node  at (0,-0.2) [scale=0.6]{$\bullet$};\node at (0.3,-0.2) [scale=0.6]{$(P_k^\tau)^\ast$};
\node  at (0,0.2) [scale=0.6]{$\bullet$};\node at (0.3,0.2)[scale=0.6] {$(P_\ell^\omega)^\ast$};} \quad\,\Bigg|\,\omega,\tau\in\Omega, {1\leq n\leq p_1}\biggr\}
\quad\quad (\text{by Eq.~\meqref{eq:ppr1}})\\
&=&\bfk\Big(\big(R(1)^\perp\big)_{\rm MT}\cup\big(R(1)^\perp\big)_{\rm TC}\Big).
\vsa
\end{eqnarray*}
This completes the proof of the first equation $\big({\lin{\mscr{P}}_\Omega}\big)^!=~\tot{\big(\mscr{P}^!\big)}_\Omega$ in the theorem.
Further by $(\mscr{P}^!)^!=\mscr{P}$, we obtain
$${\lin{\mscr{P}}_\Omega}=~\Big(\big({\lin{\mscr{P}}_\Omega}\big)^!\Big)^!=~\Big(\tot{(\mscr{P}^!)}_\Omega\Big)^!\,\text{ and so }\,
\lin{\Big(\mscr{P}^!\Big)}_\Omega = \Big(~\tot{\mscr{P}}_\Omega\Big)^!,$$
proving the second equation in the theorem.
\end{proof}
\vsa
Specializing to binary quadratic symmetric operads, we obtain the following result of Strohmayer.

\begin{corollary}~\mcite{St08}
Let $\Omega$ be a set with $|\Omega|=2$. Let $\mscr{P}=\mscr{T}(M)/\langle R\rangle $  be a binary quadratic symmetric operad. If  $M(2)$ is finite dimensional, then
$\big({\lin{\mscr{P}}_\Omega}\big)^!=~\tot{\big(\mscr{P}^!\big)}_\Omega$ and $\big(~\tot{\mscr{P}}_\Omega\big)^!=\lin{\big(\mscr{P}^!\big)}_\Omega$.
\mlabel{thm:dulu}
\end{corollary}
\vsc
\subsection{Self duality among the matching compatibilities}
\mlabel{ss:matdual}
We now prove that the duality of operads leads to a duality among the matching compatible operads. In particular, the \lev matching compatibility is self dual.

Notice that a matching relation of
\vsa
$$r(X)=\sum_{0\leq i\leq s} \alpha_i\,(t_i,\dec_i)=\sum_{0\leq i\leq s} \alpha_i\,t_i(x^i,y^i)\in R \subseteq \mscr{T}(M)^{(2)}\, \text{ for } x^i,y^i\in X, 0\neq \alpha_i\in \bfk$$
 is of the form
 \vsb
\begin{equation}\mlabel{eq:mtrq}
 \alpha_0 t_0\left(x^0_{\omega_{1}},y^0_{\omega_{2}}\right)+ \sum\limits_{ 1\leq i \leq s }\alpha_i\, t_i\left(x^i_{\omega_{\sigma_i(1)}},y^i_{\omega_{\sigma_i(2)}}\right)
 \vsa
\end{equation}
for some  $\omega_1,\omega_2\in\Omega$, $\fc\in\mwc{2}$ and  $\vec{\sigma}=(\sigma_1,\ldots,\sigma_{s})\in {S_\sbin}^{s}$, where
\begin{align*}
\left\{
  \begin{array}{ll}
\sigma_i\in S_2,& \,\text{if }\fc=(\fc_{\omega_1},\fc_{\omega_2})=(1,1),\\
\sigma_i=e\in S_1,& \,\text{if }\fc=(\fc_{\omega_1},\fc_{\omega_2})=(2,0)\,\text{ or }\,(0,2),
  \end{array}
\right.\quad 1\leq i\leq s.
\end{align*}
We will use the linear dual notation $(V, W)^\ast:= (V^\ast , W^\ast)$ for a pair $(V, W)$ of vector spaces.

\begin{theorem}\mlabel{thm:mdul}
Let $\Omega$ be a nonempty finite set.
Let $\mscr{P}=\mscr{T}(M)/\langle R\rangle $ be a finitely generated unary binary quadratic operad and $\mscr{P^!}=\mscr{T}(M^\vee)/\langle R^\perp\rangle $ the Koszul dual operad of $\mscr{P}$.
\begin{enumerate}
\item For any $\pmc\in\pc{R}$, there exists $\pmd\in\pc{R^\perp}$ such that
$\Big({\mat{\mscr{P}}_\Omega}\Big)^!=~{\Big(\mscr{P}^!\Big)}^{{\rm MT}^\pmd}_\Omega.$
\mlabel{item:mkm}
\vsc
\item For the \lev matching compatible operad ${\lmat{\mscr{P}}_\Omega}$, we have  $\Big({\lmat{\mscr{P}}_\Omega}\Big)^!=~\lmat{\Big(\mscr{P}^!\Big)}_\Omega$.
\mlabel{item:lmklm}

\item If the operad $\mscr{P}$ is self-dual (that is $\mscr{P}\cong \mscr{P}^!$), then ${\mat{\mscr{P}}_\Omega}$ is also, for all $\pmc\in\pc{R}$.
\mlabel{item:lmklmd}
\end{enumerate}
\end{theorem}
\begin{proof}
\mref{item:mkm}
Write
\vsc
$$\pmc=\Big\{\vec{\sigmab}_{\fc,r}\Big\}_{\fc\in \mwc{2}, r\in R}=\Big\{(\sigma_{\fc,r,1},\ldots, \sigma_{\fc,r,|{\rm Supp(r)}|-1})\Big\}_{\fc\in \mwc{2}, r\in R}\in\pc{R}.$$
with $\sigma_{\fc,r,1},\ldots, \sigma_{\fc,r,|{\rm Supp(r)}|-1}\in S_{\tbinom{2}{\fc }}$. Define
\vsb
$$\pmd=\Big\{\vec{\taub}_{\fc,r'}\Big\}_{\fc\in \mwc{2}, r'\in R^\perp}=\Big\{(\tau_{\fc,r',1},\ldots, \tau_{\fc,r',|{\rm Supp(r')}|-1})\Big\}_{\fc\in \mwc{2}, r'\in R^\perp}\in\pc{R^\perp}$$
by
\vsc
\begin{equation}
t\left(x^\ast_{\omega_{\pmd_{\fc,r'}(1)}},y^\ast_{\omega_{\pmd_{\fc,r'}(2)}}\right)
=t\left(x_{{\omega_{\pmc_{\fc,r}(1)}}},y_{\omega_{\pmc_{\fc,r}(2)}}\right)^\ast
=t\left(x_{{\omega_{\pmc_{\fc,r}(1)}}}^\ast,y_{\omega_{\pmc_{\fc,r}(2)}}^\ast\right),\quad \sigma_{\fc,r},\tau_{\fc,r'}\in S_{\tbinom{2}{\fc }}
\mlabel{eq:dualp}
\vsa
\end{equation}
for $t(x,y)\in{\rm Supp}(r)$ and $t(x^\ast,y^\ast)\in{\rm Supp}(r')$.

By Definition~\mref{de:mat} and Lemma~\mref{thm:ubdual}, we obtain
\vsb
$$
\Big(\mat{\mscr{P}}_\Omega\Big)^!=\mscr{T}\Big(\big(\bigoplus\limits_{\omega\in\Omega} M_\omega\big)^\vee\Big)\Big/\Big\langle (R_{{\rm MT}^{\pmc}})^\perp\Big\rangle,\quad
{\mscr{\Big(P^!\Big)}}_\Omega^{{\rm MT}^{\pmd}} =\mscr{T}\Big(\bigoplus\limits_{\omega\in\Omega} M^\vee_\omega\Big)\Big/\Big\langle (R^\perp)_{{\rm MT}^{\pmd}}\Big\rangle.
\vsb
$$
Under the identification $\Big(\bigoplus\limits_{\omega\in\Omega} M_\omega\Big)^\vee\cong \bigoplus\limits_{\omega\in\Omega} M^\vee_\omega$, we can identify $\mscr{T}\Big((\bigoplus\limits_{\omega\in\Omega} M_\omega)^\vee\Big)$ with $\mscr{T}\Big(\bigoplus\limits_{\omega\in\Omega} M^\vee_\omega\Big)$. Consequently, we only need to check
\vsa
\begin{equation}
\Big\{\big((R_{{\rm MT}^{\pmc}})(i)\big)^\perp\Big\}=\Big\{\big(R^\perp(i)\big)_{{\rm MT}^{\pmd}}\Big\},\quad  i=1,2,3.
\mlabel{eq:mmrdd}
\vsa
\end{equation}
We just verify Eq.~\meqref{eq:mmrdd} for $i=1$, as the other two cases are similarly verified.
We first write
\vsa
$$R(1)=\bigg\{r_n:=\sum_{0\leq i\leq s_n} \alpha_{i,n}\,(t_i,\dec_i)=\sum_{0\leq i\leq s_n} \alpha_{i,n}\,t_i(x^i,y^i)\,\Big|\, x^i,y^i\in X(1), 0\neq \alpha_{i,n}\in \bfk, 1\leq n\leq p\bigg\}
\vsc
$$
and
\vsa
$$R^\perp(1)=\bigg\{r_n':=\sum_{0\leq i\leq s_n} \beta_{i,n}\,(t_i,\dec_i^\ast)=\sum_{0\leq i\leq s_n} \beta_{i,n}\,t_i\big((x^i)^\ast,(y^i)^\ast\big)\,\Big|\, (x^i)^\ast,(y^i)^\ast\in X(1)^\ast, 0\neq \beta_{i,n}\in \bfk, 1\leq n\leq q\bigg\}.
\vsb
$$
On the one hand, by Eq.~\meqref{eq:mtrq} we have
\vsa
\begin{equation*}
(R_{{\rm MT}^{\pmc}})(1)= ~
\bigg\{ \alpha_{0,n} t_0\left(x^0_{\omega_{1}},y^0_{\omega_{2}}\right)+ \sum\limits_{ 1\leq i \leq s_n }\alpha_{i,n}\, t_i\left(x^i_{\omega_{\sigma_i(1)}},y^i_{\omega_{\sigma_i(2)}}\right)
\,\Big|\,\omega_1,\omega_2\in\Omega, \sigma_i\in S_2
\bigg\}
\mlabel{eq:ppmm}
\vsb
\end{equation*}
which yields
\vsb
\begin{equation*}
\left((R_{{\rm MT}^{\pmc}})(1)\right)^\perp= ~
\bigg\{ \beta_{0,n} t_0\left((x^0_{\omega_{1}})^\ast,(y^0_{\omega_{2}})^\ast\right)+ \sum\limits_{ 1\leq i \leq s_n }\beta_{i,n}\, t_i\left((x^i_{\omega_{\sigma_i(1)}})^\ast,(y^i_{\omega_{\sigma_i(2)}})^\ast\right)
\,\Big|\,\omega_1,\omega_2\in\Omega,\sigma_i\in S_2
\bigg\}.
\vsa
\end{equation*}
On the other hand, by Eqs.~\meqref{eq:mtrq} and~\meqref{eq:dualp}, we have
\vsb
\begin{align*}
(R^\perp(1))_{{\rm MT}^\pmd}
=&~\bigg\{ \beta_{0,n} t_0\left((x^0_{\omega_{1}})^\ast,(y^0_{\omega_{2}})^\ast\right)+ \sum\limits_{ 1\leq i \leq s_n }\beta_{i,n}\, t_i\left((x^i_{\omega_{\tau_i(1)}})^\ast,(y^i_{\omega_{\tau_i(2)}})^\ast\right)
\,\Big|\,\omega_1,\omega_2\in\Omega,\tau_i\in S_2
\bigg\}\\
=&~\bigg\{ \beta_{0,n} t_0\left((x^0_{\omega_{1}})^\ast,(y^0_{\omega_{2}})^\ast\right)+ \sum\limits_{ 1\leq i \leq s_n }\beta_{i,n}\, t_i\left((x^i_{\omega_{\sigma_i(1)}})^\ast,(y^i_{\omega_{\sigma_i(2)}})^\ast\right)
\,\Big|\,\omega_1,\omega_2\in\Omega,\sigma_i\in S_2
\bigg\}\\
=&~\left((R_{{\rm MT}^{\pmc}})(1)\right)^\perp.
\vsb
\end{align*}
Therefore $\big(\mrr(1)\big)^\perp=(R^\perp(1))^{{\rm MT}^\pmd}$ and so $\big(\mrr\big)^\perp=(R^\perp)_{{\rm MT}^\pmd}$
.
\\
\noindent
\mref{item:lmklm}
It is a special case of the proof of Item~\mref{item:mkm} for
\vsb
$$t\left(x^\ast_{\omega_{\pmd_{\fc,r'}(1)}},y^\ast_{\omega_{\pmd_{\fc,r'}(2)}}\right)
=t\left(x_{{\omega_{\pmc_{\fc,r}(1)}}},y_{\omega_{\pmc_{\fc,r}(2)}}\right)^\ast=t(x^\ast_{\omega_1},y^\ast_{\omega_2}).
\vsb$$
\\
\noindent
\mref{item:lmklmd}  For $\pmc\in\pc{R}$, {let $\pmd\in \pc{R^\perp}$ be as } defined by Eq.~\meqref{eq:dualp}.
Identifying $R$ with $R^\perp$, then $\pmc=\pmd$.
By Item~\mref{item:mkm}, we have
$\Big({\mat{\mscr{P}}_\Omega}\Big)^!=~{\Big(\mscr{P}^!\Big)}^{{\rm MT}^\pmd}_\Omega\cong ~{\mscr{P}}_\Omega^{{\rm MT}^\pmd}=~{\mscr{P}}_\Omega^{{\rm MT}^\pmc},
$
completing the proof.
\end{proof}
\vsc
\subsection{Compatible operads and Manin products}
\mlabel{ss:manin}
We now show that the compatibilities of arbitrary finitely generated binary quadratic operad can be obtained by the Manin product of that operad with some specific compatible structures.

To fix notations, we start with the definition of the Manin products of a pair of operads as can be found in~\mcite{LV, Val08}.
Let $\spp=\stt(M)/\langle R\rangle$ and $\sqq=\stt(N)/\langle S\rangle$ be finitely generated binary quadratic operads. Write
\begin{equation*}
R= ~\Biggr\{\sum
\kaa{i}{j}{1}\treey{\cdlr{ol}\node at (oll)[above=0.2] [scale=0.8]{$1$};\node at (olr)[above=0.2] [scale=0.8]{$2$};\node at (or)[above=0.2] [scale=0.8]{$3$};
\node at  (ol)[left=0.2] [scale=0.8]{$\mu_i$};\node at  (o)[left=0.2] [scale=0.8]{$\mu_j$};}
+\kaa{i}{j}{2}\treey{\cdlr{ol}\node at (oll)[above=0.2] [scale=0.8]{$2$};\node at (olr)[above=0.2] [scale=0.8]{$3$};\node at (or)[above=0.2] [scale=0.8]{$1$};
\node at  (ol)[left=0.2] [scale=0.8]{$\mu_i$};\node at  (o)[left=0.2] [scale=0.8]{$\mu_j$};}
+\kaa{i}{j}{3}\treey{\cdlr{ol}\node at (oll)[above=0.2] [scale=0.8]{$3$};\node at (olr)[above=0.2] [scale=0.8]{$1$};\node at (or)[above=0.2] [scale=0.8]{$2$};
\node at  (ol)[left=0.2] [scale=0.8]{$\mu_i$};\node at  (o)[left=0.2] [scale=0.8]{$\mu_j$};}
\quad \Bigg| \, \mu_i,\mu_j\in M(2)=M ,\kaa{i}{j}{1}, \kaa{i}{j}{2},\kaa{i}{j}{3}\in\bfk \Biggr\}
\vsb
\end{equation*}
and
\vsb
\begin{equation*}
S= ~\Biggr\{\sum
\ka{k}{\ell}{1}\treey{\cdlr{ol}\node at (oll)[above=0.2] [scale=0.8]{$1$};\node at (olr)[above=0.2] [scale=0.8]{$2$};\node at (or)[above=0.2] [scale=0.8]{$3$};
\node at  (ol)[left=0.2] [scale=0.8]{$\nu_k$};\node at  (o)[left=0.2] [scale=0.8]{$\nu_\ell$};}
+\ka{k}{\ell}{2}\treey{\cdlr{ol}\node at (oll)[above=0.2] [scale=0.8]{$2$};\node at (olr)[above=0.2] [scale=0.8]{$3$};\node at (or)[above=0.2] [scale=0.8]{$1$};
\node at  (ol)[left=0.2] [scale=0.8]{$\nu_k$};\node at  (o)[left=0.2] [scale=0.8]{$\nu_\ell$};}
+\ka{k}{\ell}{3}\treey{\cdlr{ol}\node at (oll)[above=0.2] [scale=0.8]{$3$};\node at (olr)[above=0.2] [scale=0.8]{$1$};\node at (or)[above=0.2] [scale=0.8]{$2$};
\node at  (ol)[left=0.2] [scale=0.8]{$\nu_k$};\node at  (o)[left=0.2] [scale=0.8]{$\nu_\ell$};}
\quad \Bigg| \, \nu_k,\nu_\ell\in N(2)=N, \ka{k}{\ell}{1}, \ka{k}{\ell}{2},\ka{k}{\ell}{3}\in\bfk \Biggr\}.
\vsa
\end{equation*}
Define
\vsb
\begin{equation} \notag
R~\bigdot ~S:= ~\Biggr\{\sum
\kaa{i}{j}{1}\ka{k}{\ell}{1}\treey{\cdlr{ol}\node at (oll)[above=0.2] [scale=0.8]{$1$};\node at (olr)[above=0.2] [scale=0.8]{$2$};\node at (or)[above=0.2] [scale=0.8]{$3$};
\node at  (ol)[left=0.8] [scale=0.8]{$\mu_i\otimes \nu_k$};\node at  (o)[left=0.2] [scale=0.8]{$\mu_j\otimes \nu_\ell$};}
+\kaa{i}{j}{2}\ka{k}{\ell}{2}\treey{\cdlr{ol}\node at (oll)[above=0.2] [scale=0.8]{$2$};\node at (olr)[above=0.2] [scale=0.8]{$3$};\node at (or)[above=0.2] [scale=0.8]{$1$};
\node at  (ol)[left=0.2] [scale=0.8]{$\mu_i\otimes \nu_k$};\node at  (o)[left=0.2] [scale=0.8]{$\mu_j\otimes \nu_\ell$};}
+\kaa{i}{j}{3}\ka{k}{\ell}{3}\treey{\cdlr{ol}\node at (oll)[above=0.2] [scale=0.8]{$3$};\node at (olr)[above=0.2] [scale=0.8]{$1$};\node at (or)[above=0.2] [scale=0.8]{$2$};
\node at  (ol)[left=0.2] [scale=0.8]{$\mu_i\otimes \nu_k$};\node at  (o)[left=0.2] [scale=0.8]{$\mu_j\otimes \nu_\ell$};}
\quad \Bigg| \,\mu_i\otimes \nu_k,\mu_j\otimes \nu_\ell\in M(2)\otimes N(2)\Biggr\},
\mlabel{eq:mb}
\end{equation}
and
\vsb
$$R ~\bigcir~S:= \Phi^{-1}\Big(R\otimes \stt\big(N\big)(3)+ \stt\big(M\big)(3)\otimes S\Big),$$
where the map $\Phi:\stt\big(M\otimes N\big)(3)\to\stt\big(M\big)(3)\otimes \stt\big(N\big)(3)$ is given by
$$ \treey{\cdlr{ol}\node at (oll)[above=0.2] [scale=0.8]{$a$};\node at (olr)[above=0.2] [scale=0.8]{$b$};\node at (or)[above=0.2] [scale=0.8]{$c$};
\node at  (ol)[left=0.8] [scale=0.8]{$\mu_i\otimes \nu_k$};\node at  (o)[left=0.2] [scale=0.8]{$\mu_j\otimes \nu_\ell$};}
\mapsto
\treey{\cdlr{ol}\node at (oll)[above=0.2] [scale=0.8]{$a$};\node at (olr)[above=0.2] [scale=0.8]{$b$};\node at (or)[above=0.2] [scale=0.8]{$c$};
\node at  (ol)[left=0.8] [scale=0.8]{$\mu_i$};\node at  (o)[left=0.2] [scale=0.8]{$\mu_j$};}
\otimes
\treey{\cdlr{ol}\node at (oll)[above=0.2] [scale=0.8]{$a$};\node at (olr)[above=0.2] [scale=0.8]{$b$};\node at (or)[above=0.2] [scale=0.8]{$c$};
\node at  (ol)[left=0.8] [scale=0.8]{$ \nu_k$};\node at  (o)[left=0.2] [scale=0.8]{$ \nu_\ell$};}\,\text{ for }\, (a,b,c)\in\{(1,2,3), (2,3,1), (3,1,2)\}.$$
{The above map $\Phi$ essentially comes from the map $M^{\otimes 2}\otimes N^{\otimes 2}\to (M\otimes N)^{\otimes 2}$.}
\begin{definition}\mcite{LV,Val08}
Let $\spp=\stt(M)/\langle R\rangle$ and $\sqq=\stt(N)/\langle S\rangle$ be finitely generated binary quadratic operads.
\begin{enumerate}
\item The {\bf Manin black product } of $\spp$ and $\sqq$ is the operad
$\spp\bigdot\sqq:= \stt(M\otimes N) /\langle R~\bigdot ~S\rangle.$
\mlabel{it:mbp}

\item The {\bf Manin white product } of $\spp$ and $\sqq$ is the operad
$\spp\bigcir\sqq:= \stt(M\otimes N) /\langle R~\bigcir ~S\rangle.$
\mlabel{it:mwp}
\end{enumerate}
\mlabel{de:mp}
\end{definition}
\vsa
It is well known that $(\spp\bigdot\sqq)^!=\spp^!\bigcir\sqq^!$ for finitely generated binary quadratic operads $\spp$ and $\sqq$. Denote by $\lie$ (resp. $\com$) the operad of Lie algebras (resp. commutative algebras).

\begin{proposition}\mlabel{prop:maninbl}
Let $\Omega$ be a nonempty finite set. Let $\spp$ be a finitely generated binary quadratic operad.
Then
\vsb
$$\lin{\spp}_\Omega\cong \lin{\lie}_\Omega\bigdot\spp\,\text{ and }\, \tot{\spp}_\Omega\cong \tot{\com}_\Omega \bigcir \spp.$$
\end{proposition}

\begin{proof}
It suffices to prove $\lin{\spp}_\Omega\cong \lin{\lie}_\Omega\bigdot\spp$, as then $\tot{\spp}_\Omega\cong \tot{\com}_\Omega \bigcir \spp$ follows from the facts $(\spp\bigdot\sqq)^!=\spp^!\bigcir\sqq^!$, $\lie^!=\com$ and Theorem~\mref{thm:dul}.

First, the operad $\lie$ is presented as $\lie=\stt(\lieg)/\langle\lier \rangle$, where
$$\lieg=\Big\{0,0,\bfk \treey{\cdlr{o}},\cdots\Big\}, \quad \lier=~\Biggr\{\treey{\cdlr{ol}\node at (oll)[above=0.2] [scale=0.8]{$1$};\node at (olr)[above=0.2] [scale=0.8]{$2$};\node at (or)[above=0.2] [scale=0.8]{$3$};}
+\treey{\cdlr{ol}\node at (oll)[above=0.2] [scale=0.8]{$2$};\node at (olr)[above=0.2] [scale=0.8]{$3$};\node at (or)[above=0.2] [scale=0.8]{$1$};}
+\treey{\cdlr{ol}\node at (oll)[above=0.2] [scale=0.8]{$3$};\node at (olr)[above=0.2] [scale=0.8]{$1$};\node at (or)[above=0.2] [scale=0.8]{$2$};} \Biggr\}.$$
\vsc
By a direct calculation, we obtain
$$\lrrs{\lier}=\Biggr\{\treey{\cdlr{ol}\node at (oll)[above=0.2] [scale=0.8]{$1$};\node at (olr)[above=0.2] [scale=0.8]{$2$};\node at (or)[above=0.2] [scale=0.8]{$3$};
\node at  (ol)[left=0.2] [scale=0.8]{$\omega$};\node at  (o)[left=0.2] [scale=0.8]{$\omega$};}
+\treey{\cdlr{ol}\node at (oll)[above=0.2] [scale=0.8]{$2$};\node at (olr)[above=0.2] [scale=0.8]{$3$};\node at (or)[above=0.2] [scale=0.8]{$1$};
\node at  (ol)[left=0.2] [scale=0.8]{$\omega$};\node at  (o)[left=0.2] [scale=0.8]{$\omega$};}
+\treey{\cdlr{ol}\node at (oll)[above=0.2] [scale=0.8]{$3$};\node at (olr)[above=0.2] [scale=0.8]{$1$};\node at (or)[above=0.2] [scale=0.8]{$2$};
\node at  (ol)[left=0.2] [scale=0.8]{$\omega$};\node at  (o)[left=0.2] [scale=0.8]{$\omega$};}
,\quad
\treey{\cdlr{ol}\node at (oll)[above=0.2] [scale=0.8]{$1$};\node at (olr)[above=0.2] [scale=0.8]{$2$};\node at (or)[above=0.2] [scale=0.8]{$3$};
\node at  (ol)[left=0.2] [scale=0.8]{$\omega$};\node at  (o)[left=0.2] [scale=0.8]{$\tau$};}
+\treey{\cdlr{ol}\node at (oll)[above=0.2] [scale=0.8]{$2$};\node at (olr)[above=0.2] [scale=0.8]{$3$};\node at (or)[above=0.2] [scale=0.8]{$1$};
\node at  (ol)[left=0.2] [scale=0.8]{$\omega$};\node at  (o)[left=0.2] [scale=0.8]{$\tau$};}
+\treey{\cdlr{ol}\node at (oll)[above=0.2] [scale=0.8]{$3$};\node at (olr)[above=0.2] [scale=0.8]{$1$};\node at (or)[above=0.2] [scale=0.8]{$2$};
\node at  (ol)[left=0.2] [scale=0.8]{$\omega$};\node at  (o)[left=0.2] [scale=0.8]{$\tau$};}
+\treey{\cdlr{ol}\node at (oll)[above=0.2] [scale=0.8]{$1$};\node at (olr)[above=0.2] [scale=0.8]{$2$};\node at (or)[above=0.2] [scale=0.8]{$3$};
\node at  (ol)[left=0.2] [scale=0.8]{$\tau$};\node at  (o)[left=0.2] [scale=0.8]{$\omega$};}
+\treey{\cdlr{ol}\node at (oll)[above=0.2] [scale=0.8]{$2$};\node at (olr)[above=0.2] [scale=0.8]{$3$};\node at (or)[above=0.2] [scale=0.8]{$1$};
\node at  (ol)[left=0.2] [scale=0.8]{$\tau$};\node at  (o)[left=0.2] [scale=0.8]{$\omega$};}
+\treey{\cdlr{ol}\node at (oll)[above=0.2] [scale=0.8]{$3$};\node at (olr)[above=0.2] [scale=0.8]{$1$};\node at (or)[above=0.2] [scale=0.8]{$2$};
\node at  (ol)[left=0.2] [scale=0.8]{$\tau$};\node at  (o)[left=0.2] [scale=0.8]{$\omega$};}
\,\Bigg|\,\omega,\tau\in\Omega \Biggr\}.
\vsa
$$
Let
\vsb
\begin{equation}
M(2)=~\bfk	\Big\{
\treey{\cdlr{o}\node  at (0,0.25)[scale=0.8] {$\mu_1$};},
\treey{\cdlr{o}\node  at (0,0.25) [scale=0.8]{$\mu_2$};},~\ldots ~,
\treey{\cdlr{o}\node  at (0,0.25) [scale=0.8]{$\mu_s$};}
\Big\},\,1<s<\infty,
\mlabel{eq:m2}
\vsa
\end{equation}
and
\vsb
\begin{equation}
R= ~\Biggr\{\sum
\kaa{i}{j}{1}\treey{\cdlr{ol}\node at (oll)[above=0.2] [scale=0.8]{$1$};\node at (olr)[above=0.2] [scale=0.8]{$2$};\node at (or)[above=0.2] [scale=0.8]{$3$};
\node at  (ol)[left=0.2] [scale=0.8]{$\mu_i$};\node at  (o)[left=0.2] [scale=0.8]{$\mu_j$};}
+\kaa{i}{j}{2}\treey{\cdlr{ol}\node at (oll)[above=0.2] [scale=0.8]{$2$};\node at (olr)[above=0.2] [scale=0.8]{$3$};\node at (or)[above=0.2] [scale=0.8]{$1$};
\node at  (ol)[left=0.2] [scale=0.8]{$\mu_i$};\node at  (o)[left=0.2] [scale=0.8]{$\mu_j$};}
+\kaa{i}{j}{3}\treey{\cdlr{ol}\node at (oll)[above=0.2] [scale=0.8]{$3$};\node at (olr)[above=0.2] [scale=0.8]{$1$};\node at (or)[above=0.2] [scale=0.8]{$2$};
\node at  (ol)[left=0.2] [scale=0.8]{$\mu_i$};\node at  (o)[left=0.2] [scale=0.8]{$\mu_j$};}
\quad \Bigg| \, \mu_i,\mu_j\in M(2)=M ,\kaa{i}{j}{1}, \kaa{i}{j}{2},\kaa{i}{j}{3}\in\bfk \Biggr\}.
\mlabel{eq:relar}
\vsa
\end{equation}
By Definition~\mref{de:mp}\mref{it:mbp}, we have
\vsb
$$\lin{\lie}_\Omega\bigdot\spp=\stt\bigg(\Big(\bigoplus_{\omega\in\Omega}\lieg_{,\omega}\Big)\otimes M\bigg)\Big/\bigg\langle \lrrs{\lier}~\bigdot ~R\bigg\rangle,
\vsc$$
where
$$\Big(\bigoplus_{\omega\in\Omega}\lieg_{,\omega}\Big)\otimes M
=\Big(\bigoplus_{\omega\in\Omega}\treey{\cdlr{o}\node at  (0,0.25) [scale=0.8]{$\omega$};}\Big)\otimes
\Big(\bigoplus_{1\leq i \leq s}\treey{\cdlr{o}\node at  (0,0.25) [scale=0.8]{$\mu_i$};}\Big)
=\bigoplus_{\omega\in\Omega, 1\leq i \leq s}\treey{\cdlr{o}\node at  (o)[left=0.8] [scale=0.8]{$\omega\otimes \mu_i$};}
\cong \bigoplus_{\omega\in\Omega, 1\leq i \leq s}\treey{\cdlr{o}\node at  (0,0.25) [scale=0.8]{$ \mu_i^\omega$};}
=\bigoplus_{\omega\in\Omega}M_\omega,
\vsb
$$
and
\vsb
{\small
\begin{eqnarray*}
\lrrs{\lier}~\bigdot~ R
&=&\Biggr\{\treey{\cdlr{ol}\node at (oll)[above=0.2] [scale=0.8]{$1$};\node at (olr)[above=0.2] [scale=0.8]{$2$};\node at (or)[above=0.2] [scale=0.8]{$3$};
\node at  (ol)[left=0.2] [scale=0.8]{$\omega$};\node at  (o)[left=0.2] [scale=0.8]{$\tau$};}
+\treey{\cdlr{ol}\node at (oll)[above=0.2] [scale=0.8]{$2$};\node at (olr)[above=0.2] [scale=0.8]{$3$};\node at (or)[above=0.2] [scale=0.8]{$1$};
\node at  (ol)[left=0.2] [scale=0.8]{$\omega$};\node at  (o)[left=0.2] [scale=0.8]{$\tau$};}
+\treey{\cdlr{ol}\node at (oll)[above=0.2] [scale=0.8]{$3$};\node at (olr)[above=0.2] [scale=0.8]{$1$};\node at (or)[above=0.2] [scale=0.8]{$2$};
\node at  (ol)[left=0.2] [scale=0.8]{$\omega$};\node at  (o)[left=0.2] [scale=0.8]{$\tau$};}
+\treey{\cdlr{ol}\node at (oll)[above=0.2] [scale=0.8]{$1$};\node at (olr)[above=0.2] [scale=0.8]{$2$};\node at (or)[above=0.2] [scale=0.8]{$3$};
\node at  (ol)[left=0.2] [scale=0.8]{$\tau$};\node at  (o)[left=0.2] [scale=0.8]{$\omega$};}
+\treey{\cdlr{ol}\node at (oll)[above=0.2] [scale=0.8]{$2$};\node at (olr)[above=0.2] [scale=0.8]{$3$};\node at (or)[above=0.2] [scale=0.8]{$1$};
\node at  (ol)[left=0.2] [scale=0.8]{$\tau$};\node at  (o)[left=0.2] [scale=0.8]{$\omega$};}
+\treey{\cdlr{ol}\node at (oll)[above=0.2] [scale=0.8]{$3$};\node at (olr)[above=0.2] [scale=0.8]{$1$};\node at (or)[above=0.2] [scale=0.8]{$2$};
\node at  (ol)[left=0.2] [scale=0.8]{$\tau$};\node at  (o)[left=0.2] [scale=0.8]{$\omega$};}
,\,
\treey{\cdlr{ol}\node at (oll)[above=0.2] [scale=0.8]{$1$};\node at (olr)[above=0.2] [scale=0.8]{$2$};\node at (or)[above=0.2] [scale=0.8]{$3$};
\node at  (ol)[left=0.2] [scale=0.8]{$\omega$};\node at  (o)[left=0.2] [scale=0.8]{$\omega$};}
+\treey{\cdlr{ol}\node at (oll)[above=0.2] [scale=0.8]{$2$};\node at (olr)[above=0.2] [scale=0.8]{$3$};\node at (or)[above=0.2] [scale=0.8]{$1$};
\node at  (ol)[left=0.2] [scale=0.8]{$\omega$};\node at  (o)[left=0.2] [scale=0.8]{$\omega$};}
+\treey{\cdlr{ol}\node at (oll)[above=0.2] [scale=0.8]{$3$};\node at (olr)[above=0.2] [scale=0.8]{$1$};\node at (or)[above=0.2] [scale=0.8]{$2$};
\node at  (ol)[left=0.2] [scale=0.8]{$\omega$};\node at  (o)[left=0.2] [scale=0.8]{$\omega$};}
\, \Bigg|\,\omega,\tau\in\Omega \Biggr\}\\
&&
\bigdot~\Biggr\{\sum
\kaa{i}{j}{1}\treey{\cdlr{ol}\node at (oll)[above=0.2] [scale=0.8]{$1$};\node at (olr)[above=0.2] [scale=0.8]{$2$};\node at (or)[above=0.2] [scale=0.8]{$3$};
\node at  (ol)[left=0.2] [scale=0.8]{$\mu_i$};\node at  (o)[left=0.2] [scale=0.8]{$\mu_j$};}
+\kaa{i}{j}{2}\treey{\cdlr{ol}\node at (oll)[above=0.2] [scale=0.8]{$2$};\node at (olr)[above=0.2] [scale=0.8]{$3$};\node at (or)[above=0.2] [scale=0.8]{$1$};
\node at  (ol)[left=0.2] [scale=0.8]{$\mu_i$};\node at  (o)[left=0.2] [scale=0.8]{$\mu_j$};}
+\kaa{i}{j}{3}\treey{\cdlr{ol}\node at (oll)[above=0.2] [scale=0.8]{$3$};\node at (olr)[above=0.2] [scale=0.8]{$1$};\node at (or)[above=0.2] [scale=0.8]{$2$};
\node at  (ol)[left=0.2] [scale=0.8]{$\mu_i$};\node at  (o)[left=0.2] [scale=0.8]{$\mu_j$};}
\quad \Bigg| \, \mu_i,\mu_j\in M ,\kaa{i}{j}{1}, \kaa{i}{j}{2},\kaa{i}{j}{3}\in\bfk \Biggr\}\\
&\cong&\Biggr\{\sum
\kaa{i}{j}{1}\treey{\cdlr{ol}\node at (oll)[above=0.2] [scale=0.8]{$1$};\node at (olr)[above=0.2] [scale=0.8]{$2$};\node at (or)[above=0.2] [scale=0.8]{$3$};
\node at  (ol)[left=0.1cm] [scale=0.8]{$\mu_i^\omega$};\node at  (o)[left=0.1cm] [scale=0.8]{$\mu_j^\tau$};}
+\kaa{i}{j}{2}\treey{\cdlr{ol}\node at (oll)[above=0.2] [scale=0.8]{$2$};\node at (olr)[above=0.2] [scale=0.8]{$3$};\node at (or)[above=0.2] [scale=0.8]{$1$};
\node at  (ol)[left=0.1cm] [scale=0.8]{$\mu_i^\omega$};\node at  (o)[left=0.1cm] [scale=0.8]{$\mu_j^\tau$};}
+\kaa{i}{j}{3}\treey{\cdlr{ol}\node at (oll)[above=0.2] [scale=0.8]{$3$};\node at (olr)[above=0.2] [scale=0.8]{$1$};\node at (or)[above=0.2] [scale=0.8]{$2$};
\node at  (ol)[left=0.1cm] [scale=0.8]{$\mu_i^\omega$};\node at  (o)[left=0.1cm] [scale=0.8]{$\mu_j^\tau$};}
+
\kaa{i}{j}{1}\treey{\cdlr{ol}\node at (oll)[above=0.2] [scale=0.8]{$1$};\node at (olr)[above=0.2] [scale=0.8]{$2$};\node at (or)[above=0.2] [scale=0.8]{$3$};
\node at  (ol)[left=0.1cm] [scale=0.8]{$\mu_i^\tau$};\node at  (o)[left=0.1cm] [scale=0.8]{$\mu_j^\omega$};}
+\kaa{i}{j}{2}\treey{\cdlr{ol}\node at (oll)[above=0.2] [scale=0.8]{$2$};\node at (olr)[above=0.2] [scale=0.8]{$3$};\node at (or)[above=0.2] [scale=0.8]{$1$};
\node at  (ol)[left=0.1cm] [scale=0.8]{$\mu_i^\tau$};\node at  (o)[left=0.1cm] [scale=0.8]{$\mu_j^\omega$};}
+\kaa{i}{j}{3}\treey{\cdlr{ol}\node at (oll)[above=0.2] [scale=0.8]{$3$};\node at (olr)[above=0.2] [scale=0.8]{$1$};\node at (or)[above=0.2] [scale=0.8]{$2$};
\node at  (ol)[left=0.1cm] [scale=0.8]{$\mu_i^\tau$};\node at  (o)[left=0.1cm] [scale=0.8]{$\mu_j^\omega$};},\\
&&
\sum
\kaa{i}{j}{1}\treey{\cdlr{ol}\node at (oll)[above=0.2] [scale=0.8]{$1$};\node at (olr)[above=0.2] [scale=0.8]{$2$};\node at (or)[above=0.2] [scale=0.8]{$3$};
\node at  (ol)[left=0.1cm] [scale=0.8]{$\mu_i^\omega$};\node at  (o)[left=0.1cm] [scale=0.8]{$\mu_j^\omega$};}
+\kaa{i}{j}{2}\treey{\cdlr{ol}\node at (oll)[above=0.2] [scale=0.8]{$2$};\node at (olr)[above=0.2] [scale=0.8]{$3$};\node at (or)[above=0.2] [scale=0.8]{$1$};
\node at  (ol)[left=0.1cm] [scale=0.8]{$\mu_i^\omega$};\node at  (o)[left=0.1cm] [scale=0.8]{$\mu_j^\omega$};}
+\kaa{i}{j}{3}\treey{\cdlr{ol}\node at (oll)[above=0.2] [scale=0.8]{$3$};\node at (olr)[above=0.2] [scale=0.8]{$1$};\node at (or)[above=0.2] [scale=0.8]{$2$};
\node at  (ol)[left=0.1cm] [scale=0.8]{$\mu_i^\omega$};\node at  (o)[left=0.1cm] [scale=0.8]{$\mu_j^\omega$};}
\quad \Bigg| \, \mu_i^\omega,\mu_j^\tau\in \bigoplus_{\omega\in\Omega}M_\omega, \kaa{i}{j}{1}, \kaa{i}{j}{2},\kaa{i}{j}{3}\in\bfk \Biggr\}\\
&=& \lrr.
\end{eqnarray*}
}
Therefore
\vsb
$$ \lin{\lie}_\Omega\bigdot\spp= \stt\bigg(\big(\bigoplus_{\omega\in\Omega}\lieg_{,\omega}\big)\otimes M\bigg)\bigg/\big\langle \lrrs{\lier}~\bigdot ~R\big\rangle\cong
\stt\big(\bigoplus_{\omega\in\Omega}M_\omega\big)\Big/\big\langle \lrr\big\rangle=\lin{\spp}_\Omega.
$$
This completes the proof.
\end{proof}
\vsb

\begin{proposition}\mlabel{prop:maninbll}
Let $\Omega$ be a nonempty finite set. Let $\spp=\mscr{T}(M)/\langle R\rangle$ be a finitely generated binary quadratic operad.
Then we have
$$\lmat{\lie}_\Omega\bigdot\spp\cong\lmat{\spp}_\Omega\cong\lmat{\com}_\Omega \bigcir \spp. $$
\end{proposition}
\vsc
\begin{proof}
By a direct calculation, we obtain
$$(\lier)_{\rm LMT}= \Biggr\{\treey{\cdlr{ol}\node at (oll)[above=0.2] [scale=0.8]{$1$};\node at (olr)[above=0.2] [scale=0.8]{$2$};\node at (or)[above=0.2] [scale=0.8]{$3$};
\node at  (ol)[left=0.2] [scale=0.8]{$\omega$};\node at  (o)[left=0.2] [scale=0.8]{$\omega$};}
+\treey{\cdlr{ol}\node at (oll)[above=0.2] [scale=0.8]{$2$};\node at (olr)[above=0.2] [scale=0.8]{$3$};\node at (or)[above=0.2] [scale=0.8]{$1$};
\node at  (ol)[left=0.2] [scale=0.8]{$\omega$};\node at  (o)[left=0.2] [scale=0.8]{$\omega$};}
+\treey{\cdlr{ol}\node at (oll)[above=0.2] [scale=0.8]{$3$};\node at (olr)[above=0.2] [scale=0.8]{$1$};\node at (or)[above=0.2] [scale=0.8]{$2$};
\node at  (ol)[left=0.2] [scale=0.8]{$\omega$};\node at  (o)[left=0.2] [scale=0.8]{$\omega$};}
,\quad
\treey{\cdlr{ol}\node at (oll)[above=0.2] [scale=0.8]{$1$};\node at (olr)[above=0.2] [scale=0.8]{$2$};\node at (or)[above=0.2] [scale=0.8]{$3$};
\node at  (ol)[left=0.2] [scale=0.8]{$\omega$};\node at  (o)[left=0.2] [scale=0.8]{$\tau$};}
+\treey{\cdlr{ol}\node at (oll)[above=0.2] [scale=0.8]{$2$};\node at (olr)[above=0.2] [scale=0.8]{$3$};\node at (or)[above=0.2] [scale=0.8]{$1$};
\node at  (ol)[left=0.2] [scale=0.8]{$\omega$};\node at  (o)[left=0.2] [scale=0.8]{$\tau$};}
+\treey{\cdlr{ol}\node at (oll)[above=0.2] [scale=0.8]{$3$};\node at (olr)[above=0.2] [scale=0.8]{$1$};\node at (or)[above=0.2] [scale=0.8]{$2$};
\node at  (ol)[left=0.2] [scale=0.8]{$\omega$};\node at  (o)[left=0.2] [scale=0.8]{$\tau$};}
\quad\Bigg|\,\omega,\tau\in\Omega  \Biggr\}.$$
With the notations in Eqs.~\meqref{eq:m2} and~\meqref{eq:relar},
it follows from Definition~\mref{de:mp}~\mref{it:mbp} that
$$\lmat{\lie}_\Omega\bigdot\spp=\stt\bigg(\big(\bigoplus_{\omega\in\Omega}\lieg_{,\omega}\big)\otimes M\bigg)\Big/\Big\langle (\lier)_{\rm LMT}~\bigdot ~R\Big\rangle,
\vsb
$$
where
\vsb
$$\Big(\bigoplus_{\omega\in\Omega}\lieg_{,\omega}\Big)\otimes M
=\Big(\bigoplus_{\omega\in\Omega}\treey{\cdlr{o}\node at  (0,0.25) [scale=0.8]{$\omega$};}\Big)\otimes
\Big(\bigoplus_{1\leq i \leq s}\treey{\cdlr{o}\node at  (0,0.25) [scale=0.8]{$\mu_i$};}\Big)
=\bigoplus_{\omega\in\Omega, 1\leq i \leq s}\treey{\cdlr{o}\node at  (o)[left=0.8] [scale=0.8]{$\omega\otimes \mu_i$};}
\cong \bigoplus_{\omega\in\Omega, 1\leq i \leq s}\treey{\cdlr{o}\node at  (0,0.25) [scale=0.8]{$ \mu_i^\omega$};}
=\bigoplus_{\omega\in\Omega}M_\omega,
\vsb
$$
and
\vsc
\small{
\begin{eqnarray*}
(\lier)_{\rm LMT}~\bigdot~ R
&=& \Biggr\{\treey{\cdlr{ol}\node at (oll)[above=0.2] [scale=0.8]{$1$};\node at (olr)[above=0.2] [scale=0.8]{$2$};\node at (or)[above=0.2] [scale=0.8]{$3$};
\node at  (ol)[left=0.2] [scale=0.8]{$\omega$};\node at  (o)[left=0.2] [scale=0.8]{$\tau$};}
+\treey{\cdlr{ol}\node at (oll)[above=0.2] [scale=0.8]{$2$};\node at (olr)[above=0.2] [scale=0.8]{$3$};\node at (or)[above=0.2] [scale=0.8]{$1$};
\node at  (ol)[left=0.2] [scale=0.8]{$\omega$};\node at  (o)[left=0.2] [scale=0.8]{$\tau$};}
+\treey{\cdlr{ol}\node at (oll)[above=0.2] [scale=0.8]{$3$};\node at (olr)[above=0.2] [scale=0.8]{$1$};\node at (or)[above=0.2] [scale=0.8]{$2$};
\node at  (ol)[left=0.2] [scale=0.8]{$\omega$};\node at  (o)[left=0.2] [scale=0.8]{$\tau$};}
,\quad
\treey{\cdlr{ol}\node at (oll)[above=0.2] [scale=0.8]{$1$};\node at (olr)[above=0.2] [scale=0.8]{$2$};\node at (or)[above=0.2] [scale=0.8]{$3$};
\node at  (ol)[left=0.2] [scale=0.8]{$\omega$};\node at  (o)[left=0.2] [scale=0.8]{$\omega$};}
+\treey{\cdlr{ol}\node at (oll)[above=0.2] [scale=0.8]{$2$};\node at (olr)[above=0.2] [scale=0.8]{$3$};\node at (or)[above=0.2] [scale=0.8]{$1$};
\node at  (ol)[left=0.2] [scale=0.8]{$\omega$};\node at  (o)[left=0.2] [scale=0.8]{$\omega$};}
+\treey{\cdlr{ol}\node at (oll)[above=0.2] [scale=0.8]{$3$};\node at (olr)[above=0.2] [scale=0.8]{$1$};\node at (or)[above=0.2] [scale=0.8]{$2$};
\node at  (ol)[left=0.2] [scale=0.8]{$\omega$};\node at  (o)[left=0.2] [scale=0.8]{$\omega$};}
\quad\Bigg|\,\omega,\tau\in\Omega  \Biggr\}\\
&&
\bigdot~ \Biggr\{\sum
\kaa{i}{j}{1}\treey{\cdlr{ol}\node at (oll)[above=0.2] [scale=0.8]{$1$};\node at (olr)[above=0.2] [scale=0.8]{$2$};\node at (or)[above=0.2] [scale=0.8]{$3$};
\node at  (ol)[left=0.2] [scale=0.8]{$\mu_i$};\node at  (o)[left=0.2] [scale=0.8]{$\mu_j$};}
+\kaa{i}{j}{2}\treey{\cdlr{ol}\node at (oll)[above=0.2] [scale=0.8]{$2$};\node at (olr)[above=0.2] [scale=0.8]{$3$};\node at (or)[above=0.2] [scale=0.8]{$1$};
\node at  (ol)[left=0.2] [scale=0.8]{$\mu_i$};\node at  (o)[left=0.2] [scale=0.8]{$\mu_j$};}
+\kaa{i}{j}{3}\treey{\cdlr{ol}\node at (oll)[above=0.2] [scale=0.8]{$3$};\node at (olr)[above=0.2] [scale=0.8]{$1$};\node at (or)[above=0.2] [scale=0.8]{$2$};
\node at  (ol)[left=0.2] [scale=0.8]{$\mu_i$};\node at  (o)[left=0.2] [scale=0.8]{$\mu_j$};}
\quad \Bigg| \, \mu_i,\mu_j\in M ,\kaa{i}{j}{1}, \kaa{i}{j}{2},\kaa{i}{j}{3}\in\bfk  \Biggr\}\\
&\cong& \Biggr\{\sum
\kaa{i}{j}{1}\treey{\cdlr{ol}\node at (oll)[above=0.2] [scale=0.8]{$1$};\node at (olr)[above=0.2] [scale=0.8]{$2$};\node at (or)[above=0.2] [scale=0.8]{$3$};
\node at  (ol)[left=0.1cm] [scale=0.8]{$\mu_i^\omega$};\node at  (o)[left=0.1cm] [scale=0.8]{$\mu_j^\tau$};}
+\kaa{i}{j}{2}\treey{\cdlr{ol}\node at (oll)[above=0.2] [scale=0.8]{$2$};\node at (olr)[above=0.2] [scale=0.8]{$3$};\node at (or)[above=0.2] [scale=0.8]{$1$};
\node at  (ol)[left=0.1cm] [scale=0.8]{$\mu_i^\omega$};\node at  (o)[left=0.1cm] [scale=0.8]{$\mu_j^\tau$};}
+\kaa{i}{j}{3}\treey{\cdlr{ol}\node at (oll)[above=0.2] [scale=0.8]{$3$};\node at (olr)[above=0.2] [scale=0.8]{$1$};\node at (or)[above=0.2] [scale=0.8]{$2$};
\node at  (ol)[left=0.1cm] [scale=0.8]{$\mu_i^\omega$};\node at  (o)[left=0.1cm] [scale=0.8]{$\mu_j^\tau$};},
\sum
\kaa{i}{j}{1}\treey{\cdlr{ol}\node at (oll)[above=0.2] [scale=0.8]{$1$};\node at (olr)[above=0.2] [scale=0.8]{$2$};\node at (or)[above=0.2] [scale=0.8]{$3$};
\node at  (ol)[left=0.1cm] [scale=0.8]{$\mu_i^\omega$};\node at  (o)[left=0.1cm] [scale=0.8]{$\mu_j^\omega$};}
+\kaa{i}{j}{2}\treey{\cdlr{ol}\node at (oll)[above=0.2] [scale=0.8]{$2$};\node at (olr)[above=0.2] [scale=0.8]{$3$};\node at (or)[above=0.2] [scale=0.8]{$1$};
\node at  (ol)[left=0.1cm] [scale=0.8]{$\mu_i^\omega$};\node at  (o)[left=0.1cm] [scale=0.8]{$\mu_j^\omega$};}
+\kaa{i}{j}{3}\treey{\cdlr{ol}\node at (oll)[above=0.2] [scale=0.8]{$3$};\node at (olr)[above=0.2] [scale=0.8]{$1$};\node at (or)[above=0.2] [scale=0.8]{$2$};
\node at  (ol)[left=0.1cm] [scale=0.8]{$\mu_i^\omega$};\node at  (o)[left=0.1cm] [scale=0.8]{$\mu_j^\omega$};}\\
&&\quad \Bigg| \, \mu_i^\omega,\mu_j^\tau\in \bigoplus_{\omega\in\Omega}M_\omega, \kaa{i}{j}{1}, \kaa{i}{j}{2},\kaa{i}{j}{3}\in\bfk  \Biggr\}\\
&=& R_{\rm LMT}.
\vsc
\end{eqnarray*}
}
Thus
\vsb
$$ \lmat{\lie}_\Omega\bigdot\spp= \stt\bigg(\big(\bigoplus_{\omega\in\Omega}\lieg_{,\omega}\big)\otimes M\bigg)\Big/\Big\langle (\lier)_{\rm LMT}~\bigdot ~R\Big\rangle\cong
\stt\big(\bigoplus_{\omega\in\Omega}M_\omega\big)\big/\big\langle  R_{\rm LMT}\big\rangle=\lmat{\spp}_\Omega.
\vsb
$$
It follows from $\lie^!=\com$ and Theorem~\mref{thm:mdul} that
$$ \lmat{\com}_{\Omega}  \bigcir \spp^!= \Big(\lmat{\lie}_{\Omega}\Big)^{!} \bigcir \spp^! =\Big(\lmat{\lie}_\Omega\bigdot\spp\Big)^! \cong \Big(\lmat{\spp}_\Omega\Big)^! = \lmat{\Big(\mscr{P}^!\Big)}_\Omega.
$$
Therefore $\lmat{\spp}_\Omega\cong\lmat{\com}_\Omega \bigcir \spp$.
\vsc
\end{proof}

\section{Koszulness of binary quadratic compatible structures}
\mlabel{sec:koszul}
This section studies the preservation of Koszulness under the process of taking a compatible construction. 

We first show that having a Gr\"obner basis~\cite{DK2} is preserved by taking the leveled matching compatible construction. 

\begin{theorem}\mlabel{thm:wak}
Let $\spp$ be a binary quadratic operad defined by a set of quadratic relations that is a Gr\"{o}bner basis.
Then the induced set of leveled matching relations is also a Gr\"{o}bner basis for the operad $\lmat{\spp}_\Omega$. Moreover, $\lmat{\spp}_\Omega\cong\spp\bigcir\lmat{\com}_\Omega$ is Koszul.
\end{theorem}
\begin{proof}
A quadratic relation of $\spp$ is of the form
$$
r=\sum\kappa_{i,j}^{1}\treey{\cdlr{ol}\node at (oll)[above=0.2] [scale=0.8]{$1$};\node at (olr)[above=0.2] [scale=0.8]{$2$};\node at (or)[above=0.2] [scale=0.8]{$3$};
\node at  (ol)[left=0.2] [scale=0.8]{$\mu_i$};\node at  (o)[left=0.2] [scale=0.8]{$\mu_j$};}
+\kappa_{i,j}^{2}\treey{\cdlr{or}\node at (ol)[above=0.2] [scale=0.8]{$1$};\node at (orr)[above=0.2] [scale=0.8]{$3$};\node at (orl)[above=0.2] [scale=0.8]{$2$};
\node at  (or)[left=0.2] [scale=0.8]{$\mu_i$};\node at  (o)[left=0.2] [scale=0.8]{$\mu_j$};}
+\kappa_{i,j}^{3}\treey{\cdlr{ol}\node at (oll)[above=0.2] [scale=0.8]{$1$};\node at (olr)[above=0.2] [scale=0.8]{$3$};\node at (or)[above=0.2] [scale=0.8]{$2$};
\node at  (ol)[left=0.2] [scale=0.8]{$\mu_i$};\node at  (o)[left=0.2] [scale=0.8]{$\mu_j$};}
\quad  \text{ for some }\, \mu_i,\mu_j\in \spp(2) ,\kappa_{i,j}^{1}, \kappa_{i,j}^{2},\kappa_{i,j}^{3}\in\bfk.$$
Then its induced leveled matching relation can be obtained by labeling the vertices of $r$ with $\alpha,\beta\in\Omega$:
$$
r^{{\rm MT}}=\sum\kappa_{i,j}^{1}\treey{\cdlr{ol}\node at (oll)[above=0.2] [scale=0.8]{$1$};\node at (olr)[above=0.2] [scale=0.8]{$2$};\node at (or)[above=0.2] [scale=0.8]{$3$};
\node at  (ol)[left=0.2] [scale=0.8]{$\mu_{i,\alpha}$};\node at  (o)[left=0.2] [scale=0.8]{$\mu_{j,\beta}$};}
+\kappa_{i,j}^{2}\treey{\cdlr{or}\node at (ol)[above=0.2] [scale=0.8]{$1$};\node at (orr)[above=0.2] [scale=0.8]{$3$};\node at (orl)[above=0.2] [scale=0.8]{$2$};
\node at  (or)[left=0.2] [scale=0.8]{$\mu_{i,\alpha}$};\node at  (o)[left=0.2] [scale=0.8]{$\mu_{j,\beta}$};}
+\kappa_{i,j}^{3}\treey{\cdlr{ol}\node at (oll)[above=0.2] [scale=0.8]{$1$};\node at (olr)[above=0.2] [scale=0.8]{$3$};\node at (or)[above=0.2] [scale=0.8]{$2$};
\node at  (ol)[left=0.2] [scale=0.8]{$\mu_{i,\alpha}$};\node at  (o)[left=0.2] [scale=0.8]{$\mu_{j,\beta}$};},
\quad \, \mu_i,\mu_j\in \spp(2) ,\kappa_{i,j}^{1}, \kappa_{i,j}^{2},\kappa_{i,j}^{3}\in\bfk.$$
It is straightforward to see that the monomial ordering on $\spp(2)$-decorated trees can be extended to $\spp(2)_\Omega$-decorated trees such that
the leading term of $r^{{\rm MT}}$ can be obtained by labeling the vertices of leading monomial of $r$ by $\alpha,\beta$. For example, if
$\bar{r}=\treey{\cdlr{ol}\node at (oll)[above=0.2] [scale=0.8]{$1$};\node at (olr)[above=0.2] [scale=0.8]{$2$};\node at (or)[above=0.2] [scale=0.8]{$3$};
\node at  (ol)[left=0.2] [scale=0.8]{$\mu_i$};\node at  (o)[left=0.2] [scale=0.8]{$\mu_j$};}$, then
$\overline{{r^{{\rm MT}}}}=\treey{\cdlr{ol}\node at (oll)[above=0.2] [scale=0.8]{$1$};\node at (olr)[above=0.2] [scale=0.8]{$2$};\node at (or)[above=0.2] [scale=0.8]{$3$};
\node at  (ol)[left=0.2] [scale=0.8]{$\mu_{i,\alpha}$};\node at  (o)[left=0.2] [scale=0.8]{$\mu_{j,\beta}$};}$.
By the leveled condition, every level of vertices of the support trees is labeled  by an element of $\Omega$. Thus the overlaps and compositions of leveled matching relations can be obtained by labeling the vertices of the overlaps and compositions of the original relations. Therefore, if the original set of relations is a Gr\"obner basis, then the set of the leveled matching relation is also a Gr\"{o}bner basis.

Further by \cite[Theorem 8.3.1]{LV}, the operads $\lmat{\spp}_\Omega\cong\spp\bigcir\lmat{\com}_\Omega$ is Koszul.
\end{proof}

\delete{
\begin{remark}
The above result does not hold for non-leveled matching opeards. The leveled condition guarantees that the overlaps and compositions of the leveled matching relation
can be obtained by labelling the vertices of the overlaps and compositions of the original relation.
\end{remark}
}

To see more precisely how the general process in the proof works, we give a detailed example. 
\begin{example}\label{ex:st}
Let $\Omega$ be a finite ordered nonempty set. Then operad  $\lmat{\com}_\Omega$ has a Gr\"{o}bner basis.

Write the ordered set $\Omega$ as $\{\omega_1<\omega_2<\cdots< \omega_n\}$, where $|\Omega|=n$.
We use the path-lexicographic order $\leq_{\rm pl}$ on binary shuffle trees~\cite[Section. 8.4.1]{LV}.
Restricting this order to all $2$-vertices binary shuffle trees gives
\begin{equation} \notag
\resizebox{\textwidth}{!}{$
\begin{split}
&\treey{\cdlr{or}\node at (ol)[above=0.2] [scale=0.8]{$1$};\node at (orl)[above=0.2] [scale=0.8]{$2$};\node at (orr)[above=0.2] [scale=0.8]{$3$};
\node at  (or)[left=0.2] [scale=0.8]{$\omega_1$};\node at  (o)[left=0.2] [scale=0.8]{$\omega_1$};}
<_{\rm pl}
\treey{\cdlr{or}\node at (ol)[above=0.2] [scale=0.8]{$1$};\node at (orl)[above=0.2] [scale=0.8]{$2$};\node at (orr)[above=0.2] [scale=0.8]{$3$};
\node at  (or)[left=0.2] [scale=0.8]{$\omega_2$};\node at  (o)[left=0.2] [scale=0.8]{$\omega_1$};}
<_{\rm pl}\cdots
<_{\rm pl}
\treey{\cdlr{or}\node at (ol)[above=0.2] [scale=0.8]{$1$};\node at (orl)[above=0.2] [scale=0.8]{$2$};\node at (orr)[above=0.2] [scale=0.8]{$3$};
\node at  (or)[left=0.2] [scale=0.8]{$\omega_n$};\node at  (o)[left=0.2] [scale=0.8]{$\omega_1$};}
<_{\rm pl}
\treey{\cdlr{or}\node at (ol)[above=0.2] [scale=0.8]{$1$};\node at (orl)[above=0.2] [scale=0.8]{$2$};\node at (orr)[above=0.2] [scale=0.8]{$3$};
\node at  (or)[left=0.2] [scale=0.8]{$\omega_1$};\node at  (o)[left=0.2] [scale=0.8]{$\omega_2$};}
<_{\rm pl}
\treey{\cdlr{or}\node at (ol)[above=0.2] [scale=0.8]{$1$};\node at (orl)[above=0.2] [scale=0.8]{$2$};\node at (orr)[above=0.2] [scale=0.8]{$3$};
\node at  (or)[left=0.2] [scale=0.8]{$\omega_2$};\node at  (o)[left=0.2] [scale=0.8]{$\omega_2$};}
<_{\rm pl}\cdots
<_{\rm pl}
\treey{\cdlr{or}\node at (ol)[above=0.2] [scale=0.8]{$1$};\node at (orl)[above=0.2] [scale=0.8]{$2$};\node at (orr)[above=0.2] [scale=0.8]{$3$};
\node at  (or)[left=0.2] [scale=0.8]{$\omega_n$};\node at  (o)[left=0.2] [scale=0.8]{$\omega_2$};}
<_{\rm pl}\cdots
<_{\rm pl}\treey{\cdlr{or}\node at (ol)[above=0.2] [scale=0.8]{$1$};\node at (orl)[above=0.2] [scale=0.8]{$2$};\node at (orr)[above=0.2] [scale=0.8]{$3$};
\node at  (or)[left=0.2] [scale=0.8]{$\omega_n$};\node at  (o)[left=0.2] [scale=0.8]{$\omega_n$};}\\
<_{\rm pl}
&\treey{\cdlr{ol}\node at (oll)[above=0.2] [scale=0.8]{$1$};\node at (olr)[above=0.2] [scale=0.8]{$3$};\node at (or)[above=0.2] [scale=0.8]{$2$};
\node at  (ol)[left=0.2] [scale=0.8]{$\omega_1$};\node at  (o)[left=0.2] [scale=0.8]{$\omega_1$};}
<_{\rm pl}
\treey{\cdlr{ol}\node at (oll)[above=0.2] [scale=0.8]{$1$};\node at (olr)[above=0.2] [scale=0.8]{$2$};\node at (or)[above=0.2] [scale=0.8]{$3$};
\node at  (ol)[left=0.2] [scale=0.8]{$\omega_1$};\node at  (o)[left=0.2] [scale=0.8]{$\omega_1$};}
<_{\rm pl}
\treey{\cdlr{ol}\node at (oll)[above=0.2] [scale=0.8]{$1$};\node at (olr)[above=0.2] [scale=0.8]{$3$};\node at (or)[above=0.2] [scale=0.8]{$2$};
\node at  (ol)[left=0.2] [scale=0.8]{$\omega_2$};\node at  (o)[left=0.2] [scale=0.8]{$\omega_1$};}
<_{\rm pl}
\treey{\cdlr{ol}\node at (oll)[above=0.2] [scale=0.8]{$1$};\node at (olr)[above=0.2] [scale=0.8]{$2$};\node at (or)[above=0.2] [scale=0.8]{$3$};
\node at  (ol)[left=0.2] [scale=0.8]{$\omega_2$};\node at  (o)[left=0.2] [scale=0.8]{$\omega_1$};}
<_{\rm pl}\cdots
<_{\rm pl}\treey{\cdlr{ol}\node at (oll)[above=0.2] [scale=0.8]{$1$};\node at (olr)[above=0.2] [scale=0.8]{$3$};\node at (or)[above=0.2] [scale=0.8]{$2$};
\node at  (ol)[left=0.2] [scale=0.8]{$\omega_n$};\node at  (o)[left=0.2] [scale=0.8]{$\omega_{n-1}$};}
<_{\rm pl}
\treey{\cdlr{ol}\node at (oll)[above=0.2] [scale=0.8]{$1$};\node at (olr)[above=0.2] [scale=0.8]{$2$};\node at (or)[above=0.2] [scale=0.8]{$3$};
\node at  (ol)[left=0.2] [scale=0.8]{$\omega_n$};\node at  (o)[left=0.2] [scale=0.8]{$\omega_{n-1}$};}
<_{\rm pl}
\treey{\cdlr{ol}\node at (oll)[above=0.2] [scale=0.8]{$1$};\node at (olr)[above=0.2] [scale=0.8]{$3$};\node at (or)[above=0.2] [scale=0.8]{$2$};
\node at  (ol)[left=0.2] [scale=0.8]{$\omega_n$};\node at  (o)[left=0.2] [scale=0.8]{$\omega_n$};}
<_{\rm pl}
\treey{\cdlr{ol}\node at (oll)[above=0.2] [scale=0.8]{$1$};\node at (olr)[above=0.2] [scale=0.8]{$2$};\node at (or)[above=0.2] [scale=0.8]{$3$};
\node at  (ol)[left=0.2] [scale=0.8]{$\omega_n$};\node at  (o)[left=0.2] [scale=0.8]{$\omega_n$};}.
\end{split}
$}
\end{equation}
\nc\rtcml{\rightarrow_{\rm LMT}}\nc\ltcml{\leftarrow_{\rm LMT}}

The rewriting rules of the operad $\lmat{\com}_\Omega$ are
\begin{align*}
\treey{\cdlr{ol}\node at (oll)[above=0.2] [scale=0.8]{$1$};\node at (olr)[above=0.2] [scale=0.8]{$2$};\node at (or)[above=0.2] [scale=0.8]{$3$};
\node at  (ol)[left=0.2] [scale=0.8]{$\omega_i$};\node at  (o)[left=0.2] [scale=0.8]{$\omega_i$};}
\rtcml
\treey{\cdlr{or}\node at (ol)[above=0.2] [scale=0.8]{$1$};\node at (orl)[above=0.2] [scale=0.8]{$2$};\node at (orr)[above=0.2] [scale=0.8]{$3$};
\node at  (or)[left=0.2] [scale=0.8]{$\omega_i$};\node at  (o)[left=0.2] [scale=0.8]{$\omega_i$};},\quad
\treey{\cdlr{ol}\node at (oll)[above=0.2] [scale=0.8]{$1$};\node at (olr)[above=0.2] [scale=0.8]{$3$};\node at (or)[above=0.2] [scale=0.8]{$2$};
\node at  (ol)[left=0.2] [scale=0.8]{$\omega_i$};\node at  (o)[left=0.2] [scale=0.8]{$\omega_i$};}
\rtcml
\treey{\cdlr{or}\node at (ol)[above=0.2] [scale=0.8]{$1$};\node at (orl)[above=0.2] [scale=0.8]{$2$};\node at (orr)[above=0.2] [scale=0.8]{$3$};
\node at  (or)[left=0.2] [scale=0.8]{$\omega_i$};\node at  (o)[left=0.2] [scale=0.8]{$\omega_i$};},\quad
\treey{\cdlr{ol}\node at (oll)[above=0.2] [scale=0.8]{$1$};\node at (olr)[above=0.2] [scale=0.8]{$2$};\node at (or)[above=0.2] [scale=0.8]{$3$};
\node at  (ol)[left=0.2] [scale=0.8]{$\omega_j$};\node at  (o)[left=0.2] [scale=0.8]{$\omega_i$};}
\rtcml
\treey{\cdlr{or}\node at (ol)[above=0.2] [scale=0.8]{$1$};\node at (orl)[above=0.2] [scale=0.8]{$2$};\node at (orr)[above=0.2] [scale=0.8]{$3$};
\node at  (or)[left=0.2] [scale=0.8]{$\omega_j$};\node at  (o)[left=0.2] [scale=0.8]{$\omega_i$};},\quad
\treey{\cdlr{ol}\node at (oll)[above=0.2] [scale=0.8]{$1$};\node at (olr)[above=0.2] [scale=0.8]{$3$};\node at (or)[above=0.2] [scale=0.8]{$2$};
\node at  (ol)[left=0.2] [scale=0.8]{$\omega_j$};\node at  (o)[left=0.2] [scale=0.8]{$\omega_i$};}
\rtcml
\treey{\cdlr{or}\node at (ol)[above=0.2] [scale=0.8]{$1$};\node at (orl)[above=0.2] [scale=0.8]{$2$};\node at (orr)[above=0.2] [scale=0.8]{$3$};
\node at  (or)[left=0.2] [scale=0.8]{$\omega_j$};\node at  (o)[left=0.2] [scale=0.8]{$\omega_i$};}
\end{align*}
for $\omega_i, \omega_j\in\Omega$. There are two types of critical monomials
\begin{align*}
\treey{\cdlr{ol}\cdlr{oll}\node at (olll)[above=0.2] [scale=0.8]{$1$};\node at (ollr)[above=0.2] [scale=0.8]{$2$};\node at (olr)[above=0.2] [scale=0.8]{$3$};\node at  (or) [above=0.2] [scale=0.8]{$4$};
\node at  (o) [left=0.2] [scale=0.8]{$\mu$};\node at  (ol) [left=0.2] [scale=0.8]{$\nu$};\node at  (oll)[left=0.2] [scale=0.8]{$\tau$};},
\treey{\cdlr{ol}\cdlr{oll}\node at (olll)[above=0.2] [scale=0.8]{$1$};\node at (ollr)[above=0.2] [scale=0.8]{$2$};\node at (olr)[above=0.2] [scale=0.8]{$4$};\node at  (or) [above=0.2] [scale=0.8]{$3$};
\node at  (o) [left=0.2] [scale=0.8]{$\mu$};\node at  (ol) [left=0.2] [scale=0.8]{$\nu$};\node at  (oll)[left=0.2] [scale=0.8]{$\tau$};}
\,\text{ for }\, (\mu,\nu,\tau)\in \Omega.
\vsb
\end{align*}
The fact that these critical monomials are confluent follows from
\begin{align*}
\treey{\cdlr{or}\cdlr{orr}\node at (ol)[above=0.2] [scale=0.8]{$p$};\node at (orl)[above=0.2] [scale=0.8]{$q$};\node at (orrl)[above=0.2] [scale=0.8]{$s$};\node at  (orrr) [above=0.2] [scale=0.8]{$t$};
\node at  (o) [left=0.2] [scale=0.8]{$\mu$};\node at  (or) [left=0.2] [scale=0.8]{$\nu$};\node at  (orr)[left=0.2] [scale=0.8]{$\tau$};}
\ltcml
\treey{\cdlr{or}\cdlr{orl}\node at (ol)[above=0.2] [scale=0.8]{$p$};\node at (orll)[above=0.2] [scale=0.8]{$q$};\node at (orlr)[above=0.2] [scale=0.8]{$s$};\node at  (orr) [above=0.2] [scale=0.8]{$t$};
\node at  (o) [left=0.2] [scale=0.8]{$\mu$};\node at  (or) [left=0.2] [scale=0.8]{$\nu$};\node at  (orl)[left=0.2] [scale=0.8]{$\tau$};}
\ltcml
\treey{\cdlr{ol}\cdlr{olr}\node at (oll)[above=0.2] [scale=0.8]{$p$};\node at (olrl)[above=0.2] [scale=0.8]{$q$};\node at (olrr)[above=0.2] [scale=0.8]{$s$};\node at  (or) [above=0.2] [scale=0.8]{$t$};
\node at  (o) [left=0.2] [scale=0.8]{$\mu$};\node at  (ol) [left=0.2] [scale=0.8]{$\nu$};\node at  (olr)[left=0.2] [scale=0.8]{$\tau$};}
\ltcml
\treey{\cdlr{ol}\cdlr{oll}\node at (olll)[above=0.2] [scale=0.8]{$p$};\node at (ollr)[above=0.2] [scale=0.8]{$q$};\node at (olr)[above=0.2] [scale=0.8]{$s$};\node at  (or) [above=0.2] [scale=0.8]{$t$};
\node at  (o) [left=0.2] [scale=0.8]{$\mu$};\node at  (ol) [left=0.2] [scale=0.8]{$\nu$};\node at  (oll)[left=0.2] [scale=0.8]{$\tau$};}
\rtcml
\treey{\cdlr[1.2]{o}\cdlr{ol}\cdlr{or}\node at (oll)[above=0.2] [scale=0.8]{$p$};\node at (olr)[above=0.2] [scale=0.8]{$q$};\node at (orl)[above=0.2] [scale=0.8]{$s$};\node at  (orr) [above=0.2] [scale=0.8]{$t$};
\node at  (o) [left=0.2] [scale=0.8]{$\mu$};\node at  (ol) [left=0.2] [scale=0.8]{$\nu$};\node at  (or)[left=0.2] [scale=0.8]{$\tau$};}
\rtcml
\treey{\cdlr{or}\cdlr{orr}\node at (ol)[above=0.2] [scale=0.8]{$p$};\node at (orl)[above=0.2] [scale=0.8]{$q$};\node at (orrl)[above=0.2] [scale=0.8]{$s$};\node at  (orrr) [above=0.2] [scale=0.8]{$t$};
\node at  (o) [left=0.2] [scale=0.8]{$\mu$};\node at  (or) [left=0.2] [scale=0.8]{$\nu$};\node at  (orr)[left=0.2] [scale=0.8]{$\tau$};}
\end{align*}
for $(p,q,s,t)\in \{(1,2,3,4), (1,2,4,3)\}$. If we forget the labelling of elements of $\Omega$, then we obtain the result for the operad $\com$.

Therefore, by \cite[Theorem 8.3.1]{LV}, the operad $\lmat{\com}_\Omega$ is Koszul.
\end{example}

\begin{corollary}\label{coro:ko}
If $\mathscr{P}$ is the operad for any of the associative algebra, Lie algebra, commutative algebra, dendriform algebra, pre-Lie algebra, perm algebra, Leibniz algebra, Zinbiel algebra or Poisson algebra, then $\lmat{\mscr{P}}_\Omega$ is Koszul. 
\end{corollary}



\begin{proof}
The list of operads are Koszul by~\mcite{Zi}.
Then the corresponding operads $\lmat{\mscr{P}}_\Omega$ are Koszul by Theorem~\ref{thm:wak} and the fact that each of the listed operads has a Gr\"{o}bner basis~\cite{BreD,LV}.
\end{proof}

For the rest of the section, fix a set $\Omega$ with two elements.  
Dotsenko~\cite[Corollary 3]{Dot1} and Strohmayer~\cite[Proposition 3.9]{St09} showed independently that for some restrictions on the operad $\spp$, the operads $\lin{\spp}_\Omega$ and  $\tot{\big(\spp^!\big)}_\Omega$ are Koszul. 
The operad $\com$ satisfies the assumption in~\cite{Dot1}. Thus we have the following result.

\begin{proposition}\label{prop:cl}
	The operads $\tot{\lie}_\Omega$ and $\lin{\com}_\Omega$ are Koszul.
\end{proposition}

\begin{theorem}\mlabel{thm:lmtk}
Let $\spp$ be a binary quadratic Koszul operad such that $\big(\spp\circ \spp\big) (4)\cong \tot{\com}_\Omega(4) \otimes \spp(4)$.  Then
the operads $\tot{\spp}_\Omega\cong \tot{\com}_\Omega \bigcir \spp$ and $\lin{\big(\spp^!\big)}_\Omega$ are Koszul.
\end{theorem}


\begin{proof}
We will use the distributive law method~\cite{Mar}. By $\tot{\com}_\Omega$ being weakly associative and~\cite[Proposition 1.27]{St09}, we have
$$\tot{\com}_\Omega \bigcir \spp\cong  \tot{\com}_\Omega \underset{{\rm H}}{\otimes} \spp.$$
Write $\Omega=\{\circ,\bullet\}$ and $\spp=\stt(M)/\langle R\rangle.$ Then
$$\tot{\spp}_\Omega=\stt(M_\circ\oplus M_\bullet)/\langle R_\circ \cup R_{\circ\bullet} \cup R_\bullet\rangle,$$
where $M_\circ \cong M_\bullet\cong M$, $R_\circ \cong R_\bullet\cong R$ and $R_{\circ\bullet}=R_{{\rm TC}}\cup (R_{{\rm LMT}}\setminus(R_\circ \cup R_\bullet))$ is the set of compatible relations.
Notice that
 $$\big(\spp\circ \spp\big) (4)\cong \tot{\com}_\Omega(4) \otimes \spp(4)\cong\big(\tot{\com}_\Omega \underset{{\rm H}}{\otimes} \spp\big)(4)= \tot{\spp}_\Omega(4)=\Big(\stt(M_\circ\oplus M_\bullet)/\langle R_\circ \cup R_{\circ\bullet} \cup R_\bullet\rangle\Big)(4).$$
 Then the relation $R_{\circ\bullet}$ induces a distributive law  in the
sense of Markl~\cite[Definition 2.6]{Mar}. Further by the property of distributive law~\cite[Theorem 5.10]{Mar},  the operad  constructed from operads $\spp$ and $\sqq$ via a distributive law is Koszul if $\spp$ and $\sqq$ are, then $\tot{\spp}_\Omega$ is Koszul.
The Koszulness of $\lin{\big(\spp^!\big)}_\Omega$ follows from Theorem~\mref{thm:dul} and the fact that
the Koszul dual operad of a Koszul operad is also Koszul.
\end{proof}

\begin{example}\label{ex:lc}
The operad $\com$ satisfies
$$\big(\com\circ \com\big) (4)\cong \tot{\com}_\Omega(4)=\tot{\com}_\Omega(4) \otimes \com(4).$$
In fact, $\big(\com\circ \com\big) (4)\to \tot{\com}_\Omega(4)$ is surjective.
It is a injective map from that $\big(\com\circ \com\big) (4)$ and $\tot{\com}_\Omega(4)$
have the same linear basis
$$\treey{\cdlr{or}\cdlr{orr}\node at (ol)[above=0.2] [scale=0.8]{$1$};\node at (orl)[above=0.2] [scale=0.8]{$2$};\node at (orrl)[above=0.2] [scale=0.8]{$3$};\node at  (orrr) [above=0.2] [scale=0.8]{$4$};
\node at  (o) [left=0.2] [scale=0.8]{$\mu_\circ$};\node at  (or) [left=0.2] [scale=0.8]{$\mu_\circ$};\node at  (orr)[left=0.2] [scale=0.8]{$\mu_\circ$};},
\treey{\cdlr{or}\cdlr{orr}\node at (ol)[above=0.2] [scale=0.8]{$1$};\node at (orl)[above=0.2] [scale=0.8]{$2$};\node at (orrl)[above=0.2] [scale=0.8]{$3$};\node at  (orrr) [above=0.2] [scale=0.8]{$4$};
\node at  (o) [left=0.2] [scale=0.8]{$\mu_\circ$};\node at  (or) [left=0.2] [scale=0.8]{$\mu_\circ$};\node at  (orr)[left=0.2] [scale=0.8]{$\mu_\bullet$};},
\treey{\cdlr{or}\cdlr{orr}\node at (ol)[above=0.2] [scale=0.8]{$1$};\node at (orl)[above=0.2] [scale=0.8]{$2$};\node at (orrl)[above=0.2] [scale=0.8]{$3$};\node at  (orrr) [above=0.2] [scale=0.8]{$4$};
\node at  (o) [left=0.2] [scale=0.8]{$\mu_\circ$};\node at  (or) [left=0.2] [scale=0.8]{$\mu_\bullet$};\node at  (orr)[left=0.2] [scale=0.8]{$\mu_\bullet$};},
\treey{\cdlr{or}\cdlr{orr}\node at (ol)[above=0.2] [scale=0.8]{$1$};\node at (orl)[above=0.2] [scale=0.8]{$2$};\node at (orrl)[above=0.2] [scale=0.8]{$3$};\node at  (orrr) [above=0.2] [scale=0.8]{$4$};
\node at  (o) [left=0.2] [scale=0.8]{$\mu_\bullet$};\node at  (or) [left=0.2] [scale=0.8]{$\mu_\bullet$};\node at  (orr)[left=0.2] [scale=0.8]{$\mu_\bullet$};}.$$
Here $\Omega=\{\circ,\bullet\}$, $\mu_\circ\in M_\circ\cong M_{\com }$ and $\mu_\bullet\in M_\bullet\cong M_{\com }$.
By Theorem~\ref{thm:lmtk}, the operads $\tot{\com}_\Omega\cong \tot{\com}_\Omega \bigcir \com$ and $\lin{\lie}_\Omega$ are Koszul.
For another approach, see~\cite{St09}.
\end{example}

\begin{remark}
If the condition $\big(\spp\circ \spp\big) (4)\cong \tot{\com}_\Omega(4) \otimes \spp(4)$ is not satisfied, then conclusion in Theorem~\ref{thm:lmtk} might not hold. For example, take $\mathscr{P}=\mathscr{M}ag$. It is Koszul. But $\lin{(\mathscr{M}ag^!)}_\Omega=\lin{\mathscr{N}il}_\Omega$ is not Koszul by~\cite{Dot1}. 
\end{remark}

\noindent
{\bf Acknowledgments.}
This work is supported by the National Natural Science Foundation of China (12071191), the Natural Science Foundation of Gansu Province (20JR5RA249) and the Natural Science Foundation of Shandong Province (ZR2020MA002). The authors thank V. Dotsenko for insightful discussions. 

\noindent
{\bf Declaration of interests. } The authors have no conflicts of interest to disclose.

\noindent
{\bf Data availability. } Data sharing is not applicable as no new data were created or analyzed.

\vspace{-.3cm}



\section*{Appendix: Polarizations of polynomials and their foliations}

\mlabel{s:polarpoly}

Here we recall the polarization of polynomials~\mcite{Pro} in invariant theory and then introduce the refined notion of foliation from a polarization.
Though interesting on their own rights and independent of the rest of the paper, they are provided here as the ``baby model'' for the corresponding concepts for operads in Section~\mref{sec:comp} used to introduce various compatibilities of operads.

\subsection*{A.1. Polarizations of polynomials}
\mlabel{ss:polar}
Let $x=(x_1,\ldots, x_n)$ be a vector of variables and $f(x)=f(x_1,\ldots,x_n)$ a homogeneous polynomial of degree $m$. Denote $x^{(i)}:=(x_1^{(i)},\ldots,x_n^{(i)}), 1\leq i\leq m$, with distinct variables.

For $\malpha_i\in \bfk, 1\leq i\leq m$, denote
\vsa
$$\malpha_1 x^{(1)}+\cdots +\malpha_m x^{(m)}\coloneqq \big(\malpha_1x_1^{(1)}+\cdots+\malpha_mx_1^{(m)}, \ldots, \malpha_1x_n^{(1)}+\cdots+\malpha_mx_n^{(m)}\big).
\vsa
$$
Regarding $\malpha_1,\ldots,\malpha_m$ as formal variables (for example by taking $\bfk=\QQ(\malpha_1,\ldots,\malpha_m)$), we have the polynomial expansion
\begin{equation}
	f( \malpha_1 x^{(1)}+\cdots+\malpha_m x^{(m)})=\sum_{\pn_1,\ldots, \pn_m} f_{\pn_1,\ldots,\pn_m}(x^{(1)},\ldots, x^{(m)})\,\malpha_1^{\pn_1}\cdots \malpha_m^{\pn_m},
	\mlabel{eq:polarp}
\end{equation}
where the sum is over tuples $\pn:=(\pn_1,\ldots,\pn_m)\in \ZZ_{\geq 0}^k$ with $\sum_{i=1}^m \pn_i=m$, that is, the weak compositions of $m$.
Thus $f_{\pn_1,\ldots,\pn_m}(x^{(1)},\ldots,x^{(m)})$ is the contribution to $f( x^{(1)}+\cdots+x^{(m)})$ from monomials in $x^{(j)}_i$ whose total degree in $x^{(i)}_1,\ldots,x^{(i)}_n$ is $\pn_i, 1\leq i\leq m$.
The polynomial $f_{\pn_1,\ldots,\pn_m}(x^{(1)},\ldots,x^{(m)})$ is called the {\bf quasipolarization} of $f$ of type $\pn$.
The unique multilinear polynomial
$$F(x^{(1)}, \ldots, x^{(m)}):=f_{1,\ldots,1}(x^{(1)},\ldots, x^{(m)})$$
is called the {\bf full polarization} of $f$.

Eq.~\meqref{eq:polarp} shows that a quasipolarization can also be given by the {\bf polarization operator}:
$$f_{\pn_1,\ldots,\pn_m}(x^{(1)},\ldots,x^{(m)})
=\frac{\partial^{c_1}}{\partial\malpha_1^{c_1}}\cdots \frac{\partial^{\pn_m}}{\partial\malpha_m^{\pn_m}}f(\malpha_1 x^{(1)}+\cdots+\malpha_m x^{(m)})\big|_{\malpha_i=0},\,\text{ where }\, \malpha_i\in \bfk, 1\leq i\leq m.$$
Here $\frac{\partial}{\partial \malpha_i}$ is the formal partial derivation treating $\malpha_i$ as variables.

Notice that the set $\{f_{\pn_1,\ldots,\pn_m}(x^{(1)},\ldots, x^{(m)})\,|\,\sum_{i=1}^m \pn_i=m\}$ of quasipolarizations is {\bf symmetrically closed} in the sense that it is closed under permutations of the variables. In fact,
$$f_{\pn_1,\ldots,\pn_m}(x^{(\sigma(1))},\ldots, x^{(\sigma(m))})=f_{\pn_{\sigma^{-1}(1)},\ldots,\pn_{\sigma^{-1}(m)}}(x^{(1)},\ldots, x^{(m)}), \quad  \sigma\in S_m.$$
Moreover, we have the {\bf restitution}
\begin{equation}\mlabel{eq:res}
	f_{\pn_1,\ldots,\pn_m}(x^{(1)},\ldots, x^{(m)})|_{x^{(i)}\mapsto x}=\tbinom{m}{(\pn_i)}f(x)
\end{equation}
for the multinomial coefficients $\tbinom{m}{(\pn_i)}:=\tbinom{m}{\pn_1,\ldots, \pn_m }$.
In particular, $F(x, \ldots, x)=m!f(x)$.

\begin{example}\mlabel{exam:alpl}
	Consider the quadratic form $f(x_1,x_2)=2x_1^2+3x_1x_2+4x_2^2$. For $\malpha_1, \malpha_2\in\bfk$,  we have
	\begin{eqnarray*}
		&&f\Big(\malpha_1({x^{(1)}_1},{x^{(1)}_2})+\malpha_2({x^{(2)}_1},{x^{(2)}_2})\Big)
		=f\Big(\malpha_1 {x^{(1)}_1}+ \malpha_2 {x^{(2)}_1}, \malpha_1 {x^{(1)}_2}+ \malpha_2 {x^{(2)}_2}\Big)\\
		&=&\malpha_1^2\Big(2{x^{(1)}_1}^2+3{x^{(1)}_1}{x^{(1)}_2}+4{x^{(1)}_2}^2\Big)
		+\malpha_2^2\Big(2{x^{(2)}_1}^2+3{x^{(2)}_1}{x^{(2)}_2}+4{x^{(2)}_2}^2\Big)\\
		&&+\malpha_1\malpha_2\Big(2 \big({x^{(1)}_1}{x^{(2)}_1}+ {x^{(2)}_1}{x^{(1)}_1}\big)
		+3\big({x^{(1)}_1}{x^{(2)}_2}+{x^{(2)}_1}{x^{(1)}_2}\big)
		+4 \big({x^{(1)}_2}{x^{(2)}_2}+ {x^{(2)}_2}{x^{(1)}_2}\big)\Big).
	\end{eqnarray*}
	So we obtain
	\vsc
	\begin{eqnarray*}
		f_{2,0}&=& 2{x^{(1)}_1}^2+3{x^{(1)}_1}{x^{(1)}_2}+4{x^{(1)}_2}^2,\\
		f_{0,2}&=&2{x^{(2)}_1}^2+3{x^{(2)}_1}{x^{(2)}_2}+4{x^{(2)}_2}^2,\\
		F:=f_{1,1}&=&2\big({x^{(1)}_1}{x^{(2)}_1}+ {x^{(2)}_1}{x^{(1)}_1}\big)+3\big({x^{(1)}_1}{x^{(2)}_2}+{x^{(2)}_1}{x^{(1)}_2}\big)+4\big({x^{(1)}_2}{x^{(2)}_2}+ {x^{(2)}_2}{x^{(1)}_2}\big),
	\end{eqnarray*}
	with the last one being the full polarization.
	Observe that $F({x^{(1)}_1},{x^{(1)}_2}, {x^{(2)}_1},{x^{(2)}_2})$ is multilinear, symmetric in $({x^{(1)}_1},{x^{(1)}_2})$ and in $({x^{(2)}_1},{x^{(2)}_2})$ respectively. Moreover,
	$F(x_1,x_2,x_1,x_2)=2f(x_1,x_2).$
\end{example}

\subsection*{A.2. Foliations and \ufos of polarizations}
\mlabel{ss:folia}
A foliation is a refinement of the quasipolarization.
For a weak composition $\pn:=(\pn_1,\ldots,\pn_m)\in \ZZ_{\geq 0}^k$ of $m$,
recall that quasipolarization $f_{\pn}$ has the restitution in Eq.~\meqref{eq:res}:
$f_{\pn}(x^{(1)},\ldots, x^{(m)})|_{x^{(i)}\mapsto x}=\tbinom{m}{(\pn_i)} f(x).$
It is natural to find a partition of $f_{\pn}$ such that the restriction of each part is exactly $f(x)$.
Thus we define a {\bf foliation} of $f_{\pn}$ to be a partition $\{f_{\pc,k}\,|\,1\leq k\leq \tbinom{m}{(\pn_i)}\}$ of $f_{\pn}$ $\big($in the sense that $f_{\pn}=f_{\pn,1}+\cdots+f_{\pn,\tbinom{m}{(\pn_i)}}$$\big)$ such that
$$f_{\pn,k}\big|_{x^{(i)}\mapsto x}=f, \quad 1\leq k \leq \tbinom{m}{(\pn_i)}.$$

We next give a classification of all the foliations of a given quasipolarization, first presenting an example as a demonstration of the general procedure.

\begin{example}\mlabel{ex:fopo}
	Continuing Example~\mref{exam:alpl},
	there the polarization
	of $ f(x_1,x_2)=2x_1^2+3x_1x_2+4x_2^2$ is
	\vsb
	\begin{eqnarray} \notag
		f_{1,1}&=&
		2\big({x^{(1)}_1}{x^{(2)}_1}+ {x^{(2)}_1}{x^{(1)}_1}\big)+3\big({x^{(1)}_1}{x^{(2)}_2}+{x^{(2)}_1}{x^{(1)}_2}\big) +4\big({x^{(1)}_2}{x^{(2)}_2}+ {x^{(2)}_2}{x^{(1)}_2}\big)\\
		&=& \mlabel{eq:foliex}
		\begin{split}
			2{x^{(1)}_1}{x^{(2)}_1}
			+3{x^{(1)}_1}{x^{(2)}_2} +4{x^{(1)}_2}{x^{(2)}_2}\\
			+	2{x^{(2)}_1}{x^{(1)}_1}
			+3{x^{(2)}_1}{x^{(1)}_2} +4 {x^{(2)}_2}{x^{(1)}_2}.
		\end{split}
	\end{eqnarray}
	Since the restitution of $f_{1,1}$ is
	$f_{1,1}\Big|_{x^{(i)}\mapsto x}=\tbinom{2}{1}f=2f$, a foliation of $f_{1,1}$ is a partition $f_{1,1}=f_{1,1,1}+f_{1,1,2}$ such that $f_{1,1,k}\big|_{x^{(i)}\mapsto x}=f, k=1,2$. As displayed in Eq.~\meqref{eq:foliex}, the obvious one is
	$$
	\{f_{1,1,1}, f_{1,1,2}\}:=\{ 2x_1^{(1)}x_1^{(2)}+3x_1^{(1)}x_2^{(2)}+4x_2^{(1)}x_2^{(2)}, \quad
	2x_1^{(2)}x_1^{(1)}+3x_1^{(2)}x_2^{(1)}+4x_2^{(2)}x_2^{(1)}\}.
	$$
	Furthermore, by exchanging the entries in the second (resp. the third, resp. both) columns in Eq.~\meqref{eq:foliex}, before summing up the two rows, we obtain three more foliations:
	$$\{ 2x_1^{(1)}x_1^{(2)}+3x_1^{(2)}x_2^{(1)}+4x_2^{(1)}x_2^{(2)}, \quad
	2x_1^{(2)}x_1^{(1)}+3x_1^{(1)}x_2^{(2)}+4x_2^{(2)}x_2^{(1)}\},
	$$
	$$\{ 2x_1^{(1)}x_1^{(2)}+3x_1^{(1)}x_2^{(2)}+4x_2^{(2)}x_2^{(1)},
	\quad
	2x_1^{(2)}x_1^{(1)}+3x_1^{(2)}x_2^{(1)}+4x_2^{(1)}x_2^{(2)}\},
	$$
	$$\{ 2x_1^{(1)}x_1^{(2)}+3x_1^{(2)}x_2^{(1)}+4x_2^{(2)}x_2^{(1)}, \quad
	2x_1^{(2)}x_1^{(1)}+3x_1^{(1)}x_2^{(2)}+4x_2^{(1)}x_2^{(2)}\}.
	$$
	These are all the foliations of $f_{1,1}$.
\end{example}

In general, write a homogeneous polynomial $f$ of degree $m$ as
\begin{equation*}
	f(x)=f(x_1,\ldots,x_n)=\sum_{0\leq i \leq s} \alpha_i q_i(x_1,\ldots,x_n),
	\mlabel{eq:homof}
\end{equation*}
where $\alpha_i\in\bfk$ and $q_i$ is a monomial of degree $m$.
Let $c$ be a weak partition of $m$. For notational clarity, we denote
$$\sbin:=\mbin=\tbinom{m}{(\pn_i)}=\tbinom{m}{\pn_1,\ldots, \pn_m }.$$

Note that the quasipolarization $q_{i,c}$ of the monomial $q_i$ has the unique foliation $\{q_{i,c,k}\,|\, 1\leq k\leq \sbin\}$.
Then
$$f_{\pn}(x^{(1)},\ldots, x^{(m)})
=\sum_{0\leq i \leq s} \alpha_i q_{i,c}(x_1,\ldots,x_n)
=\sum_{0\leq i \leq s} \alpha_i \sum_{1\leq k\leq {\sbin}}q_{i,c,k}(x_1,\ldots,x_n)
=\sum_{1\leq k\leq {\sbin}}\sum_{0\leq i \leq s} \alpha_i q_{i,c,k}(x_1,\ldots,x_n),$$
giving the natural partition of $f_{\pn}$ into the system
\vsc
\begin{equation}
	\left\{\sum_{0\leq i \leq s} \alpha_i q_{i,c,k}\,\Big|\, {1\leq k\leq {\sbin}}\right\}
	=\left (\begin{array}{cccc}
		q_{0,c,1} & q_{1,c,1} & \cdots & q_{s,c,1} \\
		q_{0,c,2} & q_{1,c,2} & \cdots & q_{s,c,2} \\
		\vdots & \vdots & \ddots & \vdots \\
		q_{0,c,{\sbin}} & q_{1,c,{\sbin}} & \cdots & q_{s,c,{\sbin}}
	\end{array} \right )
	\left(
	\begin{array}{c}
		\alpha_0\\
		\alpha_1\\
		\vdots\\
		\alpha_s
	\end{array}
	\right).
	\label{eq:stfolio}
\end{equation}
We call this system the {\bf standard foliation} of $f_c$, with respect to the natural ordering of the partition $\{q_{i,c,k}\,|\,1\leq k\leq {\sbin}\}$ for $0\leq i\leq s$.

Starting with the standard foliation one can construct other foliations of $f_c$. Indeed, let $S_{\sbin}$ denote the symmetric group on $\sbin$ letters. Then replacing the $\sbin\times (s+1)$-matrix in Eq.~\eqref{eq:stfolio} by the one with the $i$-column permuted by $\sigma_i\in S_{\sbin}$ for $1 \leq i \leq s$, we obtain another foliation of $f_c$
\begin{equation}
	\left\{\alpha_0 q_{0,c,k}+\sum_{1\leq i \leq s} \alpha_i q_{i,c,\sigma_i(k)}\,\Big|\, {1\leq k\leq {\sbin}}\right\}
	=\left (\begin{array}{cccc}
		q_{0,c,1} & q_{1,c,\sigma_1(1)} & \cdots & q_{s,c,\sigma_s(1)} \\
		q_{0,c,2} & q_{1,c,\sigma_1(2)} & \cdots & q_{s,c,\sigma_s(2)} \\
		\vdots & \vdots & \ddots & \vdots \\
		q_{0,c,{\sbin}} & q_{1,c,\sigma_1({\sbin})} & \cdots & q_{s,c,\sigma_s({\sbin})}
	\end{array} \right )
	\left(
	\begin{array}{c}
		\alpha_0\\
		\alpha_1\\
		\vdots\\
		\alpha_s
	\end{array}
	\right),
	\label{eq:permfolia}
\end{equation}
called the {\bf foliation of $f_c$ with respect to $(\sigma_1,\ldots,\sigma_s)$}.
\smallskip

Finally, as a prototype of the total compatibility of operads in \S\,\mref{ss:totc},
we add additional linear forms to the system of a foliation in Eq.~\meqref{eq:permfolia}, so that all the foliations become the same, independent of the permutations $\sigma_i\in S_\sbin$.
This is achieved by adding $q_{i,c,\sigma(i)}-q_{i,c,\tau(i)}$ for all $\sigma, \tau\in S_\sbin$ and $0\leq i\leq s$. The resulting system is called the {\bf \ufo} of $f_c$.

\begin{example}
	In Example~\mref{ex:fopo}, there are four foliations of $f_{1,1}$.
	The \ufo of $f_{1,1}$ is given by adding
	$$x_1^{(1)}x_1^{(2)}-x_1^{(2)}x_1^{(1)},\quad  x_1^{(1)}x_2^{(2)}-x_1^{(2)}x_2^{(1)},\quad  x_2^{(1)}x_2^{(2)}-x_2^{(2)}x_2^{(1)}$$
	to either of the four foliations. The resulting set becomes independent of the choice of the foliations.
	\vsb
\end{example}

As a model of the dualities for operadic compatibilities established in Section~\ref{sec:km}, the three classes of polynomials (quasipolarizations, foliations and \ufos) for noncommutative homogeneous quadratic polynomials are related by the Koszul duality for homogeneous quadratic algebras, as introduced by Stewart Priddy in~\mcite{Pri70}. We just give a remark and leave the general treatment elsewhere.

\begin{remark}
	Let
	$f(x_1,\ldots,x_n)=\sum_{0\leq i \leq s} \alpha_i q_i(x_1,\ldots,x_n)$
	be a noncommutative homogeneous quadratic polynomial.
	Let $\Omega=\{1,2\}$ and $c=(1,1)$. Then the \ufo of a quasipolarization $f_c$ in $\bfk\langle x_1^{(1)},x_2^{(2)},\ldots,x_n^{(1)},x_n^{(2)}\rangle$ is the set
	$$J_c:=\Big\{\alpha_0 q_{0,c,k}+\sum_{1\leq i \leq s} \alpha_i q_{i,c,\sigma_i(k)}\,\Big|\, {1\leq k\leq 2}\Big\}\cup \{q_{i,c,1}-q_{i,c,2}\,|\,0\leq i\leq s\}.$$
	Then the two quadratic algebras
	$$A:=\bfk \langle x_1^{(1)},x_1^{(2)} \cdots,x_n^{(1)}, x_n^{(2)}\rangle/\langle f_{c}\rangle,
	\quad
	B:=\bfk\langle x_1^{(1)},x_1^{(2)} \cdots,x_n^{(1)}, x_n^{(2)}\rangle/
	\langle J_c \rangle$$
	are in Koszul dual of each other.
\end{remark}

\end{document}